\documentclass[10pt, reqno, a4paper]{amsart}

\usepackage[foot]{amsaddr}

\usepackage[T1]{fontenc}
\usepackage[english]{babel}
\usepackage[dvipsnames,table,xcdraw]{xcolor}
\definecolor{myred}{RGB}{228,26,28}
\definecolor{myblue}{RGB}{55,124,184}
\definecolor{mygreen}{RGB}{77,175,74}
\usepackage[colorlinks=true,urlcolor=myblue,linkcolor=myblue,runcolor=myred,citecolor=myred,linktoc=all]{hyperref}

\usepackage{geometry}
\usepackage{graphicx}
\usepackage{array}
\usepackage{wrapfig}
\usepackage{enumitem}
\setlist[itemize]{label={$\vcenter{\hbox{\tiny$\bullet$}}$}}
\usepackage{appendix}
\usepackage{caption}
\usepackage{subcaption}
\usepackage{csquotes}
\usepackage{booktabs}

\usepackage{amsmath,mathtools,amssymb,amsthm,amsfonts}
\usepackage{bm}
\usepackage{esint}
\usepackage{chemformula}

\newcommand{\ds}{\displaystyle}

\newcommand{\set}[1]{\left\{ #1 \right\}}
\newcommand{\C}{\mathbb{C}}
\newcommand{\N}{\mathbb{N}}
\newcommand{\R}{\mathbb{R}}
\renewcommand{\S}{\mathbb{S}}
\newcommand{\J}{\mathbb{J}}
\renewcommand{\P}{\mathbb{P}}
\newcommand{\Z}{\mathbb{Z}}
\renewcommand{\S}{\mathbb{S}}
\newcommand{\Forall}{\forall\ }

\newcommand{\ie}{\emph{i.e.}\ }
\newcommand{\eg}{\emph{e.g.}\ }

\newcommand{\Span}{{\rm Span} \, }
\renewcommand{\Re}{{\mathfrak R} \, }
\renewcommand{\Im}{{\mathfrak I} \, }

\renewcommand{\d}{{\rm d}}
\newcommand{\prt}[1]{\left( #1 \right)}

\renewcommand{\i}{\imath}
\renewcommand{\L}{{\rm L}}

\newcommand{\W}{{\rm W}}

\newcommand{\norm}[1]{\left\| #1 \right\|}

\newcommand{\wb}[1]{\overline{#1}}
\newcommand{\abs}[1]{\left| #1 \right|}

\newcommand{\Xc}{\mathcal{X}}

\newcommand\inner[1]{\langle #1 \rangle}
\DeclareMathOperator{\dist}{dist}
\renewcommand{\epsilon}{\varepsilon}

\theoremstyle{plain}
\newtheorem{proposition}{Proposition}

\theoremstyle{definition}
\newtheorem{remark}{Remark}

\theoremstyle{definition}
\newtheorem{definition}{Definition}

\theoremstyle{plain}
\newtheorem{theorem}{Theorem}

\theoremstyle{plain}
\newtheorem{lemma}{Lemma}

\addto\extrasenglish{%
}


\title{Numerical simulation of the Gross--Pitaevskii equation via vortex
  tracking}

\author{Thiago Carvalho Corso$^1$}

\author{Gaspard Kemlin$^2$}

\author{Christof Melcher$^3$}

\author{Benjamin Stamm$^1$}

\address[1]{IANS-NMH, University of Stuttgart, 70569 Stuttgart, Germany}
\address[2]{LAMFA, Universit\'e de Picardie Jules Verne, 80039 Amiens, France}
\address[3]{Applied Analysis and JARA FIT, RWTH Aachen University, 52056
  Aachen, Germany}

\email{thiago.carvalho-corso@mathematik.uni-stuttgart.de}
\email{gaspard.kemlin@u-picardie.fr}
\email{melcher@math1.rwth-aachen.de}
\email{best@ians.uni-stuttgart.de}

\begin{document}

\begin{abstract}
  This paper deals with the numerical simulation of the Gross--Pitaevskii (GP)
  equation, for which a well-known feature is the appearance of quantized vortices with
  core size of the order of a small parameter $\varepsilon$. Without a magnetic field and
  with suitable initial conditions, these vortices interact, in the singular limit
  $\varepsilon\to0$, through
  an explicit Hamiltonian dynamics. Using this analytical framework, we develop and
  analyze a numerical strategy based on the reduced-order Hamiltonian system to
  efficiently simulate the infinite-dimensional GP equation for small, but finite,
  $\varepsilon$. This method allows us to avoid numerical stability issues in solving the GP
  equation, where small values of $\varepsilon$ typically require very fine meshes and time
  steps. We also provide a mathematical justification of our method in terms of rigorous error estimates of the error in
  the supercurrent, together with numerical illustrations.
\end{abstract}

\maketitle

\section{Introduction}

\subsection{Model and motivations}

Gross--Pitaevskii equations (GP) are a class of nonlinear Schr\"odinger equations
for quantum mechanical many particle systems in terms of a complex-valued field
$u$, featured \eg in Bose--Einstein
condensation, superconductivity and other condensed matter systems.
A prototypical form of GP arises from the usual Ginzburg–Landau Hamiltonian of a super-fluid
\begin{equation} \label{eq:GP-energy}
  E_\varepsilon(u) = \int_\Omega e_\varepsilon(u)
  \quad\text{with}\quad
  e_\varepsilon(u) = \frac12 |\nabla u|^2 + \frac{1}{4\varepsilon^2} (1-|u|^2)^2 \,
\end{equation}
on a simply connected smoothly bounded domain $\Omega \subset \R^2$. The
mathematical investigation of the variational model and its minimizers in the
singular limit $\varepsilon \to 0$ has been initiated in the seminal work
\cite{bethuelGinzburgLandauVortices1994}. The corresponding Schr\"odinger flow
$\i\partial_t \psi = -\nabla_{L^2} E_\varepsilon(\psi)$ (with $\i$ the imaginary
unit),
on a smooth bounded domain $\Omega\subset\R^2$ with homogeneous Neumann boundary
conditions, has the form
\begin{equation}\label{eq:GPE}
  \begin{cases}
    \ds\i \partial_t \psi_\varepsilon(x,t) = \Delta \psi_\varepsilon(x,t) +
    \frac1{\varepsilon^2}\prt{1-\abs{\psi_\varepsilon(x,t)}^2}\psi_\varepsilon(x,t),
    \quad x\in\Omega,\\
    \psi_\varepsilon(x,0) = \psi_\varepsilon^0(x) \quad x\in\Omega,\\
    \partial_\nu\psi_\varepsilon(x,t) = 0 \quad x\in\partial\Omega,\ t\geq0,
  \end{cases}
\end{equation}
where $\psi_\varepsilon^0\in H^1(\Omega,\C)$ and $\nu$ is the outward normal to
the boundary $\partial\Omega$.
This equation is globally well-posed in time (see for instance
\cite{bourgainFourierTransformRestriction1993,cazenaveSemilinearSchrodingerEquations2003,galloSchrodingerGroupZhidkov2004,gerardCauchyProblemGross2006,ogawaTrudingerTypeInequalities1991}
and references therein)
for any $\varepsilon>0$ and,
because of the homogeneous Neumann boundary condition, the energy
$E_\varepsilon(\psi_\varepsilon)$ is preserved: $E_\varepsilon(\psi_\varepsilon(t)) =
E_\varepsilon(\psi_\varepsilon^0)$ for all $t\geq0$. In this context $|\psi|^2$
is interpreted as the density, and other physical quantities such as the supercurrents
$j(\psi) = \frac 1 {2\i} ( \overline{\psi}\nabla \psi
- \psi \nabla \overline{\psi})$ and the vorticity $J\psi = \frac 1 2
\nabla\times j(\psi)$ can be defined.

A striking feature of super-fluids is the occurrence of quantized
vortices that evolve to leading order according to the ODE system known for
point vortices in an ideal incompressible fluid. This behavior has been
predicted by formal analysis in physics
\cite{fetterStabilityLatticeSuperfluid1966}
and verified by matched asymptotics in applied mathematics literature
\cite{eDynamicsVorticesGinzburgLandau1994a,neuVorticesComplexScalar1990}. The mathematical link is the
Madelung transform which is known to relate Schr\"odinger equations in quantum
mechanics and Euler equations in fluid mechanics. The rigorous description of
vortex dynamics arising from GP in the variational framework of
\cite{bethuelGinzburgLandauVortices1994} has been established in
\cite{collianderVortexDynamicsGinzburgLandauSchrodinger1998,linIncompressibleFluidLimit1999},
see also \cite{ovchinnikovGinzburgLandauEquationIII1998}.
It considers a family of solutions $\psi_\varepsilon$ to \eqref{eq:GPE}
with initial data $\psi^{0}_\varepsilon$ for which the vorticity converges to a sum
of point masses at distinct points (vortex locations) $a^0_{j} \in \Omega$
with a single quantum of vorticity $d_j =\pm 1$, \ie
\begin{equation} \label{eq:Jacobi_convergence}
  J\psi^0_\varepsilon \to  \pi \sum_{j=1}^N d_j \delta_{a^0_j}
\end{equation}
weakly in the sense of measures. Suitable norms to control this convergence
include the $\dot{W}^{-1,1}$ norm dual to the Lipschitz norm, closely related to
minimal connection and the $1$-Wasserstein distance.
A further key assumption is the so-called \enquote{well-preparedness} of initial conditions,
\ie initial energies are asymptotically minimal given boundary conditions,
vortex locations $a_j^0$ and vorticities $d_j$ as above. The result is that
$J\psi_\varepsilon(t) \to  \pi \sum_{j=1}^N d_j \delta_{a_j(t)}$ for $t > 0$
where the singular points $\bm{a}(t)=(a_1(t),...,a_N(t))$ of degree
$d=(d_1,...,d_N)$ solve the point vortex system
\begin{equation}
  \label{eq:Effective-equation}
  \begin{cases}
    \dot{a}_j(t) = -\frac1\pi d_j\J\nabla_{a_j} W({\bm a}(t), d),\\
    {a}_j(0) = {a}^0_j,
  \end{cases}
  \quad \text{with} \quad \J = \begin{bmatrix} 0 & 1 \\ -1 & 0 \end{bmatrix},
\end{equation}
with the standard symplectic matrix $\mathbb{J} \in \R^{2 \times 2}$ and the two-dimensional Coulomb-type Hamiltonian
\begin{equation}\label{eq:renormalized}
  W({\bm a},{d}) = - {\pi} \sum_{1\leq i\neq j \leq N} d_i d_j \ln|a_i-a_j| + \mbox{boundary terms}.
\end{equation}
For $\Omega=\R^2$ in the absence of boundary conditions, this Hamiltonian system is classical and well-understood. The general case of a bounded domain $\Omega$
is more subtle and recently became a topic of intense investigations
particularly regarding the existence of periodic solutions, see
\eg \cite{bartschPeriodicSolutionsSingular2016}.

From a variational point of view, $W$ is the renormalized energy introduced in
\cite{bethuelGinzburgLandauVortices1994}, corresponding to the $\Gamma$~-~limit of
the Ginzburg--Landau energy \eqref{eq:GP-energy} minus the logarithmically
diverging self-energy of $N$ vortices of degree $d_j = \pm 1$:
\begin{equation} \label{eq:expansion}
  E_\varepsilon(u_\varepsilon) - N \left( \pi  \ln(1/\varepsilon) + \gamma
  \right) \approx  W({\bm a},{d}) \quad \text{as}  \quad \varepsilon \to 0,
\end{equation}
for an explicit constant $\gamma>0$, with respect to the convergence \eqref{eq:Jacobi_convergence}.
Results also characterize the asymptotic limits of the supercurrents
$j(\psi_\varepsilon(t))$ in $L^p(\Omega)$ and of the wave functions
$\psi_\varepsilon(t)$ in $W^{1,p}(\Omega)$ for $p < 2$, respectively, towards
$u^*(t)$ with values in the unit sphere~$\mathbb{S}^1$ having point singularities at the site of
vortex locations $\bm{a}(t)$. More precisely, $u^*=u^*(\cdot ; \bm{a}, d) \in
\mathbb{S}^1$ is the uniquely determined harmonic map with singularities at
locations~$\bm{a}$ with local degrees~$d$, and with boundary conditions.
While GP is globally (in time) well-posed in
appropriate function spaces,
the asymptotic results are only valid up to the first time of collision of the
point vortex system which may happen in finite time.
A substantial improvement of the above asymptotic results has been obtained in
\cite{jerrardRefinedJacobianEstimates2008}, which establish a quantitative
description based on largely improved Jacobian and supercurrents estimates.
This proves that the vortex motion law holds approximately for small but finite
$\varepsilon$, rather than in the limit $\varepsilon \to 0$, and is accompanied
by an estimate of the rate of convergence and the time interval for which the
results remain valid.

Finally, let us stress out that the two-dimensional aspect of the problem is a
key assumption, both from a mathematical point of view (as it allows to identify
$\R^2$ and $\C$) and a physical point of view (it is known that for such models,
the relevant physics is mainly driven by two-dimensional approximations, see for
instance \cite{simulaEmergenceOrderTurbulence2014}
and references inside). In three dimensions, the singularities
that appear in the Gross--Pitaevskii equation have more complex shapes, such as
vortex filaments
\cite{jerrardDynamicsNearlyParallel2021,kaltIdentificationVorticesQuantum2023,
  kobayashiQuantumTurbulenceSimulations2021a,villoisVortexFilamentTracking2016}
which are one-dimensional structures where the density vanishes. The two-dimensional
case can alternatively be seen as the intersection between a plane and, for
instance, a cylinder: vortices are then the intersection of the filaments
with this plane. Such singularities go beyond the scope of this paper.

On the other hand, the numerical approximation of solutions of equation
\eqref{eq:GPE} is a well-established research field.
There is a variety of numerical methods to solve nonlinear Schr\"odinger
equations such as equation~\eqref{eq:GPE}.
A comprehensive overview of the main numerical methods to simulate such
equations can be found in \cite{antoineComputationalMethodsDynamics2013},
where the most common time-space-discretization strategies are compared and
different properties listed.
The large majority of numerical methods employ time-splitting strategies
together with a uniform spatial approximation, either by the Fourier
(pseudo-)spectral, finite element or finite difference method, see \eg
\cite{aftalionRotationBoseEinsteinCondensate2010,
  baoDynamicsRotatingBose2006,
  baoExplicitUnconditionallyStable2003,
  baoNumericalStudyQuantized2013,
  baoNumericalStudyQuantized2014,
  brachetGrossPitaevskiiDescription2012,
  kongSymplecticStructurepreservingIntegrators2011,
  zhangNumericalSimulationVortex2007}
and references inside, to mention a few of the many contributions.
From a perspective of approximation theory, such methods are of course fine as
long as the solutions are smooth during the simulation.
In the regime of small values of $\varepsilon$, which is the regime that we
target in this paper, this becomes delicate in terms of accuracy.
For irregular solutions containing singularities, it is common practice in
numerical analysis and scientific computing to perform mesh refinements. In the
case of finite element approximations, this leads to so-called adaptive finite
element methods which require an element-wise quantification of the error over
the mesh, which is typically achieved by considering the residual of the PDE and results in
residual-based \emph{a posteriori} estimators. Interestingly, the numerical study
presented in \cite{thalhammerNumericalStudyAdaptive2012} demonstrates
superiority of the Fourier (pseudo-) spectral method with adaptive higher-order
time-splitting schemes over an adaptive finite element method with local
time-stepping. However, this paper is limited to the case of the semi-classical
limit (\ie when the left-hand side of~\eqref{eq:GPE} is multiplied by
$1/\varepsilon$). In the case of discrete minimizers of the Ginzburg--Landau energy, recent works
\cite{dorichErrorBoundsDiscrete2023} highlighted, by rigorous \emph{a priori}
bounds in suitable Sobolev norms, the necessity of very fine element mesh when
the model parameter $\varepsilon$ is small. Additionally, in
\cite{baoNumericalStudyQuantized2013,baoNumericalStudyQuantized2014}, the
authors propose a numerical study of the vortex interaction for different types of
nonlinear Schr\"odinger equations.
They compare the vortex trajectory for
finite and small $\varepsilon$ to the one given by the effective dynamics
\eqref{eq:Effective-equation}, providing a numerical illustration of the
singular limit $\varepsilon\to0$.

These considerations motivate the search for novel numerical methods, taking
advantages of the known (low-dimensional) ODE in the singular limit
$\varepsilon\to0$ in order to approximate the solution to the GP equation
\eqref{eq:GPE} for small, but finite, $\varepsilon$. The idea we develop in this
paper is based on the construction of \enquote{well-prepared} initial conditions presented
in \cite{jerrardRefinedJacobianEstimates2008}, where the authors suggest to construct such initial
conditions by smoothing out the canonical harmonic map with provided vortex
locations. From an abstract perspective, this gives rise to a
nonlinear projection $\mathcal P_\varepsilon: (\bm{a},d)
\in\Omega^N\times\set{\pm 1}^N \mapsto u^*_{\varepsilon}(\bm a,d)$ on
$N$-dimensional submanifolds $M^\varepsilon\subset H^1(\Omega; \mathbb{C})$
defined by $M^\varepsilon = \{  u_\varepsilon^* = u_\varepsilon^*(x ; \bm{a},
d):  \bm{a} \in \Omega^{N} \} \cong \Omega^N$
of elements $u_\varepsilon^*$ obtained from smoothing out canonical harmonic
maps $u^*= u^*(x ; \bm{a},  d)$
corresponding to vortices at $\bm{a} \in \Omega^{* N}$ and with $d=\{\pm 1\}^N$
fixed in an energetically optimal fashion such that
\begin{equation}
  \label{eq:a_from_u}
  \left\| Ju_\varepsilon^* -  \pi \sum_{j=1}^N d_j \delta_{a_j} \right\|_{\dot{W}^{-1,1}} \lesssim \varepsilon^\alpha
  \quad\text{and}\quad
  \| j(u_\varepsilon^*) -  j(u^*(\bm a,d)) \|_{L^{\frac43}} \lesssim \varepsilon^\gamma,
\end{equation}
for some $\alpha, \gamma \in (0,1)$. The result of
\cite{jerrardRefinedJacobianEstimates2008} is that such a projection can be
defined in a (tubular) neighborhood of $M^\varepsilon$.
Moreover, if vortices are moved according to
\eqref{eq:Effective-equation}, the estimates remain valid for the GP solution,
giving rise to space-time projection onto a subspace in terms of vortex
trajectories. In other terms, well-preparedness is conserved in the time
interval for which these estimates are valid: in
this paper we suggest to use this property to recover an approximation of the
solution $\psi_\varepsilon$ to the GP equation at time $t>0$. The main idea is, starting
from a well-prepared initial condition $\psi^0_\varepsilon$, to evolve the vortices
according to the Hamiltonian ODE \eqref{eq:Effective-equation} up to some time
$t$. Then, by the same projection used to set-up well-prepared initial
conditions, build back an approximation $\psi_\varepsilon^*(t)$ of
$\psi_\varepsilon(t)$ by
smoothing out the canonical harmonic map with singularities given by the vortex
locations at time $t$. Numerically, the time consuming step is thus the
computation of the solutions to the Hamiltonian dynamics
\eqref{eq:Effective-equation}, which can done in a few seconds on a personal
laptop, a significant improvement with respect to the simulation of the full PDE
\eqref{eq:GPE} for very small $\varepsilon$, which can take several days on a
small sized cluster for
small values of $\varepsilon$. However, the analytical framework
of the singular limit being valid as long as the vortices stay away from each
other, our method cannot be used to reproduce nonlinear effects such as
radiation and sound waves triggered by vortex collisions.
Finally, the main contributions of this paper can be summarized as follows:
\begin{itemize}
  \item We propose, and implement, a new method to approximate numerically the
    solution to \eqref{eq:GPE} in the regime of small, but finite,
    $\varepsilon$, with numerical evidence of its accuracy. The originality of
    this method is that the limiting parameter for the numerical discretization
    is no longer the size of vortices $\varepsilon$ but the inter-vortices
    distance, as shown in Theorem~\ref{thm:errorestimate}.
  \item We also propose an
    efficient and cheap way to solve the Hamiltonian ODE
    \eqref{eq:Effective-equation} by using harmonic polynomials to evaluate the
    boundary terms in the renormalized energy $W$ defined in
    \eqref{eq:renormalized}, which involve the resolution of a Laplace equation.
  \item We prove that our method is asymptotically exact when $\varepsilon\to0$
    together with the discretization parameters by deriving an \emph{a priori}
    bound on the supercurrents, using the results from \cite{jerrardRefinedJacobianEstimates2008} and classical elliptic estimates.
\end{itemize}

\subsection{Structure of the paper}

This paper is organized as follows. First, we conclude this introductory section with some
general notations. Then, in Section~\ref{sec:lite}, we present a short review of the
analytical results on the limit $\varepsilon\to0$, as well as the notion of
well-prepared initial conditions. In Section~\ref{sec:HD_sim}, we detail the
numerical resolution of the vortex trajectories as solutions of the Hamiltonian
ODE \eqref{eq:Effective-equation}, together with some numerical experiments.
Finally, in Section~\ref{sec:method}, we present our new method, based on
vortex tracking.
We also provide an \emph{a priori} bound of the error on the supercurrents in
terms of $\varepsilon$ and the discretization parameters, with a numerical
illustration of this bound.

\subsection{Notations}

In this paper and the numerical simulations we present, we consider a
domain $\Omega$ which is a bounded, simply connected, open subset of $\R^2$ with
smooth boundary. We denote by $\nu$ the (unit) outward normal to $\Omega$ and by $\tau$
the tangential vector such that the basis $(\nu,\tau)$ is direct.
Vortices are considered as point-like defects in the bounded domain $\Omega$:
we will use throughout the paper the notation $\bm{a}^0 = (a^0_j)_{j=1,\dots,N}
\in\Omega^{* N}$, ${\bm a}(t) = (a_j(t))_{j=1,\dots,N}\in\Omega^{* N}$
and $d = (d_j)_{j=1,\dots,N} \in
\set{\pm1}^N$ to describe the positions (at time $t=0$ and $t>0$) and the
degrees of $N$ vortices. Here, $\Omega^{* N} = \set{\bm{a} =
  (a_j)_{j=1,\dots,N}\in\Omega^N,\ a_i \neq
  a_j \text{ for } 1\leq i\neq j\leq N}$.
Moreover, time-dependent wave functions will be denoted by $\psi$ or $\psi(t):\Omega\to\C$ for
their evaluation at time $t$ while $u_\varepsilon$ and $u^*$ are used to denote
respectively minimizers of the Ginzburg--Landau energy $E_\varepsilon$ and
canonical harmonic maps.

We will often use the notation $A \lesssim B$ to
indicate the existence of a constant $C>0$ such that $A \leq C B$. If the
constant $C$ depends on some additional parameter (\eg, $\epsilon$), then we
use the notation $A \lesssim_{\epsilon} B$ to emphasize this dependence.

When no ambiguity occurs,
$x\in\R^2$ is identified with the complex number $x_1 + \i x_2\in\C$.
For $x\in\R^2$, $|x|$
is the Euclidean norm of $x$ and $\theta(x)$ is the unique element of
$(-\pi, \pi]$ representing the equivalence class of $\arg(x)$.
For $x = (x_1,x_2), y = (y_1,y_2)\in\R^2$, $x \times y =
x_1y_2 - x_2y_1$. If
$w:\R^2 \to \R^2$, we define $\nabla \times w = \partial_{x_1} w_2 -
\partial_{x_2} w_1$
and if $w:\R^2\to\R$, we define $\nabla\times w = (\partial_{x_2} w,
-\partial_{x_1} w)$.
Most of the analytical results we use to justify our numerical method deal
with physical quantities defined as the supercurrents $j(\psi) =
\frac1{2\i}(\wb\psi\nabla\psi-\psi\nabla\wb\psi)$ and the Jacobian
\begin{equation}\label{eq:jacobian_def}
  J\psi = \frac12\nabla\times j(\psi) = \det{\nabla\psi} = \begin{vmatrix}
    \partial_{x_1}\Re(\psi) & \partial_{x_2}\Re(\psi) \\\partial_{x_1}\Im(\psi) &
    \partial_{x_2}\Im(\psi) \end{vmatrix}.
\end{equation}
The appropriate space and norm to deal with the Jacobian $J\psi$ is the negative
Sobolev space $\dot\W^{-1,1}(\Omega)$, which is the dual of the space of
Lipschitz functions vanishing on $\partial\Omega$:
\begin{equation*}
  \norm{\mu}_{\dot\W^{-1,1}} = \sup\set{\int_\Omega\phi\d\mu,\
    \norm{\nabla\phi}_{L^\infty} \leq 1,\ \phi\in\W^{1,\infty}_0(\Omega)}.
\end{equation*}
This norm naturally appears in quantity of papers because of its interpretation
as the length of minimal connection~\cite{brezisHarmonicMapsDefects1986}:
if $\bm{a},\bm{b}\in\Omega^{* N}$ are such that $|a_j - b_j|\leq\rho(\bm{a})$ for all
$j$, then
\begin{equation*}
  \norm{\pi\sum_{j=1}^Nd_j(\delta_{a_j} - \delta_{b_j})}_{\dot\W^{-1,1}} = \pi\sum_{j=1}^N |d_j||a_j- b_j|,
\end{equation*}
where
\begin{equation}\label{eq:rho}
  \rho(\bm{a}) = \frac14\min\set{\min_{j\neq k}|a_j-a_k|, {\rm dist}(x_j,
    \partial\Omega)}.
\end{equation}
We will see that a wave function $\psi$ is made of \emph{almost vortices} if the
Jacobian concentrates around Dirac masses, making the interpretation of $J\psi$
as the vorticity natural. The $\dot\W^{-1,1}$ norm therefore
appears as the natural norm to measure the distance between $J\psi$ and the
point in $\Omega$ where the Dirac masses around which it concentrates are
localized.

We end this section by summarizing in Table~\ref{tab:notations} some notations used
repeatedly throughout the paper, as well as where they are defined.

\begin{table}[h!]
\begin{tabular}{@{}ccc@{}}
\toprule
\textbf{Notation}    & \textbf{Name}                                                                                                                          & \textbf{Definition}                                                                \\ \midrule
$E_\varepsilon$      & Ginzburg--Landau energy                                                                                                                & \eqref{eq:GP-energy}                                                               \\ \midrule
$\psi_\varepsilon$   & solution to GP equation                                                                                                                & \eqref{eq:GPE}                                                                     \\ \midrule
$j$                  & supercurrent                                                                                                                           & above \eqref{eq:jacobian_def}                                                      \\ \midrule
$J$                  & Jacobian                                                                                                                               & \eqref{eq:jacobian_def}                                                            \\ \midrule
$\bm a,\bm b\in\Omega^N$     & positions of $N$ vortices in $\Omega\subset\R^2$                                                                               & $\times$                                                                           \\ \midrule
$\rho(\bm a)$               & inter-vortices distance                                                                                                         & \eqref{eq:rho}                                                                     \\ \midrule
$W(\bm a,d)$         & \begin{tabular}[c]{@{}c@{}}renormalized energy\\ for singular points $\bm a$\\ with degrees $d$\end{tabular}                           & \begin{tabular}[c]{@{}c@{}}\eqref{eq:renormalized}\\ \eqref{eq:W_BBH}\end{tabular} \\ \midrule
$R(\cdot; \bm a,d)$  & \begin{tabular}[c]{@{}c@{}}boundary terms in the\\ renormalized energy\end{tabular}                                                     & \eqref{eq:R}                                                                       \\ \midrule
$u^*$                & canonical harmonic map                                                                                                                 & \eqref{eq:canon}                                                                   \\ \midrule
$H$                  & phase factor of $u^*$                                                                                                                  & \eqref{eq:H_rebuilt}                                                               \\ \midrule
$\psi_\varepsilon^*$ & \begin{tabular}[c]{@{}c@{}}approximation of $\psi_\varepsilon$\\ obtained by smoothing out $u^*$ with\\ the method we present\end{tabular} & \eqref{eq:smoothed_approx}                                                     \\ \bottomrule
\end{tabular}
\caption{Main objects used throughout the paper.}
\label{tab:notations}
\end{table}

\section{Results on Gross--Pitaevskii vortices and their dynamics}\label{sec:lite}

\subsection{Properties of canonical harmonic maps and the renormalized energy}

For $\bm a\in\Omega^{* N}$ and $d~\in~\set{\pm 1}^N$ given, we define the
canonical harmonic map $u^*(\bm a,d) \in\W^{1,p}(\Omega,\S^1)$, $p<2$,
with Neumann boundary conditions as the solution to
\[
  \nabla \cdot j(u^*) = 0,\quad \nabla \times j(u^*) =
  2\pi\sum_{j=1}^N{d_j\delta_{a_j}},\quad \nu\cdot j(u^*) = 0 \text{ on } \partial\Omega.
\]
These conditions uniquely determine the supercurrents $j(u^*)$, which in turn
determines $u^*$ only up to a constant phase factor, see
\cite[Chapter 1]{bethuelGinzburgLandauVortices1994}. The typical form for $u^*$
is then given by
\begin{equation}\label{eq:canon}
  u^*(x) = u^*(x; \bm{a}, d) \coloneqq
  e^{\i H(x)}\prod_{j=1}^N \prt{\frac{x - a_j}{|x - a_j|}}^{d_j},
\end{equation}
where $H:\R^2\to\R^2$ is some harmonic function such that $u^*$ satisfies
homogeneous Neumann boundary conditions. Introducing the solution $R(\cdot;\bm
a,d)$ to
\begin{equation}\label{eq:R}
  \begin{cases}
    \ds\Delta R = 0 \quad\text{in }\Omega,\\
    \ds R = - \sum_{j=1}^N d_j \ln|x - a_j| \quad\text{on }\partial\Omega,
  \end{cases}
\end{equation}
it is easy to check that $j(u^*(\bm a,d)) = - \nabla \times G$ where
\begin{equation}\label{eq:j_G}
  G(x; \bm a, d) = \sum_{j=1}^N d_j\ln|x-a_j| + R(x;\bm a,d).
\end{equation}
We now recall the formulation of the renormalized energy introduced in
\cite{bethuelGinzburgLandauVortices1994} as
\begin{equation}\label{eq:W_BBH}
  W(\bm a,d) = - \pi \sum_{1\leq i\neq j\leq N} d_id_j\ln|a_i-a_j|
  - \pi \sum_{j=1}^N d_jR(a_j;\bm a, d).
\end{equation}

Finally, we introduce the following minimization
problem, which considers only one vortex of degree $+1$ inside the ball
$B_{r}(0)$ of radius $r$ centered at $0$:
\begin{equation}\label{eq:inf_ball}
  I(r, \varepsilon) = \inf\set{\int_{B_{r}(0)} e_\varepsilon(u);\ u\in H^1(B_{r}(0), \C),\ u =
    e^{\\i \theta} \text{ on } \partial B_{r}(0)}.
\end{equation}
Let
\begin{equation}\label{eq:gamma}
  \gamma = \lim_{r\to\infty}\prt{I(r,\varepsilon) - \pi\ln{\frac{r}\varepsilon}}.
\end{equation}
It is known that $\gamma$ exists, is finite and independent of $\varepsilon$,
see  \cite{bethuelGinzburgLandauVortices1994}. In addition, it is proved in
\cite[Lemma 6.8]{jerrardRefinedJacobianEstimates2007} that
\begin{equation*}
  \gamma - \prt{I(r,\varepsilon) - \pi\ln\frac{r}\varepsilon} =
  O\prt{\prt{\frac \varepsilon {r}}^2}.
\end{equation*}
For fixed $\varepsilon>0$ and given $\bm{a}\in\Omega^{* N}$, $d\in\set{\pm1}^N$, let
\begin{equation*}
  W_\varepsilon(\bm{a},d) = N\prt{\gamma + \pi\ln\frac1\varepsilon} + W(\bm{a},d).
\end{equation*}
$W_\varepsilon$ defines an approximation of $E_\varepsilon$ for
a wave function made of vortices of degrees $d$ at positions $\bm{a}$. In the
next section, we recall how the renormalized energy $W$ and the canonical
harmonic maps can be used to describe the vortex motion.

\subsection{Asymptotic dynamics of vortices}

Let $(\psi_\varepsilon^0)_{\varepsilon>0}$ be a family of initial
conditions for which the vorticity converges to a sum of Dirac masses at
distinct points $\bm a^0=(a_j^0)_{j=1,\dots,N}\in\Omega^{* N}$ that correspond to the
initial positions of the vortices. These vortices are associated with degrees
$d_j = \pm 1$ and the convergence reads, when $\varepsilon\to0$,
\begin{equation}\label{eq:init_cvg}
  J\psi_\varepsilon^0 \to \pi\sum_{j=1}^N d_j\delta_{a_j^0},
\end{equation}
weakly in the sense of measures. If in addition the energy
$E_\varepsilon(\psi_\varepsilon^0)$ is asymptotically optimal as
$\varepsilon\to0$, in the sense that, when $\varepsilon\to0$,
\begin{equation}\label{eq:init_nrj}
  E_\varepsilon(\psi_\varepsilon^0) = W_\varepsilon(\bm{a}^0,d) + o(1),
\end{equation}
then the result is that
\begin{equation*}
  J{\psi_\varepsilon(t)} \to \pi\sum_{j=1}^N d_j\delta_{a_j(t)},
\end{equation*}
where the points $\bm a(t)=(a_j(t))_{j=1,\dots,N}$ evolve according to the
Hamiltonian ODE \eqref{eq:Effective-equation}
\cite{collianderVortexDynamicsGinzburgLandauSchrodinger1998,
  linIncompressibleFluidLimit1999}.
Convergence results are also obtained for the supercurrents
$j(\psi_\varepsilon(t))$ in $L^p(\Omega,\C)$ and for the wave
function $\psi_\varepsilon(t)$ in $\W^{1,p}(\Omega,\C)$ for $p<2$ at time $t>0$. In
particular, we will need later that, under
\eqref{eq:init_cvg}--\eqref{eq:init_nrj},
\begin{equation} \label{eq:L2_cvg}
  \psi_\varepsilon(t) \to u^*(\bm{a}(t),d) \quad\text{in }
  W^{1,p}(\Omega,\C),
  \qquad \text{and} \qquad j(\psi_\varepsilon(t)) \to j(u^*(\bm
  a(t), d)) \quad\text{in } L^p(\Omega,\C),\qquad p<2.
\end{equation}
Finally, note that, while \eqref{eq:GPE} is globally (in time) well-posed for
any $\varepsilon >0$, all
these results on the asymptotic ODE system are valid up to the
first vortex collision, which might happen in finite time.

These asymptotic results only describe the limiting behavior of solutions obtained
from a family of well-prepared initial conditions, but give no quantitative
estimates on the solution to \eqref{eq:GPE} for small, but fixed, $\varepsilon$.
A significant improvement has been obtained in
\cite{jerrardRefinedJacobianEstimates2008}. In this article, the authors
prove in particular that these results hold for small but finite $\varepsilon$,
rather than in the limit $\varepsilon\to0$, given that the initial conditions
are well-prepared in the following sense.

\begin{definition}\label{def:WP}
  A family of initial conditions $(\psi_\varepsilon^0)_{\varepsilon>0}$ is said to be \emph{well-prepared} if
  it satisfies the following assumptions, for some constant $C>0$, $0<\alpha<1$ and
  $\varepsilon$ small enough:
  \begin{enumerate}
    \item there exist $N$ vortices with positions $\bm{a}^0 =
      (a_j^0)_{j=1,\dots,N}\in\Omega^{* N}$ and degrees $d =
      (d_j)_{j=1,\dots,N} \in \set{\pm 1}^N$ such that
      \begin{equation*}
        \norm{J{\psi_\varepsilon^0} - \pi\sum_{j=1}^N
          d_j\delta_{a_j^0}}_{\dot{\W}^{-1,1}} \lesssim\varepsilon^{\alpha},
      \end{equation*}
      and the vortices are distant enough;
    \item the energy of $\psi_\varepsilon^0$ is close to be optimal:
      \begin{equation*}
        E_\varepsilon(\psi_\varepsilon^0) \leq W_{\varepsilon}(\bm{a}^0, d) +
        C\varepsilon^{\frac12}.
      \end{equation*}
  \end{enumerate}
\end{definition}

The main result of \cite{jerrardRefinedJacobianEstimates2008}, in a form
simplified for our purpose, is the following theorem.

\begin{theorem}[{\cite[Theorem
    1]{jerrardRefinedJacobianEstimates2008}}]\label{thm:JS}
  Let $\psi_\varepsilon$ solve \eqref{eq:GPE} with well-prepared initial
  conditions, in the sense of Definition~\ref{def:WP} for some initial vortices with
  positions $\bm{a}^0 = (a_j^0)_{j=1,\dots,N}$ and degrees
  $(d_j)_{j=1,\dots,N}$. Then, there exists
  $\varepsilon_0$, $0~<~\beta$, $\gamma<1$ and $C>0$, depending only on $\Omega$ and the constants in
  $Definition~\ref{def:WP}$, such that, for any $\varepsilon < \varepsilon_0$,
  well-preparedness is preserved along time. In particular,
  \begin{equation}\label{eq:jac_estimate}
    \norm{J{\psi_\varepsilon(t)} - \pi\sum_{j=1}^N
      d_j\delta_{a_j(t)}}_{\dot{\W}^{-1,1}} \lesssim\varepsilon^{\beta},
    \quad\text{and}\quad
    \norm{j(\psi_\varepsilon(t)) - j(u^*(\bm a(t), d))}_{L^{\frac43}} \lesssim
    \varepsilon^\gamma
  \end{equation}
  for any $0\leq t\leq \tau_{\varepsilon,\bm{a}^0}$, where ${\bm a}(t) = (a_j(t))_{j=1,\dots,N}$
  solves the Hamiltonian ODE \eqref{eq:Effective-equation} and $\tau_{\varepsilon,\bm{a}^0}$
  depends on $\varepsilon$ and $\bm{a}^0$ .
\end{theorem}

The theorem proved in \cite{jerrardRefinedJacobianEstimates2008} also contains
other convergence estimates on the energy at any time $t$.
We omit them for the sake
of clarity, as the key ingredients for our numerical method are the Jacobian and
supercurrents estimates \eqref{eq:jac_estimate}.
Regarding the time validity, these estimates are valid as long as $
\tau_{\varepsilon,\bm{a}^0} / \rho_{\tau_{\varepsilon,\bm{a}^0}}^2 \lesssim \log
1/\varepsilon$, where $\rho_{\tau_{\varepsilon,\bm{a}^0}}$ denotes the minimum
distance between two vortices and between the vortices and the boundary.
In particular, if the Hamiltonian trajectory is periodic (which may happen for certain initial data) or the vortices all have the same degree, collision never happens and the approximation is valid up to time $\tau_{\varepsilon,\bm{a}^0} \sim \log 1/\varepsilon$.
We end this section by recalling that all the powers of $\varepsilon$ that appear here
are a bit arbitrary (the original values of $\alpha$, $\beta$ and $\gamma$
are respectively $9/10$, $1/4$ and
$1/9$) and have no reason to be optimal, see
\cite[Theorem 1]{jerrardRefinedJacobianEstimates2008}.

\subsection{Well-prepared initial conditions}\label{sec:WP}

We now focus on the (numerical) construction of a well-prepared family
$(\psi_\varepsilon^0)_{\varepsilon>0}$ in
the sense of Definition~\ref{def:WP}, based on the minimization problem
\eqref{eq:inf_ball}. It is well-known (see
\cite[Corollary 1.5]{pacardLinearNonlinearAspects2000}) that, for $\varepsilon$
small enough, there is a unique minimizer $\phi_{\varepsilon,r_0}$ of
\eqref{eq:inf_ball}: this minimizer is radially symmetric and reads
\begin{equation}\label{eq:ansatz}
  \phi_{\varepsilon,r_0}(x) = f_{\varepsilon,r_0}(|x|)e^{\i
    \theta(x)},
\end{equation}
where $f_{\varepsilon,r_0}$ satisfies the ODE
\begin{equation}\label{eq:ODE}
  \frac1r\prt{rf_{\varepsilon,r_0}'(r)}' - \frac{1}{r^2}f_{\varepsilon,r_0}(r) +
  \frac1{\varepsilon^2}\prt{1-\abs{f_{\varepsilon,r_0}(r)}^2}f_{\varepsilon,r_0}(r) = 0,
\end{equation}
together with the boundary condition
\begin{equation}\label{eq:BVP}
  f_{\varepsilon,r_0}(0) = 0 \quad\text{and}\quad f_{\varepsilon,r_0}(r_0) = 1.
\end{equation}
The 1D boundary value problem \eqref{eq:ODE}--\eqref{eq:BVP} is studied for
instance in \cite{herveEtudeQualitativeSolutions1994} in the case $r_0=\infty$.
For finite $r_0$, it can be solved numerically with very high precision (\eg
using a nonlinear finite element solver), even for $\varepsilon$ as small as
$10^{-3}$, see Figure~\ref{fig:feps}. Finally, note that, by a scaling
argument, $f_{\varepsilon,r_0}$ only depends on the ratio $r_0/\varepsilon$, so
that we may write only $f_\varepsilon$.

\begin{figure}[h!]
  \centering
  \includegraphics[width=0.5\linewidth]{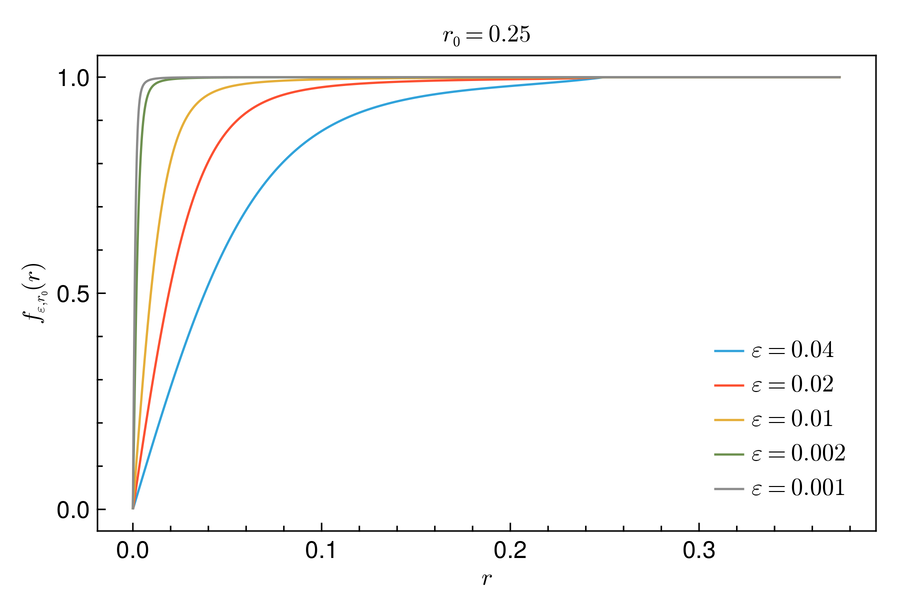}
  \caption{Numerical approximation of $f_{\varepsilon,r_0}$.}
  \label{fig:feps}
\end{figure}

Starting from the localized approximation $\phi_{\varepsilon}$ of a single vortex,
one can then build the following initial condition for \eqref{eq:GPE}: given
vortex positions $\bm{a}^0 = (a_j^0)_{j=1,\dots,N}\in\Omega^{* N}$ with degrees
$d = (d_j)_{j=1,\dots,N}\in\set{\pm1}^N$, we consider
\begin{equation}\label{eq:init}
  \psi_{\varepsilon}^*(x; \bm{a}^0, d) =
  u^*(x; \bm{a}^0, d) \prod_{j=1}^N f_{\varepsilon}(|x -
  a_j^0|), \quad x\in\overline\Omega,
\end{equation}
where $u^*$ is the canonical harmonic map defined in \eqref{eq:canon}, with
phase factor $H$ computed as in \eqref{eq:H_rebuilt}.
This numerical strategy can be seen as an operator from the manifold
$\{u^*(\bm{a},d),\ \bm{a}\in\Omega^{* N},\ d\in\set{\pm1}^N\}$ to $H^1(\Omega,\C)$
which smoothes out the canonical harmonic map $u^*(\bm{a},d)$ at the
singularities locations.
Moreover, such initial conditions are also used for instance in
\cite{baoNumericalStudyQuantized2013,baoNumericalStudyQuantized2014}
and are well-prepared in the sense of Definition~\ref{def:WP}, as
shown in \cite[Lemma 14]{jerrardRefinedJacobianEstimates2008} in the case of
Neumann boundary conditions and that we recall here.

\begin{lemma}[{\cite[Lemma 14]{jerrardRefinedJacobianEstimates2008}}]\label{lem:WP}
  For any $\bm{a}\in\Omega^{* N}$ and $d\in\set{\pm 1}^N$ and for $r_0$ small enough, there
  exists some constant $C>0$ such that the function
  $\psi^* = \psi_{\varepsilon}^*(\bm{a},d)$ constructed above satisfies
  \begin{equation}\label{eq:WP_nrj}
    E_\varepsilon(\psi_{\varepsilon}^*) \leq W_{\varepsilon}(\bm{a},d)
    + CN\prt{\frac{\varepsilon}{r_0}}^2
  \end{equation}
  and
  \begin{equation}\label{eq:WP_jac}
    \norm{J{\psi_{\varepsilon}^*} - \pi\sum_{j=1}^N
      d_j\delta_{a_j}}_{\dot{\W}^{-1,1}} \lesssim \varepsilon.
  \end{equation}
  In particular, the assumptions from Definition~\ref{def:WP} are satisfied.
\end{lemma}

We end this section by stating another important property of such well-prepared
initial conditions:
\begin{proposition}\label{prop:L2_WP}
  Let $\psi_{\varepsilon}^*(\bm a, d)$ be as in Lemma~\ref{lem:WP} for
  $\bm{a}\in\Omega^{* N}$, $d\in\set{\pm1}^N$ and $r_0$ small enough.
  Then,
  \begin{align}
    &\norm{\psi_{\varepsilon}^*(\bm{a},d) - u^*(\bm{a},d)}_{L^2(\Omega,\C)} \lesssim
    \varepsilon. \\
    &\norm{j\bigr(\psi_{\varepsilon}^*(\bm{a},d)\bigr) -
      j\bigr(u^*(\bm{a},d)\bigr)}_{L^p(\Omega)} \lesssim_p
    \epsilon^{\frac{2}{p} - 1} \quad \mbox{for $1\leq p < 2$.} \label{eq:jest}
  \end{align}
\end{proposition}

\begin{proof}
  The convergence is a standard result of Ginzburg--Landau theory and a direct
  consequence of the energy bound \eqref{eq:WP_nrj}, see \eg
  \cite{linIncompressibleFluidLimit1999}. Then, the bound on the
  $L^2$ norm of the difference is obtained thanks to the specific structure of
  $\psi_{\varepsilon}^*$. As $\psi_{\varepsilon}^*$ coincides
  with $u^*$ outside of the balls $B_{r_0(a_j)}$, we have
  \[
    \int_\Omega|\psi_{\varepsilon}^* - u^*|^2 = \sum_{j=1}^N
    \int_{B_{r_0}(a_j)} |1 - f_{\varepsilon}(\cdot-a_j)|^2 =
    N\int_{B_{r_0(0)}} |1-f_{\varepsilon}|^2.
  \]
  Recalling that $1\geq f_{\varepsilon}\geq f_{\varepsilon,\infty}$ together
  with the lower bound
  $f_{\varepsilon,\infty}(r)\geq\max\prt{0,1-c\prt{\frac\varepsilon r}^2}$ (see
  \cite{shafrirRemarksSolutionsDelta1994}), we get
  \begin{align*}
    \int_\Omega|\psi_{\varepsilon}^* - u^*|^2 \lesssim
    \int_0^{r_0} \min\prt{1,c^2\prt{\frac\varepsilon r}^4} r \mathrm{d}r = \int_0^{\epsilon \sqrt{c}} r \mathrm{d}r + c^2 \varepsilon^4 \int_{\varepsilon \sqrt{c}}^{r_0} \frac{1}{r^3} \mathrm{d} r \lesssim \epsilon^2.
  \end{align*}
  For the supercurrents, we note that, since $|u^*|^2 =1$ and $f_\epsilon$ is
  real-valued, we find
  \begin{align*}
    j(\psi_\epsilon^*) = j(u^*) \prod_{j=1}^N f_\epsilon(|x-a_j|)^2
  \end{align*}
  Thus using that $|j(u^*)| \lesssim \frac{1}{|x-a_j|}$ for $x$ close to $a_j$ (which follows from \eqref{eq:j_G}) and the bound $1 \geq f_{\varepsilon,\infty}(r) \geq\max\prt{0,1-c\prt{\frac\varepsilon r}^2}$, we find
  \begin{align*}
    \norm{j(\psi_\varepsilon^*) - j(u^*)}_{L^p(\Omega)}^p &= \sum_{j=1}^N \int_{B_{r_0}(a_j)} |j(u^*)|^p \bigr(1-f_\varepsilon^2(|x-a_j|)\bigr)^p \mathrm{d} x \\
    &\lesssim N \int_0^{\varepsilon \sqrt{c}} \frac{1}{r^p} r \mathrm{d} r + c^{2p} \varepsilon^{2p} \int_{\varepsilon \sqrt{c}}^{r_0} \frac{1}{r^{3p}} r \mathrm{d}r \lesssim  \epsilon^{2-p}.
  \end{align*}
\end{proof}

\section{Numerical simulation of the Hamiltonian dynamics}\label{sec:HD_sim}

In the previous section, we recalled the main analytical results on the vortex
motion in the singular limit $\varepsilon\to0$. We also mentioned sufficient
conditions for the initial conditions $\psi_\varepsilon^0$ to be well-prepared,
as well as a numerical strategy to build such functions.
Now, we focus on the vortex motion itself and we present an efficient numerical
method for the numerical simulation of the Hamiltonian dynamics
\begin{equation}\label{eq:Effective-equation2}
  \begin{cases}
    \dot{a}_j(t) = -\frac1\pi d_j\J\nabla_{a_j} W({\bm a}(t), d),\\
    {a}_j(0) = {a}^0_j,
  \end{cases}
  \quad \text{with} \quad \J = \begin{bmatrix} 0 & 1 \\ -1 & 0 \end{bmatrix},
\end{equation}
for given initial positions $\bm{a}^0 \in\Omega^{* N}$ and
degrees $d\in\set{\pm1}^N$. Recall that
\begin{equation}\label{eq:W}
  \Forall (\bm{a},d)\in\Omega^{* N}\times\{\pm1\}^N,\quad
  W(\bm{a}, d) = - \pi\sum_{1\leq i \neq j\leq N} d_id_j \ln|a_i - a_j|
  - \pi\sum_{j=1}^N d_jR(a_j; \bm{a}, d),
\end{equation}
with $R$ the solution to \eqref{eq:R}:
\begin{equation*}
  \begin{cases}
    \ds\Delta R = 0 \quad\text{in }\Omega,\\
    \ds R = - \sum_{j=1}^N d_j \ln|x - a_j| \quad\text{on }\partial\Omega.
  \end{cases}
\end{equation*}
The numerical resolution of the ODE \eqref{eq:Effective-equation2} requires the evaluation
of the gradient $\nabla_{a_j} W({\bm a}(t), d)$, which can be done using the
following identity (see \eg \cite{baoNumericalStudyQuantized2013} or
\cite[Theorem VIII.3]{bethuelGinzburgLandauVortices1994} in the case of Dirichlet
boundary conditions),
\[
  \nabla_{a_j} W({\bm a}, d) = - 2 \pi d_j \nabla_x \prt{
    R(x; \bm{a}, d) + \sum_{k\neq j}^N d_k \ln|x-a_k|
  }\Biggr|_{x=a_j}.
\]
The solution $R$ to \eqref{eq:R} being harmonic and the boundary condition being
smooth as long as the vortices stay away from the boundary, it makes sense to use harmonic
polynomials to solve it with spectral accuracy for a negligible cost. We now detail
how to solve \eqref{eq:R} at each time step with such a method. For the sake of
simplicity, we choose $\Omega$ as the unit disk in $\R^2$ but this strategy is
easily transposable to other domains (see Remark~\ref{rmk:other_domains}). The
choice of $\Omega$ as the unit disk also makes sense since trapping potential
with radial symmetry are common in experimental setups, see
\cite{simulaEmergenceOrderTurbulence2014} and references therein.

We define the family $(h_n)_{n \geq 1}$ of harmonic polynomials as $h_1
\equiv 1$ and
\begin{equation}\label{eq:harm}
  \Forall m \in \N,\ m\geq 1,\ \Forall (x,y)\in\R^2, \quad \begin{cases}
    h_{2m}(x,y) = \Re{(x +\i y)^m},\\
    h_{2m+1}(x,y) = \Im{(x +\i y)^m}.
  \end{cases}
\end{equation}
For any integer $n\geq 1$, we have $\Delta h_n = 0$. Moreover, $h_1$ is
of degree $0$ and $h_{2m}, h_{2m+1}$ are the only two linearly independent
harmonic polynomials of degree $m\geq 1$ in two dimensions. In the case where
$\Omega$ is the unit disk, the restriction of these polynomials to the boundary
$\partial\Omega$ is nothing else than the basis of the Fourier modes for
$2\pi$-periodic functions, \ie
\[
  \Xc_n \coloneqq
  \Span\set{h_k|_{\partial\Omega},\ 1 \leq k \leq 2n+1} =
  \Span\set{\partial\Omega\ni z\mapsto z^k,\
    \partial\Omega\ni z\mapsto\overline{z}^\ell,\ 0
    \leq k \leq n,\ 1\leq \ell\leq n}.
\]
Solving numerically \eqref{eq:R} can then be done following these three steps:
\begin{enumerate}
  \item Choose a maximum degree $n$, and fix it once and for all. Denote by
    $\P_n$ the $L^2(\partial\Omega)$-orthogonal
    projection operator from $L^2(\partial\Omega)$ to $\Xc_n$:
    for any $g\in L^2(\partial\Omega)$,
    \[
      (\P_n g)(e^{\\i \theta}) = \sum_{k=-n}^n \widehat{g}(k) e^{\i k\theta},
    \]
    where
    \[
      \widehat{g}(k) = \frac1{2\pi} \int_0^{2\pi} e^{-\i
        k\theta}g(e^{\\i \theta})\d\theta.
    \]
  \item Compute the Fourier coefficients $(\widehat{g}_{a}(k))_{-n\leq k \leq n}$
    of the Dirichlet boundary condition in \eqref{eq:R}
    \[[0,2\pi) \ni \theta \mapsto g_a(e^{\\i \theta}) \coloneqq -\sum_{j=1}^N
      d_j\ln|(\cos \theta,\sin \theta)-a_j|\] up to order $n$, for instance
    using a Fast Fourier Transform (FFT).
  \item Compute the (approximate) solution $R_n$ to \eqref{eq:R} as the harmonic expansion of
    $\P_ng_a$:
    \begin{equation*}
      \Forall r \in [0,1),\ \Forall \theta\in[0,2\pi), \quad
      R_n(re^{\\i \theta}) = \sum_{k=-n}^n r^{|k|}\widehat{g}_{a}(k)e^{\i k\theta},
    \end{equation*}
    which is still harmonic by linear combination of harmonic functions.
\end{enumerate}
This strategy differs from the one in
\cite{baoNumericalStudyQuantized2013,baoNumericalStudyQuantized2014}, where the
PDE \eqref{eq:R} is typically solved with finite differences in the case
$\Omega=(-1,1)^2$, or with the Fourier pseudo-spectral method in the
$\theta$-direction and with the FEM in the $r$-direction when $\Omega$ is the
unit disk. Here, we suggest to use the smoothness of the boundary data $g_a$ to
use a Fourier approximation, which requires only one FFT to obtain an
approximate solution to \eqref{eq:R}. Moreover, the smoothness of the boundary
data $g_a$ enables to reach spectral accuracy, see Section~\ref{sec:error}
for more details.

To conclude this section, we provide in Figure~\ref{fig:trajectories}
first numerical results describing vortex
trajectories in the singular limit $\varepsilon\to0$. The computational domain
is the unit disk $\Omega=B_1(0)$, with different initial conditions $\bm{a}^0$
and $d$ for each
case. The ODE \eqref{eq:Effective-equation2} is solved with a 4th-order Runge--Kutta method
(RK4, see \eg \cite[Chapter II]{hairerSolvingOrdinaryDifferential1993})
with time step $\delta t=10^{-3}$ up to some time $T$. The PDE \eqref{eq:R} is
solved at each time step using harmonic polynomials up to degree $n=64$, as
described above. We consider here eight cases, some of which are directly taken from
\cite{baoNumericalStudyQuantized2014} in order to validate our method. Note that
all the trajectories were obtained in a few seconds on a personal laptop, a
significant improvement when compared to the computational cost of the complete
PDE \eqref{eq:GPE} for small $\varepsilon$.

\begin{table}[h!]
  \centering
  \begin{tabular}{cccc}
    \toprule
    \textbf{Case}       & $N$ & $\bm a^0$                                & $d$                                          \\
    \hline
    \textbf{Case $1$} & $2$          & $\prt{(-0.5,0.0), (0.5,0.0)}$                          & $(1,1)$                                                \\
    \textbf{Case $2$} & $2$          & $\prt{(-0.7,0.0), (0.7,0.0)}$                          & $(-1,1)$                                               \\
    \textbf{Case $3$} & $4$          & $\prt{(-0.6,-0.6), (-0.6,0.6), (0.6,0.6), (0.6,-0.6)}$ & $(1,-1, 1, -1)$                                        \\
    \textbf{Case $4$} & $2$          & $\prt{(-0.25,-0.25), (0.25,0.25)}$                     & $(-1,1)$                                               \\
    \textbf{Case $5$} & $3$          & $\prt{(-0.5,0.0), (0.5,0.0), (0.0,0.0)}$               & $(1,1,1)$                                              \\
    \textbf{Case $6$} & $3$          & $\prt{(-0.5,0.0), (0.5,0.0), (0.0,0.0)}$               & $(1,1,-1)$                                             \\
    \textbf{Case $7$} & $9$          & $\prt{(x,y)}$ for $x,y\in\set{-0.3,0,0.3}$           & $d_i=1$ for $i=1,\dots,9$                           \\
    \textbf{Case $8$} & $9$          & $\prt{(x,y)}$ for $x,y\in\set{-0.3,0,0.3}$           & $d_i=1$ for $i=1,\dots,9$, $i\neq5$, $d_5=-1$  \\
    \bottomrule
  \end{tabular}
  \caption{Different settings for the initial conditions of the reduced
    Hamiltonian dynamics.}
\end{table}

\begin{figure}[p!]
  \includegraphics[width=0.47\linewidth]{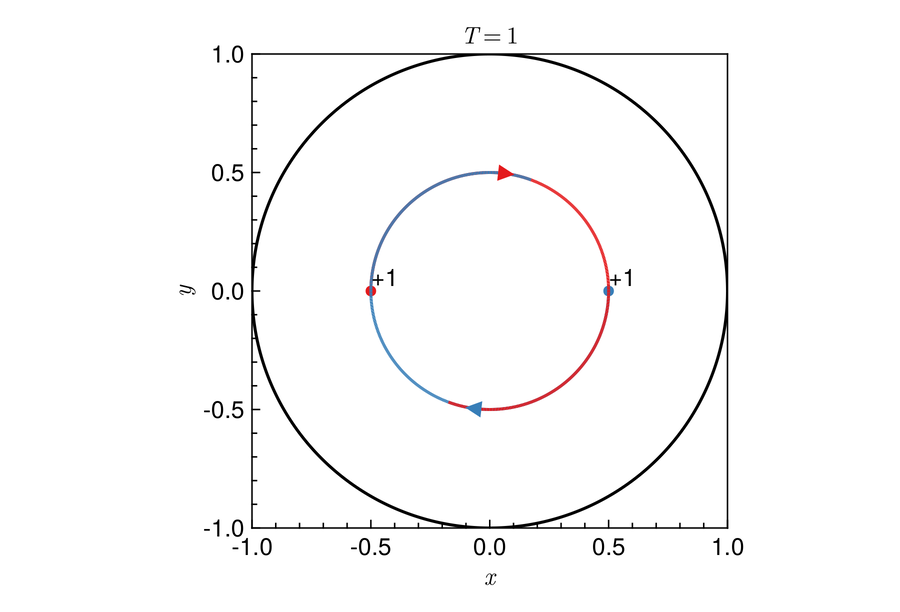}\hfill
  \includegraphics[width=0.47\linewidth]{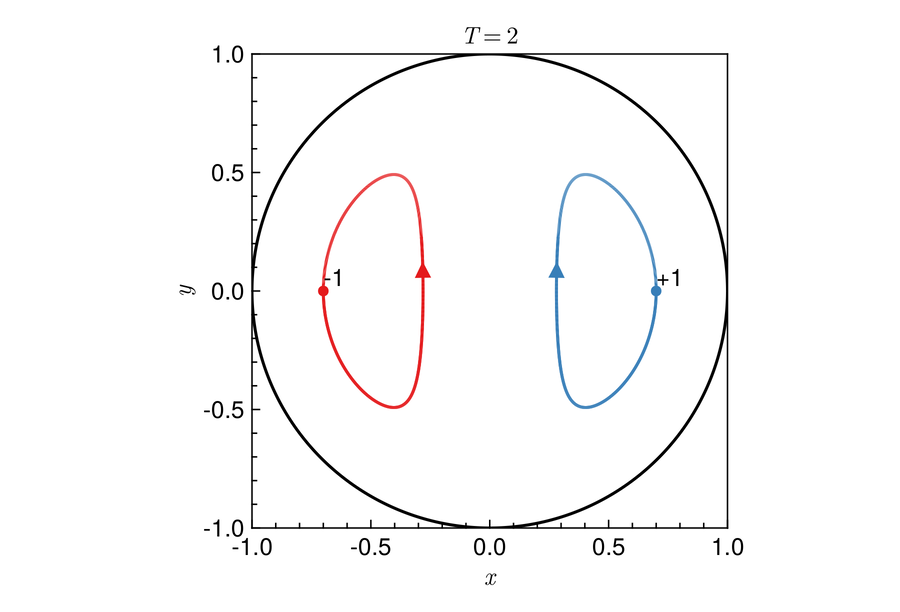}\\
  \includegraphics[width=0.47\linewidth]{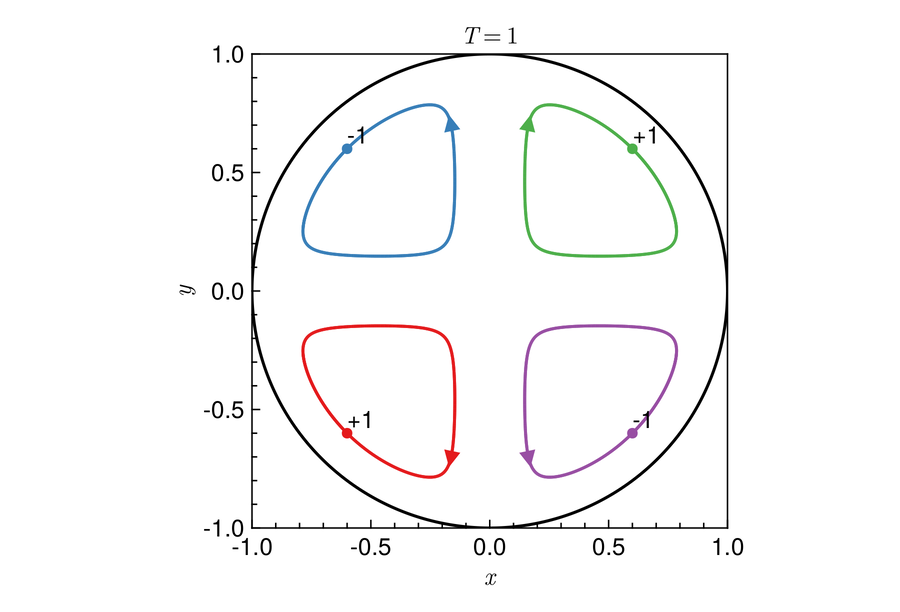}\hfill
  \includegraphics[width=0.47\linewidth]{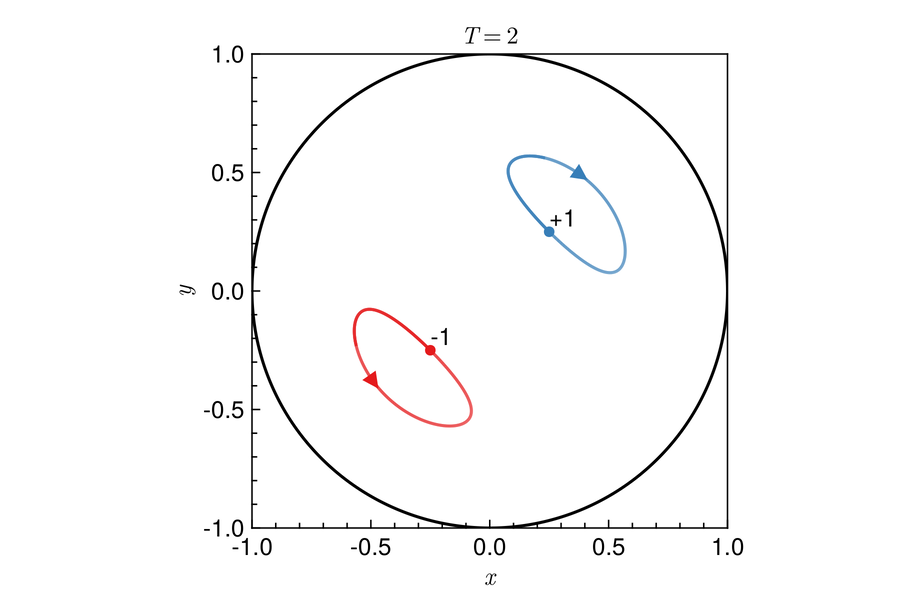}\\
  \includegraphics[width=0.47\linewidth]{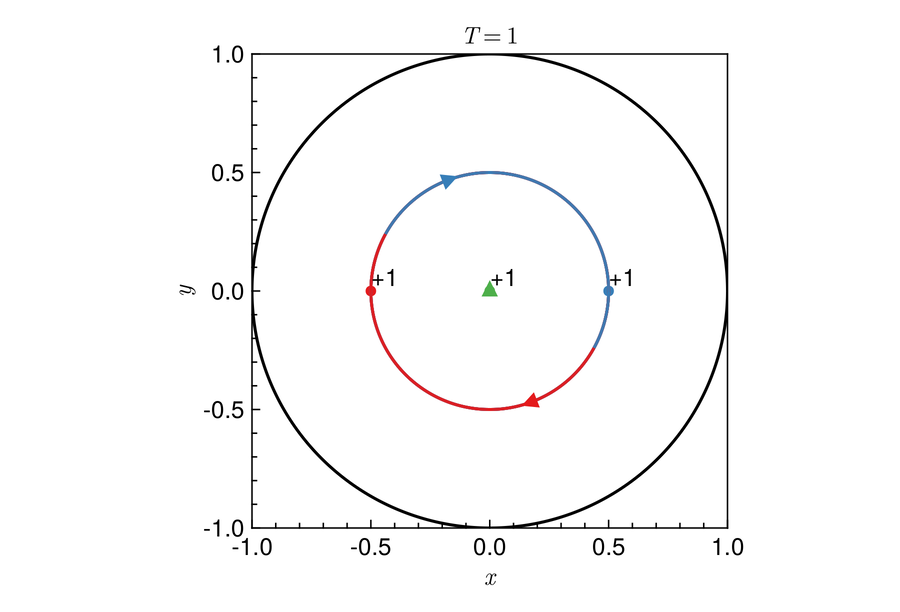}\hfill
  \includegraphics[width=0.47\linewidth]{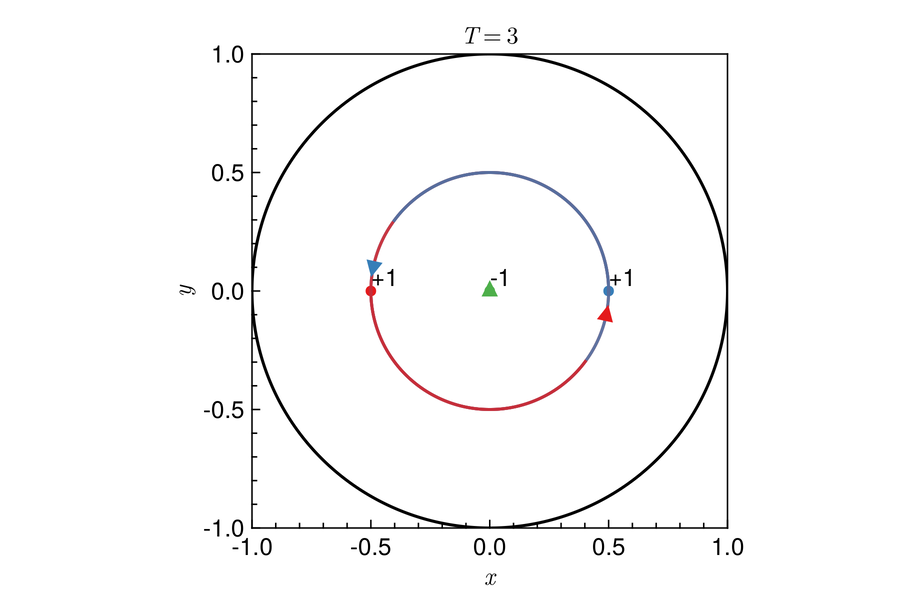}\\
  \includegraphics[width=0.47\linewidth]{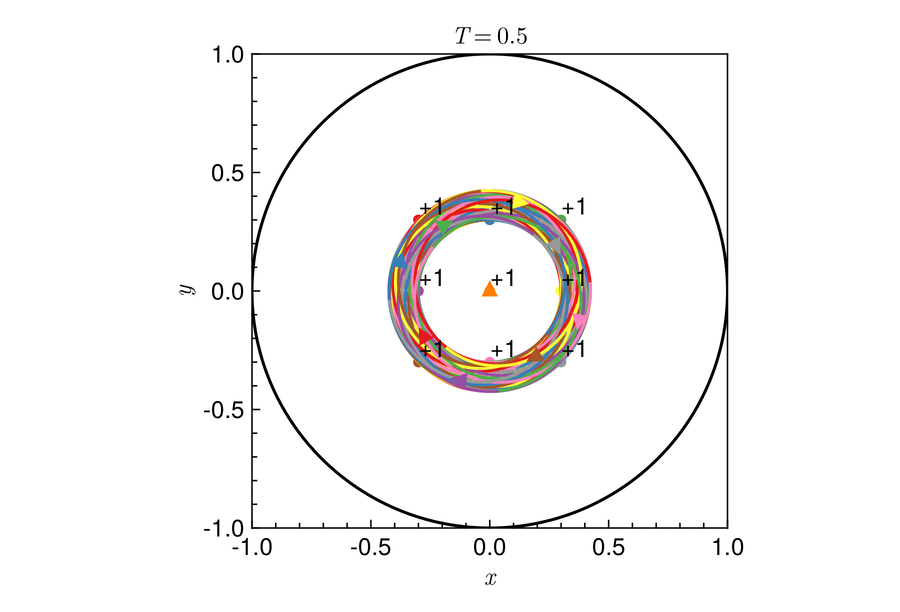}\hfill
  \includegraphics[width=0.47\linewidth]{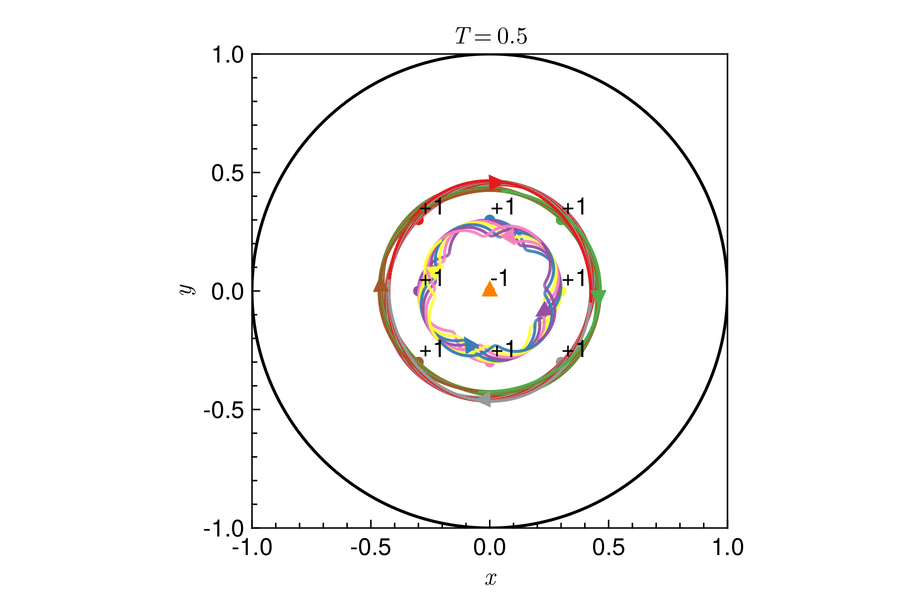}\\
  \caption{Trajectories of vortices in the singular limit $\varepsilon\to0$,
    cases 1 to 8, from left to right and up to bottom.
    Initial conditions are represented by dots and the different degrees
    used are specified for each trajectory.}
  \label{fig:trajectories}
\end{figure}

\begin{remark}[Generic smooth domains]
  The same strategy can be used for more general smooth domains. From a
  numerical point of view, one needs to consider the harmonic polynomials in 2D
  defined in \eqref{eq:harm}, orthonormalize them in $\L^2(\partial\Omega)$ and
  then project the boundary data in this basis before obtaining the (approximated)
  solution to \eqref{eq:R} by harmonic extension. For instance, we obtained
  similar results with $\Omega=[0,1]^2$ and we suspect the spectral convergence
  that we could prove (see Theorem~\ref{thm:errorestimate}) for the case where
  $\Omega$ is the unit disk to also hold for more general domains.
  \label{rmk:other_domains}
\end{remark}

\begin{remark}[ODE solvers]
  We chose here a RK4 solver for the sake of simplicity, but other solvers can
  be used, such as symplectic solvers for non-separable Hamiltonian systems. As
  long as the associated error analysis is known, the time discretization error
  in Theorem~\ref{thm:errorestimate} can be adapted accordingly.
\end{remark}

\section{Numerical simulation of the Gross--Pitaevskii equation: a new method
  based on vortex tracking}\label{sec:method}

In this section, we present the main contribution of this paper:
a new numerical method for the simulation of \eqref{eq:GPE} in the
regime where $\varepsilon>0$ is small but finite. First, we start by describing
the method and then we provide some simulations when the initial conditions are
given by known vortex positions.

\subsection{Description of the method and mathematical justification}\label{sec:method_1}

Instead of solving directly \eqref{eq:GPE}, we propose to take advantage of the
well-known behavior in the singular limit $\varepsilon\to0$ as well as the
properties of the smoothing procedure used to obtain well-prepared initial
conditions in Section~\ref{sec:WP}. This yields the following algorithm:
\begin{enumerate}
  \item Define initial vortex positions $\bm{a}^0\in\Omega^{* N}$ and degrees
    $d\in\set{\pm1}^N$. Set up the initial phase $H$ such that the initial
    condition $\psi_{\varepsilon}^0$  from Lemma~\ref{lem:WP} satisfies
    homogeneous Neumann boundary conditions.
    Compute and store the radial function $f_{\varepsilon}$.
  \item Evolve $\bm{a}(t)$ according to the ODE \eqref{eq:Effective-equation2} up to some maximum
    time $T$.
  \item At time $t>0$, build back an approximation of the wave function from the
    vortex positions ${\bm a}(t)$ as
    \begin{equation}\label{eq:smoothed_approx}
      \psi_{\varepsilon}^*(t) = \psi_{\varepsilon}^*({\bm a}(t), d) =
      u^*(\cdot; {\bm a}(t), d) \prod_{j=1}^N f_{\varepsilon}(|\cdot -
      a_j(t)|),
    \end{equation}
    where $u^*(x; {\bm a}(t), d)$ is the canonical harmonic map defined by
    \begin{equation*}
      u^*(x; {\bm a}(t), d) = \exp\prt{\i H(x)} \prod_{j=1}^N
      \prt{\frac{x - a_j}{|x - a_j|}}^{d_j},
    \end{equation*}
    with $H$ the unique zero-mean harmonic function solving
    \begin{equation}\label{eq:H_rebuilt}
      \begin{cases}
        \Delta H = 0 \quad\text{in } \Omega,\\ \displaystyle
        \partial_\nu H(x) = - \sum_{j=1}^N d_j\partial_\nu\theta(x-a_j)
        = \sum_{j=1}^N d_j \partial_\tau \ln|x-a_j|
        \quad\text{on } \partial\Omega.
      \end{cases}
    \end{equation}
    Computing $H$ with the appropriate boundary conditions is required so that
    the approximate solution $\psi_{\varepsilon}^*(t)$ satisfies homogeneous
    Neumann boundary conditions. This Laplace's equation
    can be solved using the same harmonic polynomial basis than the
    one used to solve \eqref{eq:R}. However, note that the harmonic function $H$
    is defined only up to a constant: this implies that the reconstructed wave
    function $\psi_\varepsilon^*$ is an approximation of $\psi_\varepsilon$ only
    up to a constant phase. However, when looking at quantity of interest such
    as the supercurrents or the vorticity, this phase factor has no influence.
\end{enumerate}

The method we propose can be summarized by the diagram in Figure~\ref{fig:diagram}.
\begin{figure}[h!]
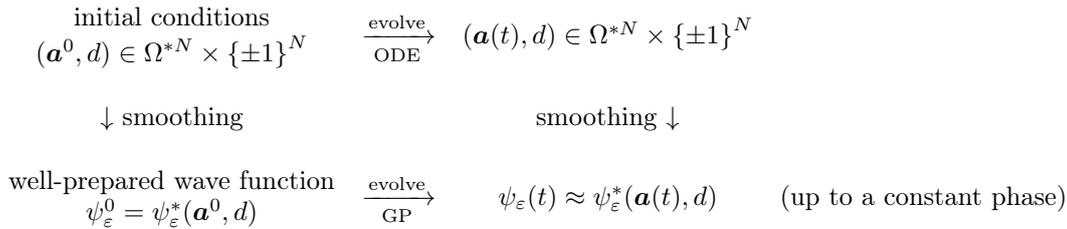

  \centering
  \begin{tabular}{cccc}
    \begin{tabular}[c]{@{}c@{}}initial conditions\\
      $(\bm{a}^0,d)\in\Omega^{* N}\times\set{\pm1}^N$\end{tabular}
    & $\xrightarrow[\text{ODE}]{\text{evolve}}$ & $({\bm
      a}(t),d)\in\Omega^{* N}\times\set{\pm1}^N$                      & \\
    \multicolumn{1}{l}{}
    & \multicolumn{1}{l}{}                      & \multicolumn{1}{l}{}
    & \\
    $\downarrow$ smoothing                                                                                                          &                                           &  smoothing $\downarrow$\\
    \multicolumn{1}{l}{}
    & \multicolumn{1}{l}{}                      & \multicolumn{1}{l}{}
    & \\
    \begin{tabular}[c]{@{}c@{}}well-prepared wave function\\ $\psi_\varepsilon^0
      = \psi_{\varepsilon}^*(\bm{a}^0,d)$\end{tabular} &
    $\xrightarrow[\text{GP}]{\text{evolve}}$ &
    $\psi_\varepsilon(t)\approx\psi_{\varepsilon}^*({\bm a}(t),d)$ &
    \text{(up to a constant phase)}
  \end{tabular}
  \caption{Diagram summarizing the numerical simulation of the GP
    equation \eqref{eq:GPE} via vortex tracking.}
  \label{fig:diagram}
\end{figure}

One can then observe that, the smaller $\varepsilon$, the more commutative the
diagram from Figure~\ref{fig:diagram} is.
Mathematically, the justification of our method is a simple combination of the
results we recalled in Section~\ref{sec:lite}, that we present here in three
points.
\begin{proposition}\label{prop:method}
  Let $\bm{a}^0\in\Omega^{* N}$ and $d\in\set{\pm1}^N$ be given as initial data. Let
  ${\bm a}(t)$ evolves according to \eqref{eq:Effective-equation2}.
  Let $\psi_\varepsilon$ be the solution to the GP equation
  \eqref{eq:GPE} with
  initial conditions $\psi_\varepsilon^0 = \psi_{\varepsilon}(\bm{a}^0)$ for
  $\varepsilon$ and $r_0$ small enough. Then, for all $0\leq t\leq\tau_{\varepsilon,\bm{a}^0}$,
  \begin{enumerate}
    \item Both $\psi_\varepsilon(t)$ and $\psi_{\varepsilon}^*(t)$ are close to
      be energetically optimal:
      \[
        E_\varepsilon(\psi_\varepsilon(t)), E_\varepsilon(\psi_{\varepsilon}^*(t))
        \leq W_{\varepsilon}({\bm a}(t), d) + C\varepsilon^{\frac12}.
      \]
    \item Up to a constant phase, the error
      $\norm{\psi_\varepsilon(t) - \psi_{\varepsilon}^*(t)}_{L^2}$ goes
      to $0$ as $\varepsilon \to 0$.
    \item The Jacobians and supercurrents of $\psi_\varepsilon$ and $\psi_{\varepsilon}^*$
      are close, in the sense
      \[
        \norm{J\psi_\varepsilon(t) - J\psi_{\varepsilon}^*(t)}_{\dot\W^{-1,1}}
        \lesssim\varepsilon^\beta
        \quad\text{and}\quad
        \norm{j(\psi_\varepsilon(t)) - j(\psi_{\varepsilon}^*(t))}_{L^{\frac43}}
        \lesssim\varepsilon^\gamma.
      \]
  \end{enumerate}
  Here, $\beta$, $\gamma\leq\frac12$ and $\tau_{\varepsilon,\bm{a}^0}$ are defined in
  Theorem~\ref{thm:JS}.
\end{proposition}

\begin{proof}
  These properties are all consequences of the general fact that, if the initial
  conditions are well-prepared in the sense of Definition~\ref{def:WP}, then
  well-preparedness is conserved along time. In particular,

  (1) is a direct consequence of the conservation of energy for \eqref{eq:GPE}
  and Hamiltonian systems as well as the bound \eqref{eq:WP_nrj} satisfied by the
  smoothed canonical harmonic maps.

  (2) is obtained after recalling that both $\psi_\varepsilon(t)$ and
  $\psi_{\varepsilon}^*(t)$ converge towards $u^*({\bm a}(t),d)$ in
  $L^2(\Omega)$ up to a constant phase
  (see \eqref{eq:L2_cvg} and Proposition~\ref{prop:L2_WP}).

  (3) is the triangular inequality applied to \eqref{eq:jac_estimate} and
  \eqref{eq:WP_jac} for the Jacobian and to \eqref{eq:jac_estimate} and
  \eqref{eq:jest} for the supercurrents, the bound $\gamma \leq \frac12$ coming from
  \eqref{eq:jest}.
\end{proof}

At this point, some comments have to be made.
First, the smaller epsilon, the more accurate is
the approximation $\psi_{\varepsilon}^*$. This is a clear improvement with
respect to standard techniques used to solve numerically nonlinear
Schr\"odinger type equations such as those mentioned in the introduction, where
very small values of $\varepsilon$ requires very fine discretization in both
space and time to ensure numerical stability. Here, $\varepsilon$ only intervenes
through the numerical resolution of the nonlinear ODE \eqref{eq:ODE} which is
only one dimensional, and thus easily approximated with high accuracy. Moreover,
this step only has to be done once at the beginning of the simulation as the
function $f_{\varepsilon}$ used to build the approximation
$\psi_{\varepsilon}$ is time-independent and can be saved for the whole
simulation.
Of course, such an improvement is only possible thanks
to the well-known analytical theory of the singular limit $\varepsilon\to0$
presented in the beginning of this paper.
Second, there is, up to our knowledge, no available quantitative bounds on the
error $\norm{\psi_\varepsilon(t) - u^*({\bm a}(t),d)}_{L^2}$
in the limit $\varepsilon\to0$ as most of the asymptotic
results are based on compactness arguments. The only quantitative bounds
available in the literature are the refined Jacobian estimates derived in
\cite{jerrardRefinedJacobianEstimates2008}, but these are bounds in weak norms
that are not computable numerically. However, in the same paper, the authors
also derived bounds on the $L^p$ norms of the supercurrents for $p<2$, which can be
exploited to evaluate numerically the accuracy of our method, see
Section~\ref{sec:error}.
Finally, note that the asymptotic dynamics of the vortices
is only valid as long as the vortices stay away from each other and from the
boundary, a condition traduced by the requirement of $r_0$ being small enough
for the initial condition to be well-prepared and to reconstruct the
wave function without any overlap between the vortices.
In particular, nonlinear physical phenomena triggered by overlapping
vortices (or vortices too close to the boundary), such as the radiations and
sound waves observed numerically in \cite{baoNumericalStudyQuantized2014},
cannot be produced with our method. However, a possible improvement could be to
use the vortex-tracking method up to some time $t$ where the vortices become
close enough to each other and then switch to a standard PDE solver to simulate
the vortex interaction.

\subsection{Vortex positions as initial conditions}\label{sec:numres_1}

We now present some numerical simulations of approximate solutions to
\eqref{eq:GPE} in the regime of small $\varepsilon$. The reference
vortex trajectories are those from Section~\ref{sec:HD_sim}. The solution of
\eqref{eq:ODE} is approximated numerically using 1D finite elements with small
mesh size $\delta r = 10^{-5}$ and $r_0=0.1$. Finally, the solution to
Laplace's equation \eqref{eq:H_rebuilt} in the smoothing process at time $t>0$ is computed using
harmonic polynomials of degree $n=64$. Here, the initial vortex positions $\bm{a}^0$
are provided as input data and the initial phase $H$ is computed in order to
match the Neumann boundary conditions (see \eqref{eq:H_rebuilt}).
We present the numerical results obtained by the application of our
numerical method to cases 1, 2, 3 and 6 from Section~\ref{sec:HD_sim}, for
$\varepsilon=10^{-2}$. Moreover, note that the phases displayed
in Figures~\ref{fig:smooth_case1}--\ref{fig:smooth_case6} are those of the
canonical harmonic map $u^*(\bm a(t), d)$ when far from the vortices.

\begin{remark}[Unknown initial vortex locations]
  One could imagine to use the same numerical method with \emph{unknown} vortex
  positions and degrees. An additional localization step is therefore required,
  for instance using one of the algorithms presented in
  \cite{dujardinNumericalStudyVortex2022,kaltIdentificationVorticesQuantum2023}.
  If one can estimate the localization error, it can then be transported by looking at the difference
  between $\bm{a}^0$ and $\bm{b}^0$ in Theorem~\ref{thm:errorestimate}.
\end{remark}

\begin{figure}[p!]
  \includegraphics[width=0.48\linewidth]{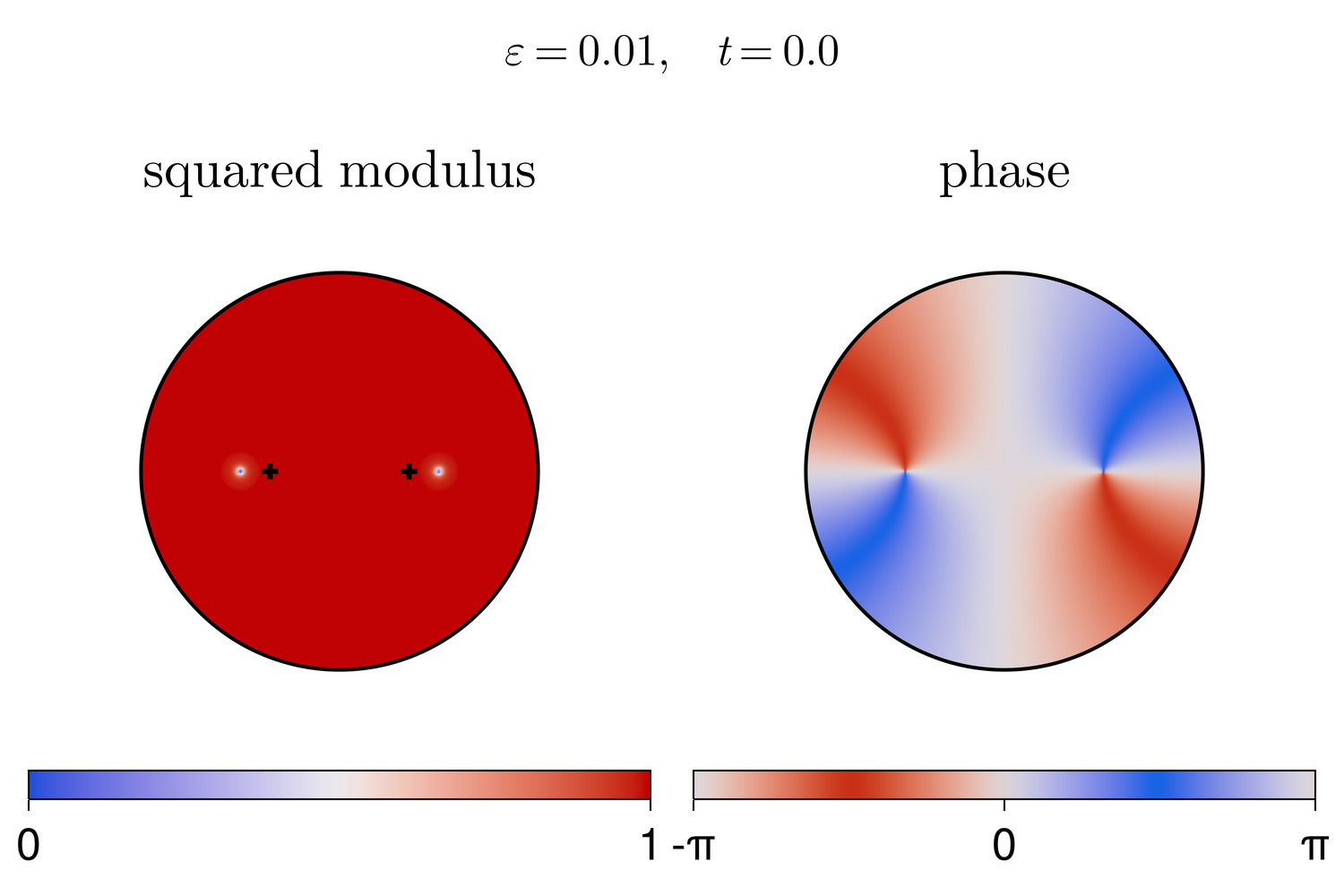}\hfill
  \includegraphics[width=0.48\linewidth]{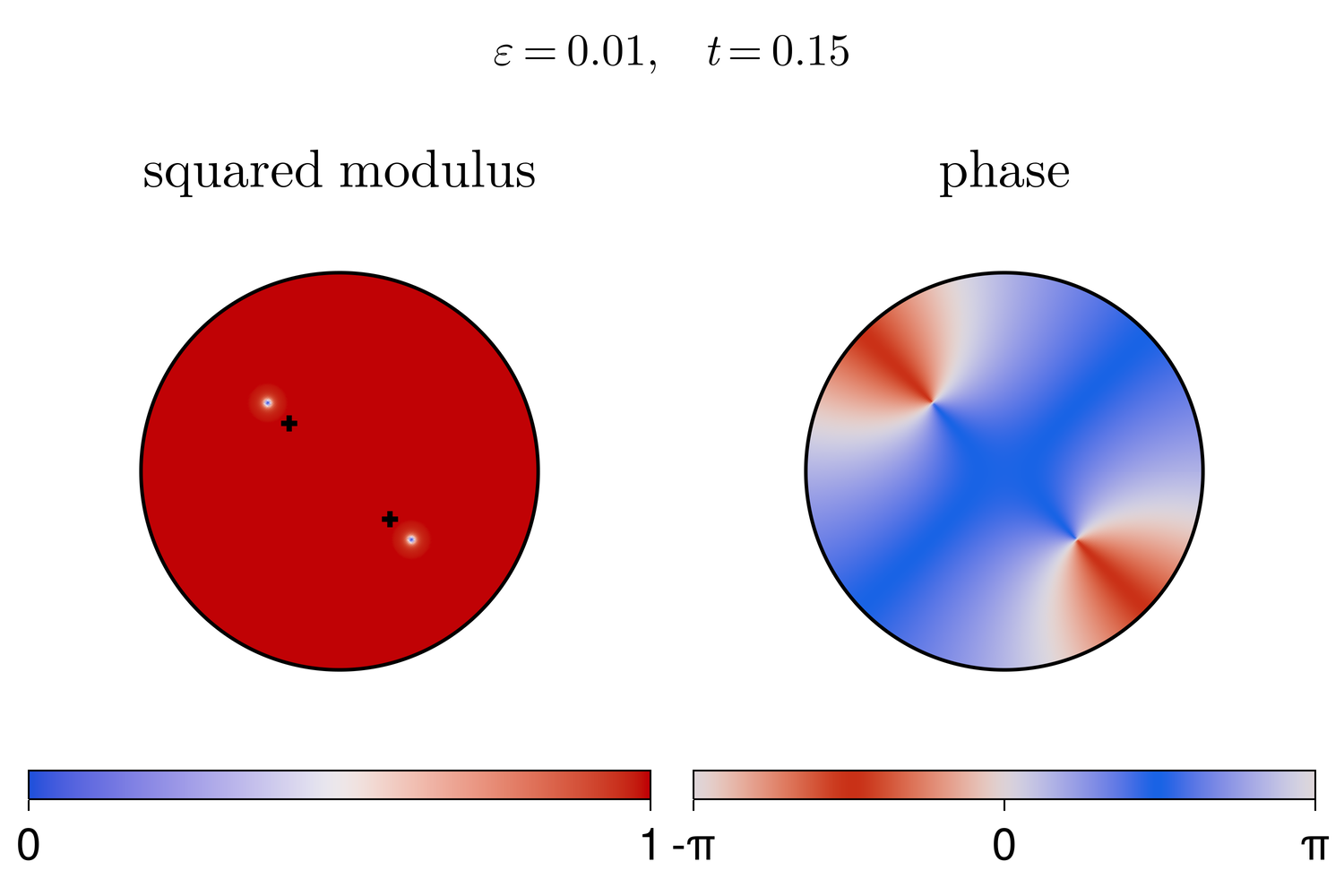}\\
  \includegraphics[width=0.48\linewidth]{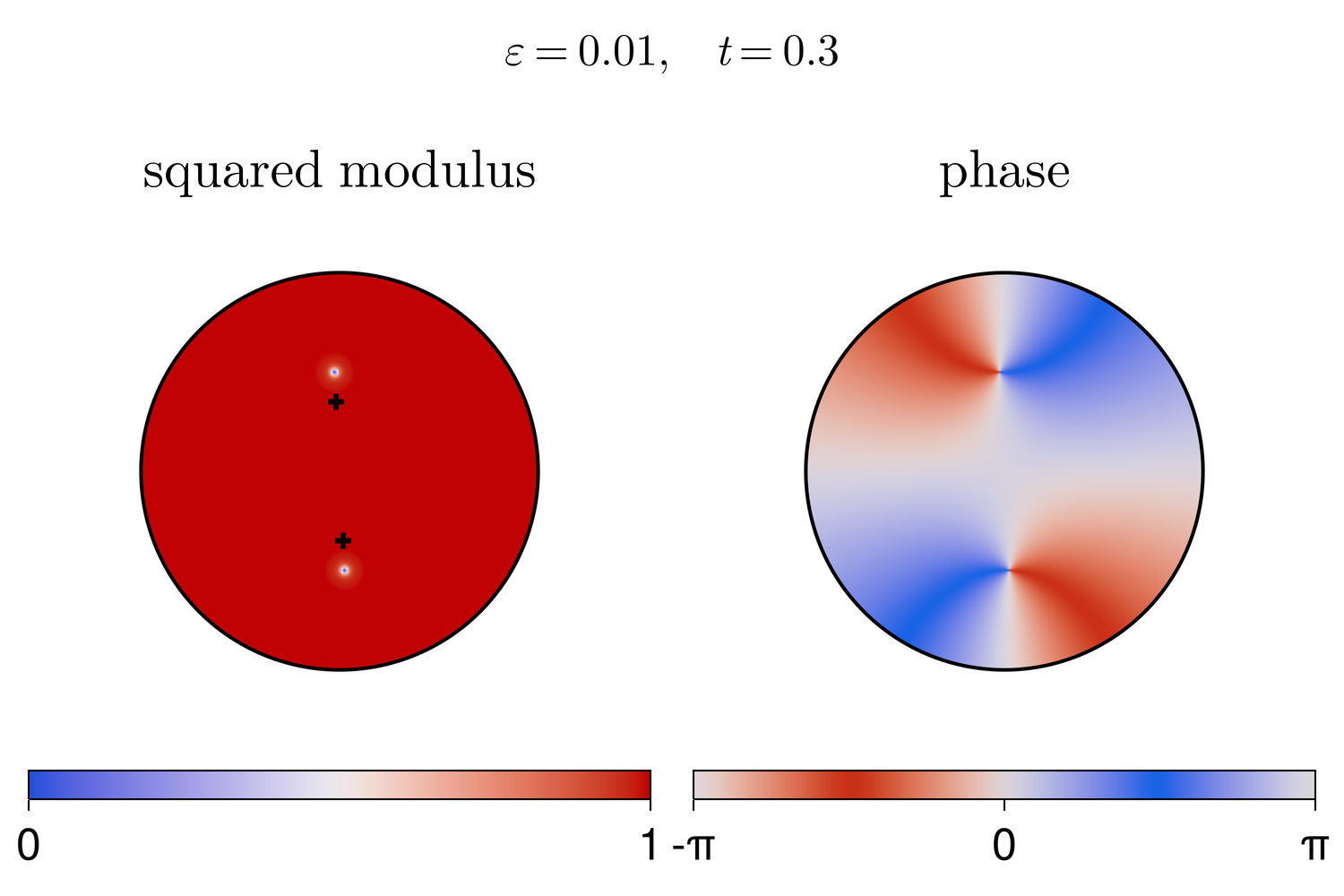}\hfill
  \includegraphics[width=0.48\linewidth]{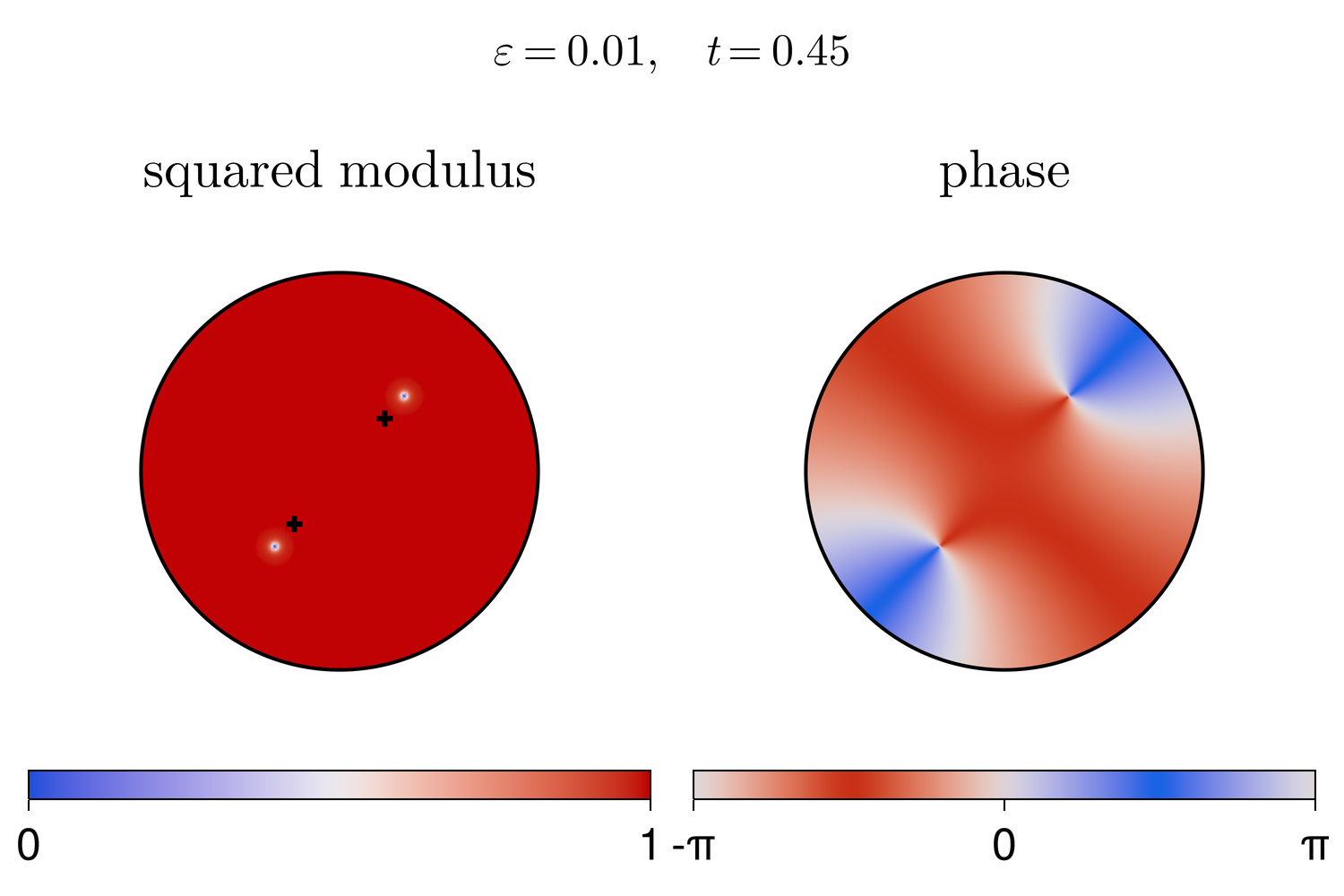}\\
  \caption{\textbf{Case 1}: Squared modulus and phase of
    $\psi_{\varepsilon}^*(t)$ for different times $t$.}
  \label{fig:smooth_case1}
\end{figure}

\begin{figure}[p!]
  \includegraphics[width=0.48\linewidth]{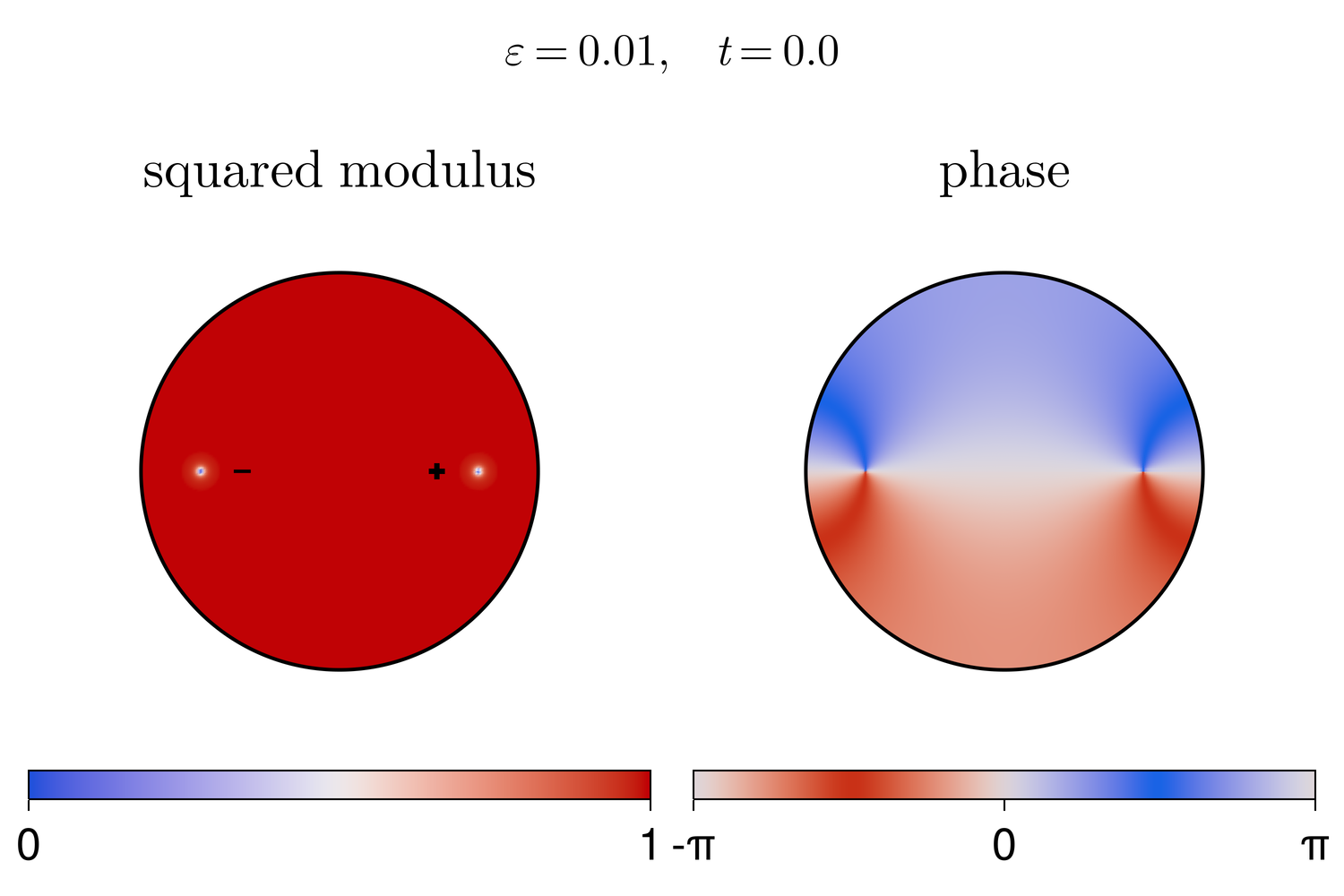}\hfill
  \includegraphics[width=0.48\linewidth]{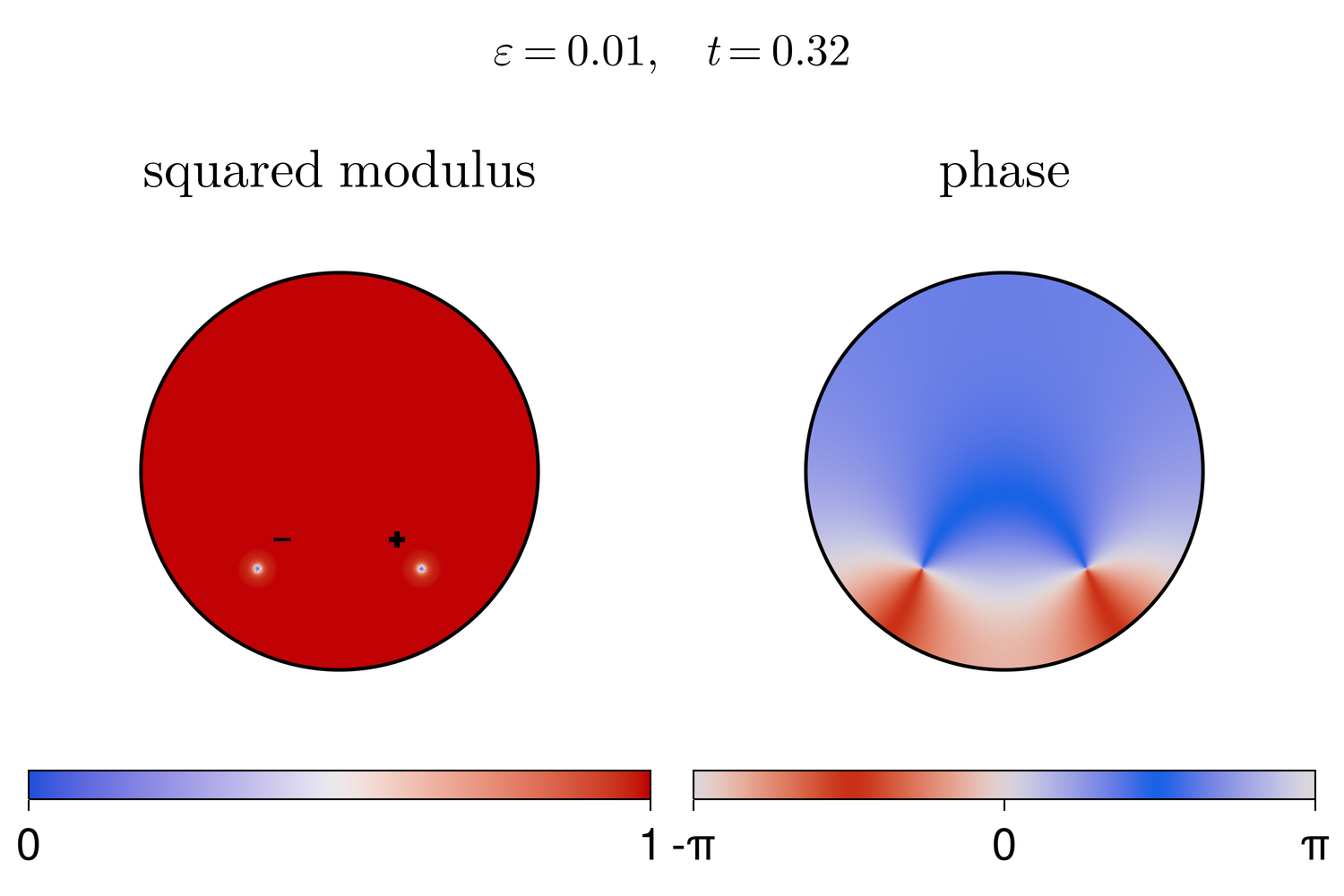}\\
  \includegraphics[width=0.48\linewidth]{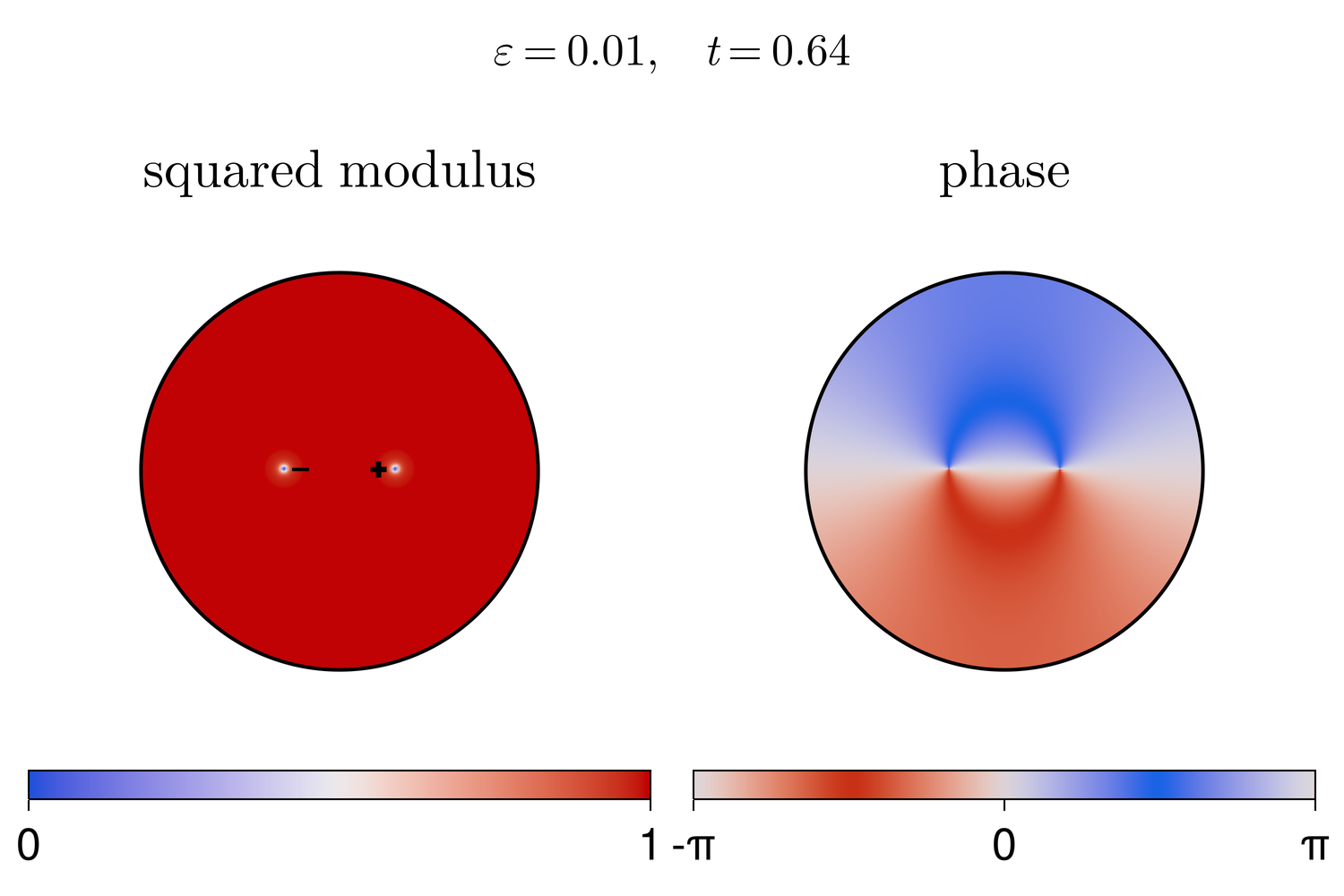}\hfill
  \includegraphics[width=0.48\linewidth]{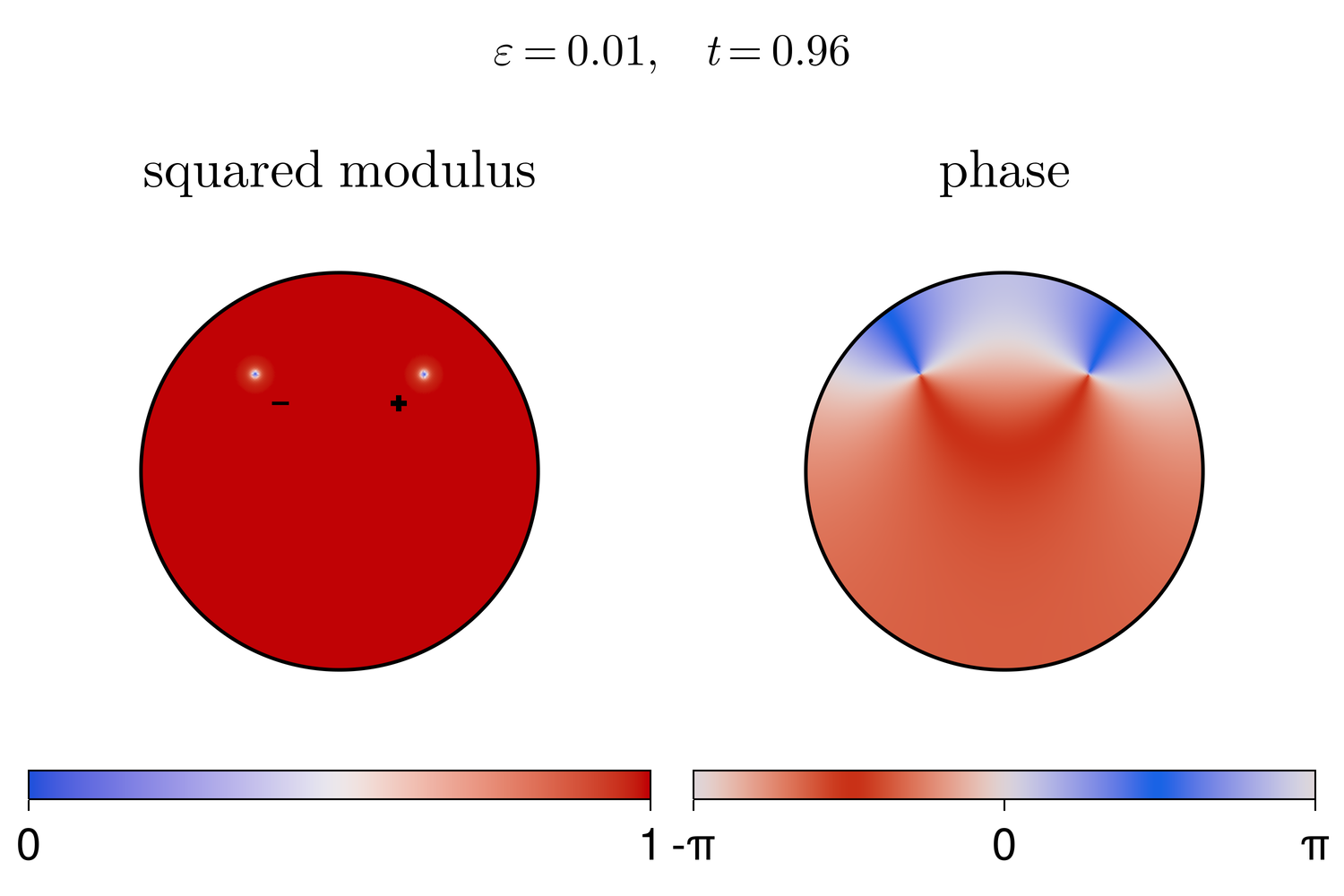}\\
  \caption{\textbf{Case 2}: Squared modulus and phase of
    $\psi_{\varepsilon}^*(t)$ for different times $t$.}
  \label{fig:smooth_case2}
\end{figure}

\begin{figure}[p!]
  \includegraphics[width=0.48\linewidth]{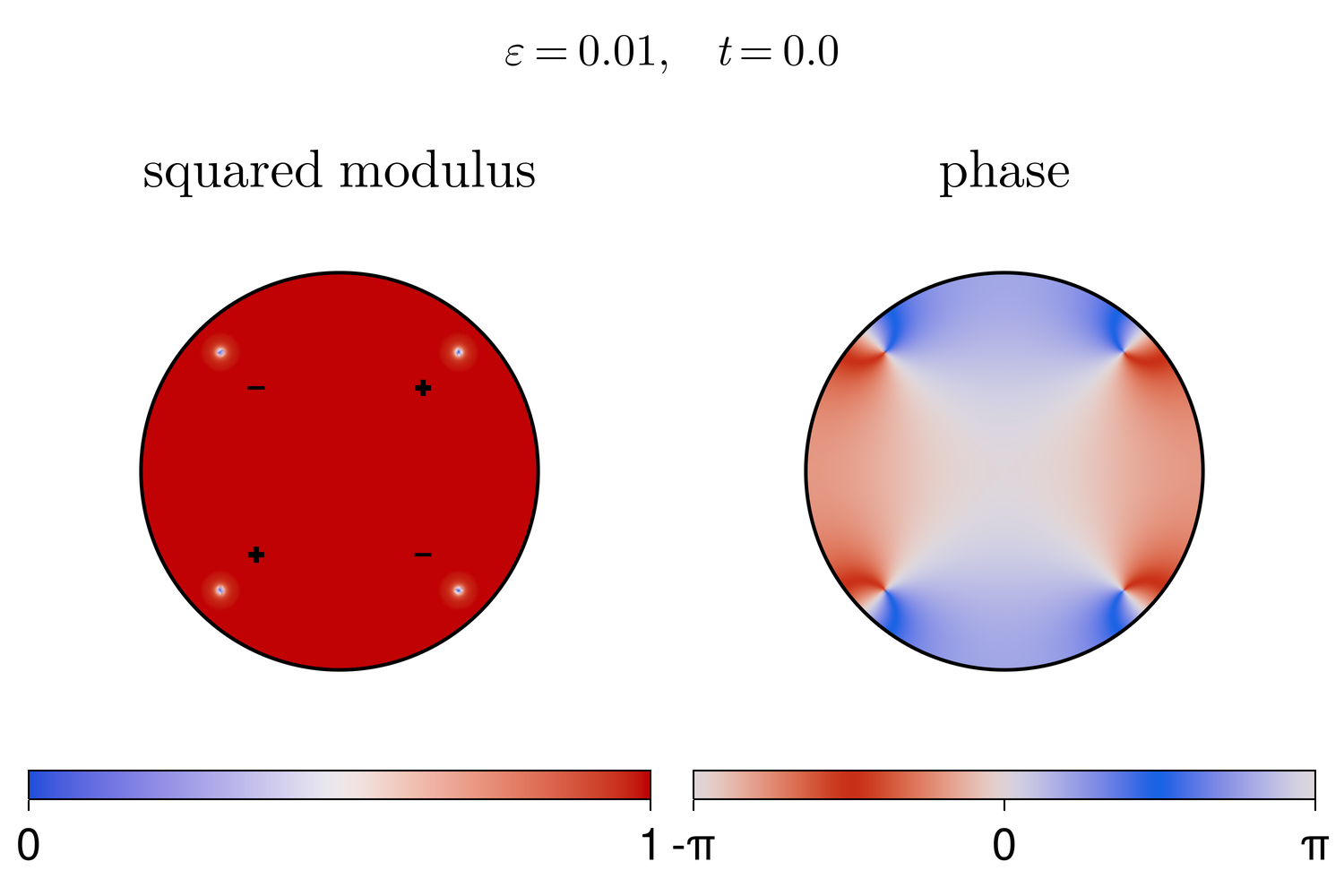}\hfill
  \includegraphics[width=0.48\linewidth]{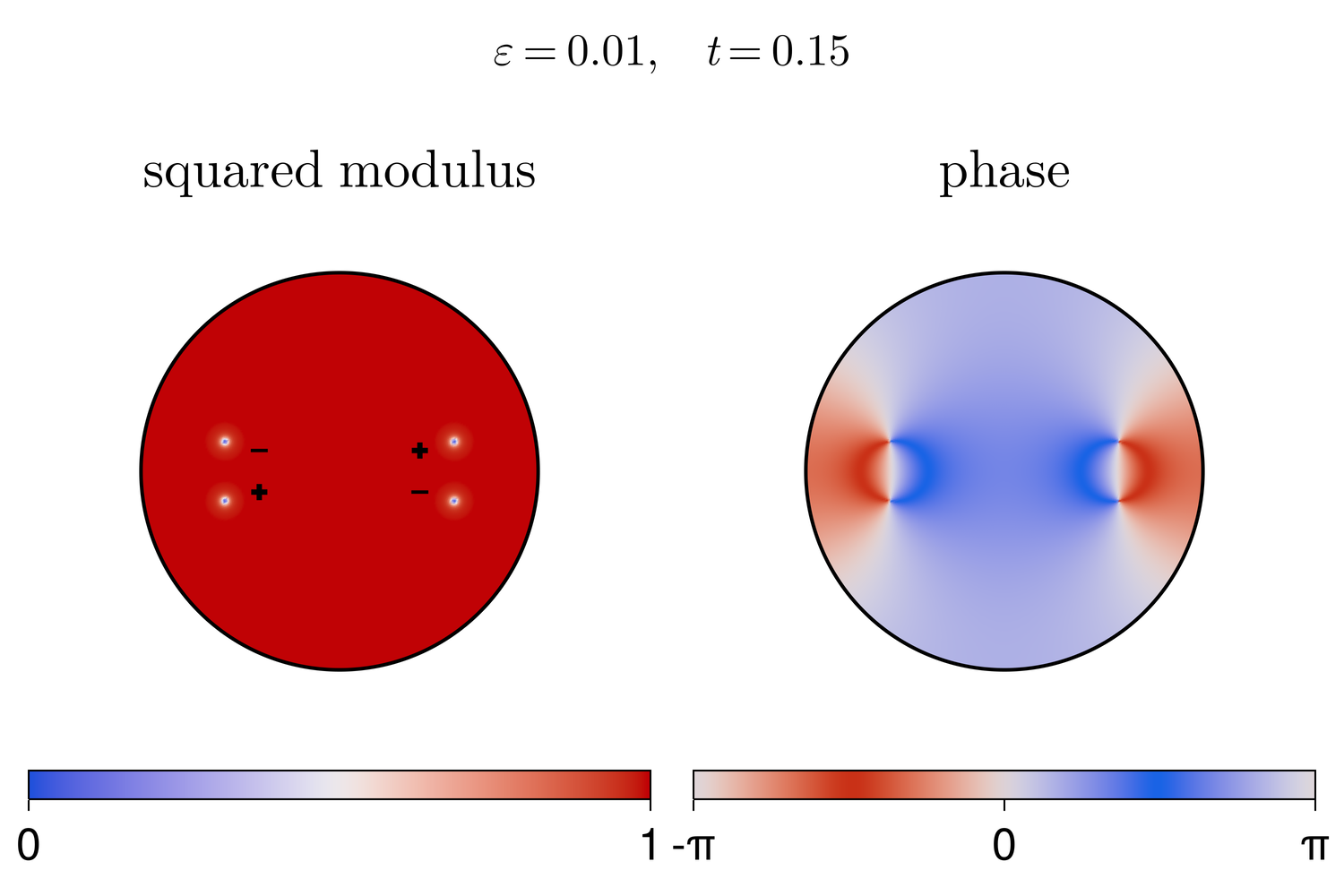}\\
  \includegraphics[width=0.48\linewidth]{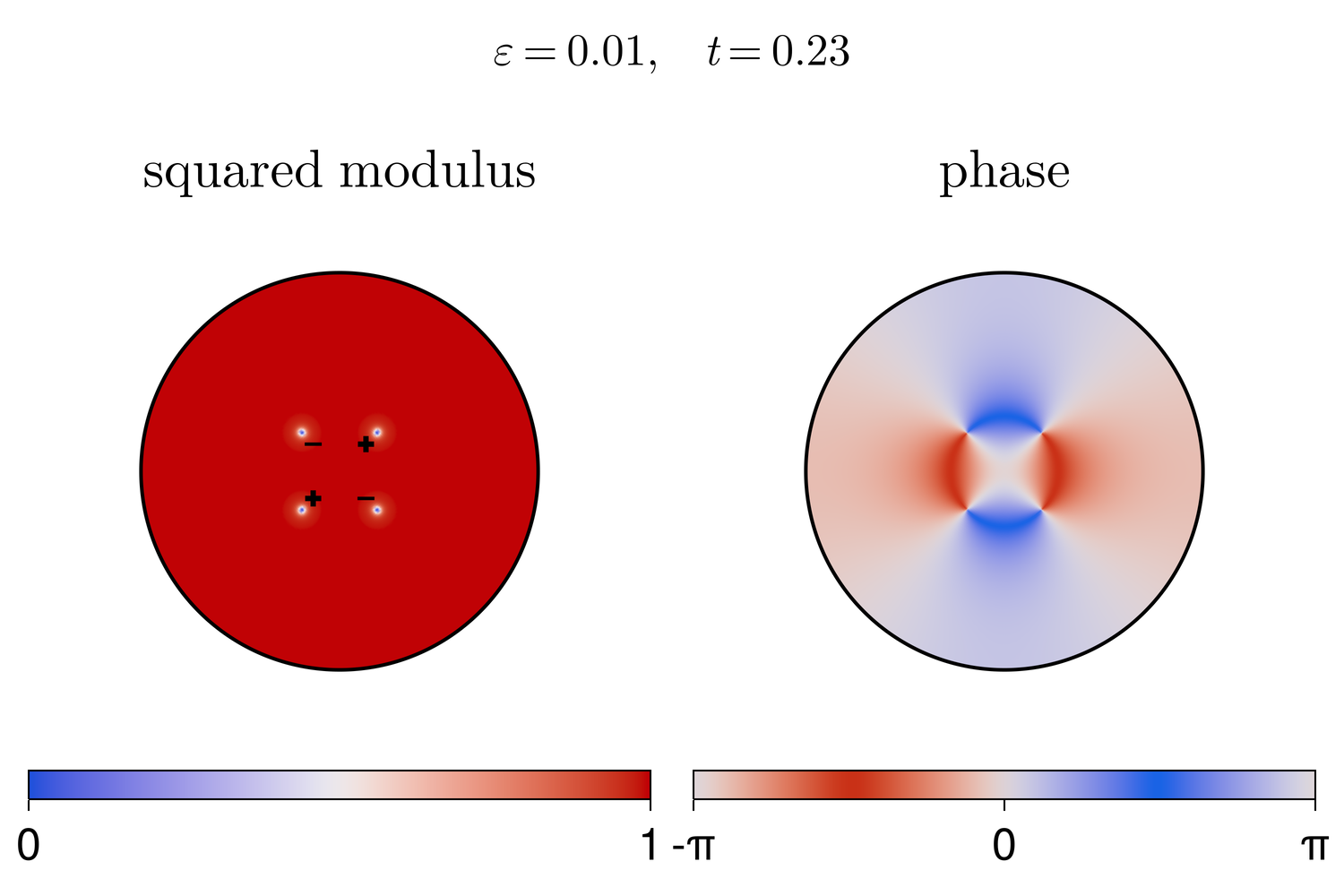}\hfill
  \includegraphics[width=0.48\linewidth]{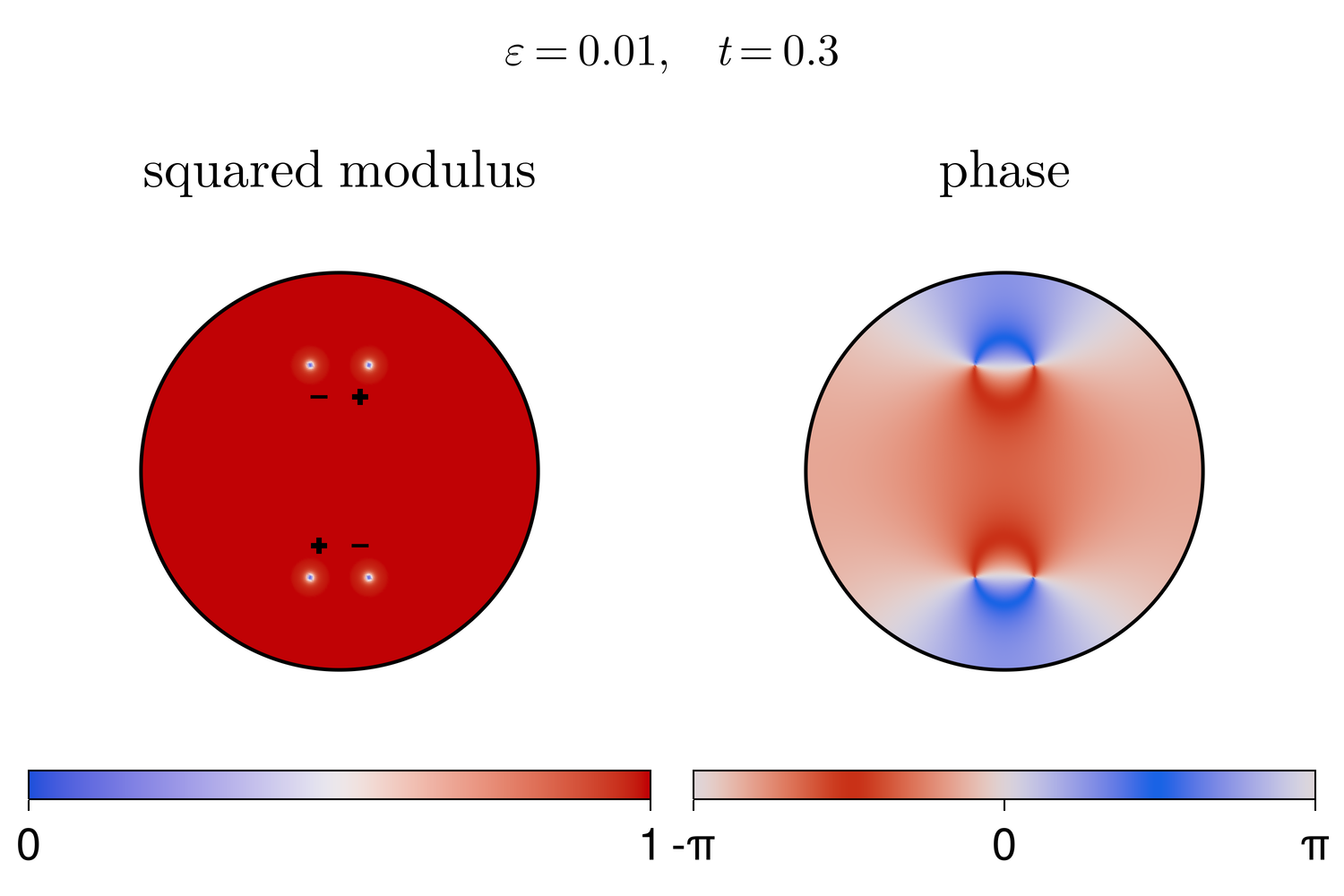}\\
  \caption{\textbf{Case 3}: Squared modulus and phase of
    $\psi_{\varepsilon}^*(t)$ for different times $t$.}
\end{figure}

\begin{figure}[p!]
  \includegraphics[width=0.48\linewidth]{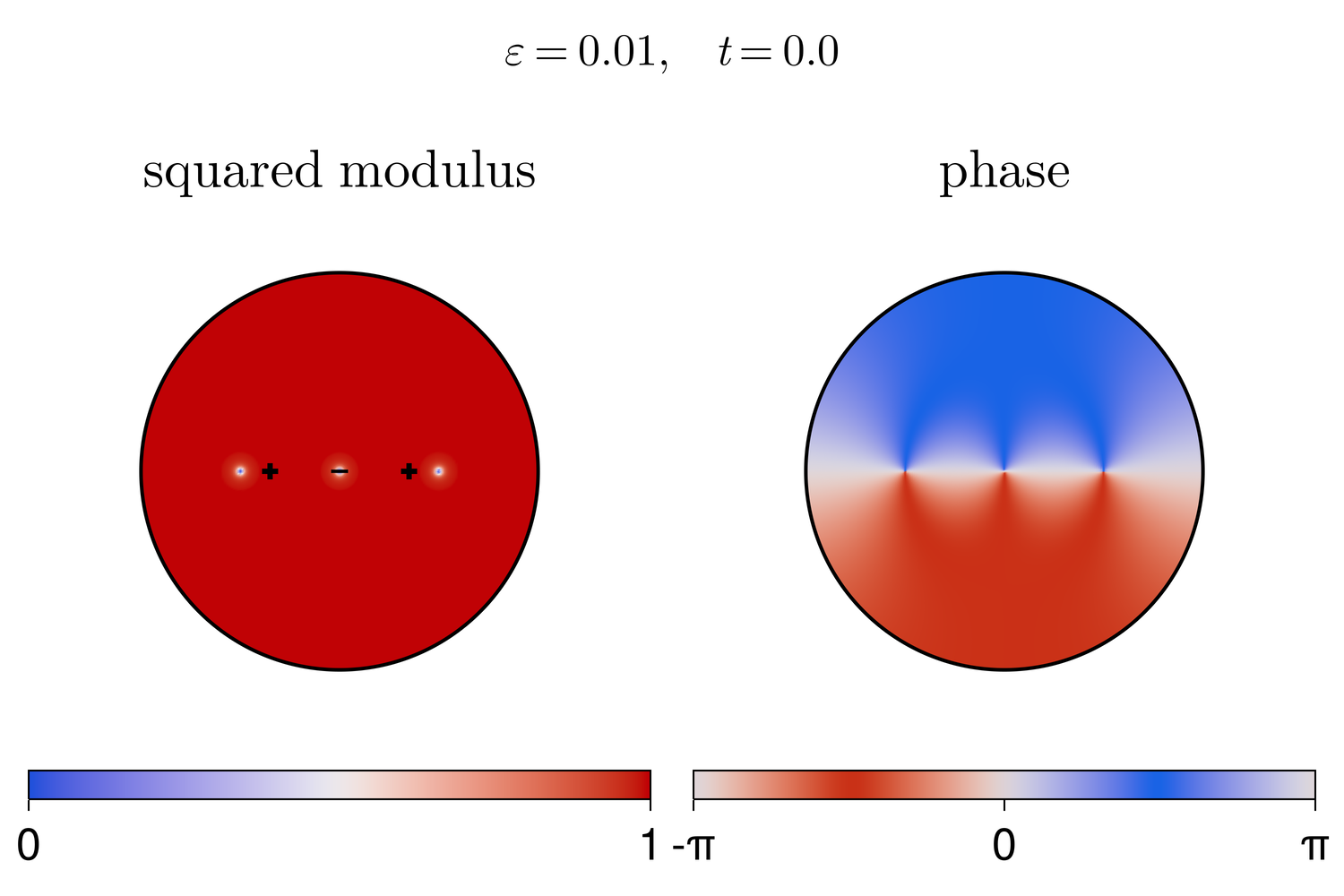}\hfill
  \includegraphics[width=0.48\linewidth]{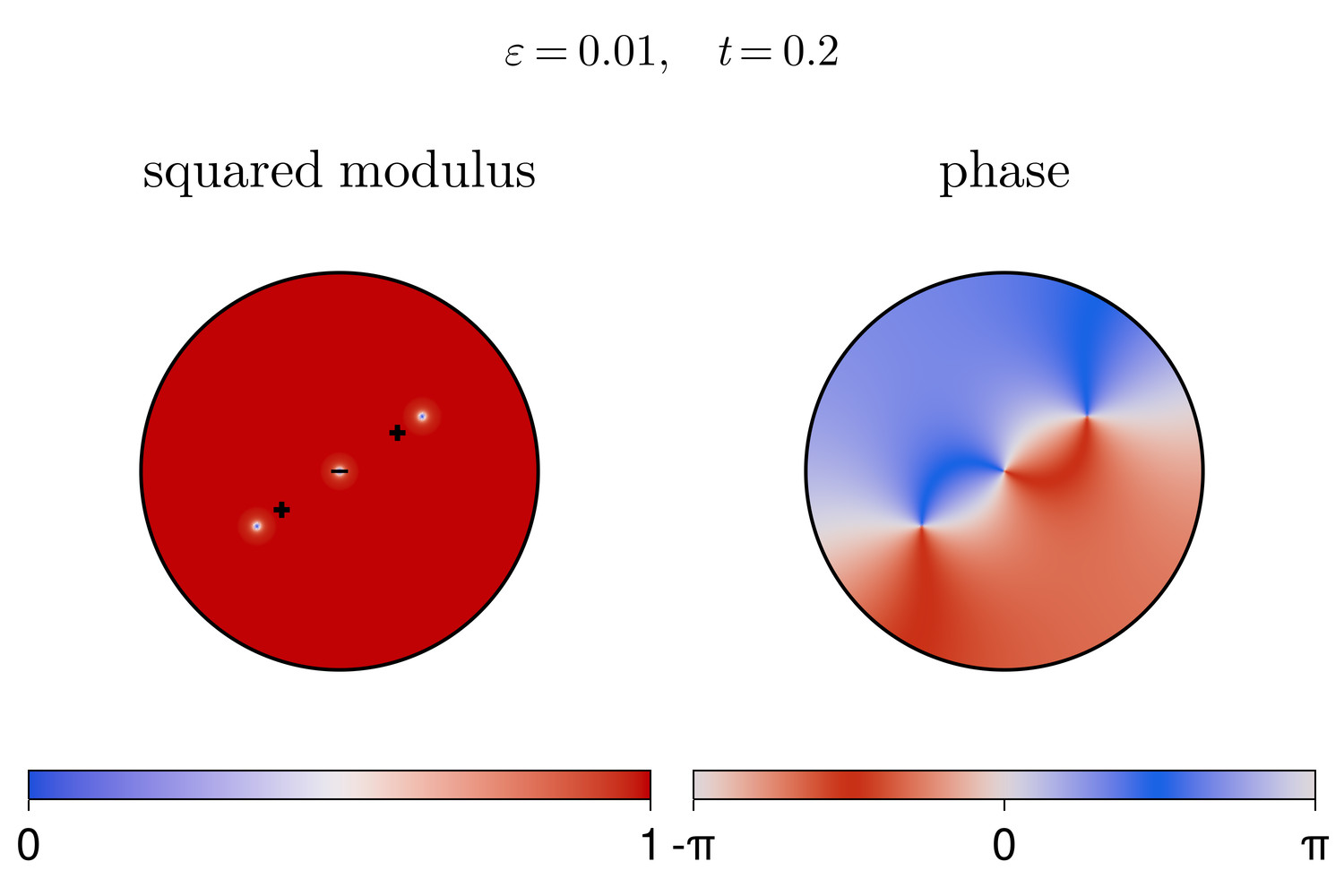}\\
  \includegraphics[width=0.48\linewidth]{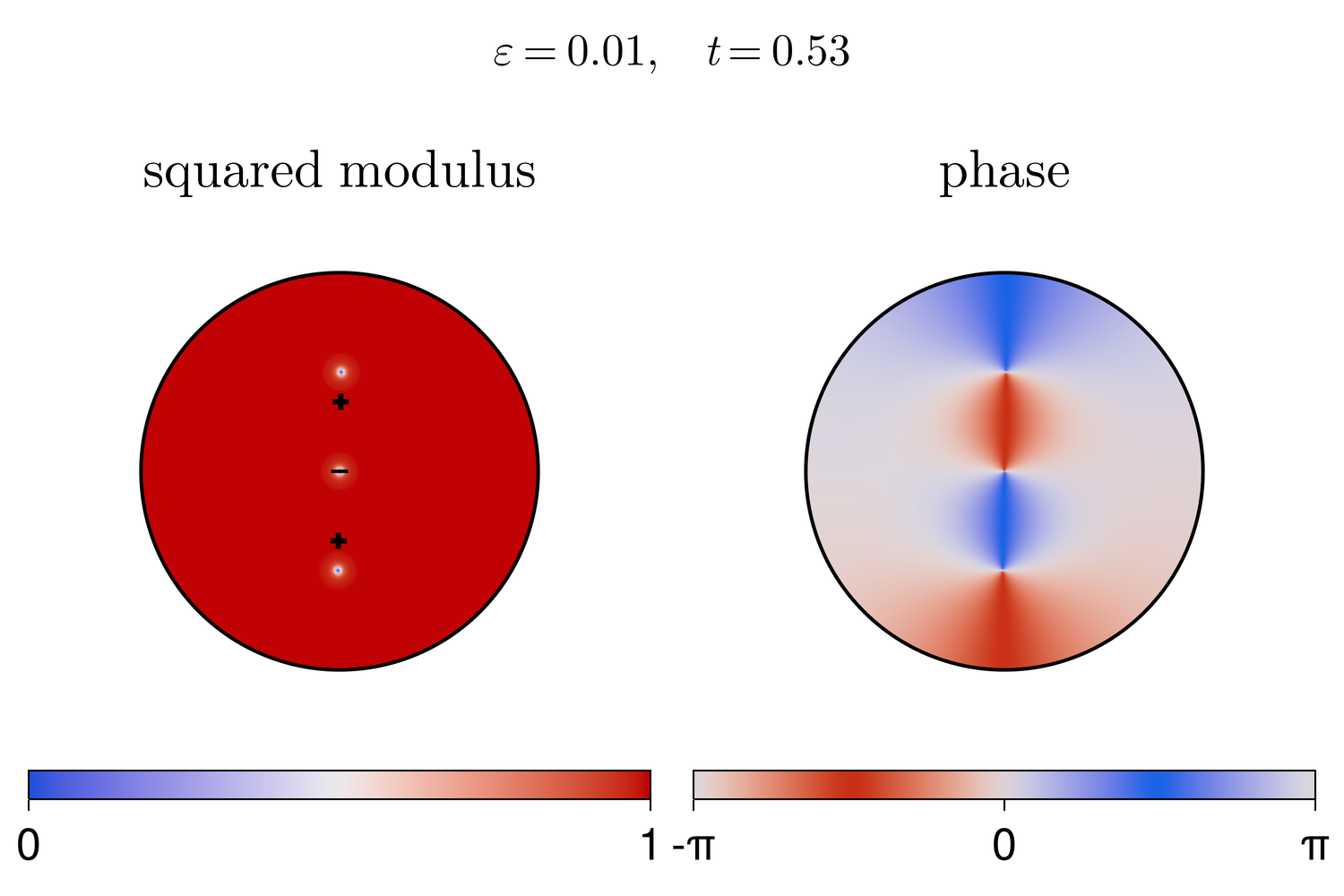}\hfill
  \includegraphics[width=0.48\linewidth]{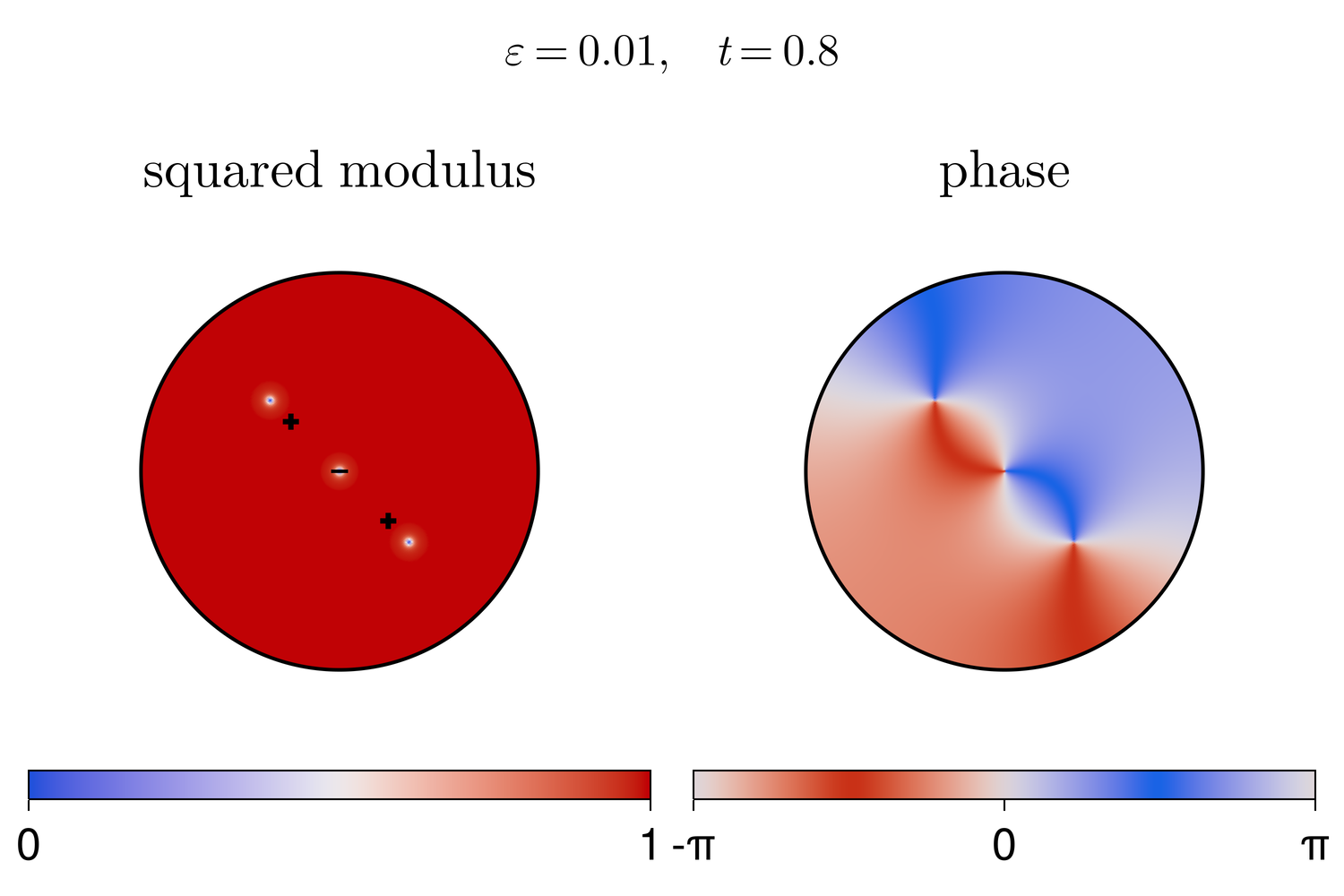}\\
  \caption{\textbf{Case 6}: Squared modulus and phase of
    $\psi_{\varepsilon}^*(t)$ for different times $t$.}
  \label{fig:smooth_case6}
\end{figure}

\section{Error control on the supercurrents}\label{sec:error}

We focus in this section on providing error estimates on the supercurrents in
the $L^{\frac43}$ norm from Theorem~\ref{thm:JS}. Precisely, we bound the error
between the supercurrents of the exact solution of \eqref{eq:GPE} and
the approximated solution obtained with our method in terms of the parameter
$\epsilon$ and the two discretization parameters, \ie the time step $\delta t$ and
the number of harmonic polynomials $n$.

To this end, let us denote by $\{ \bm{b}(t)\}_{ t \in \delta t \N}$ the
approximated trajectory for the Hamiltonian dynamics~\eqref{eq:Effective-equation2}
obtained via a RK4 ODE solver with time step $\delta
t$ and $n$ harmonic polynomials in the numerical
resolution of the PDE \eqref{eq:R} at each time step. More precisely, $\{\bm{b}(t)\}_{t \in \delta t \N}$ is the numerical solution of
\begin{align*}
  \begin{cases} \dot{\bm{b}}(t) = F^n(\bm{b}(t)), \\
    \bm{b}(0) = \bm{b}^0,
  \end{cases}
\end{align*}
where the approximate forcing term $F^n  = (F^n_1,...,F^n_N) : \Omega^N \rightarrow \R$ we evaluate at each time step is given by
\begin{align*}
  F_j^n(\bm{b}) = 2\J\biggr( \nabla_{x} R_n(x;\bm{b},d)\bigr\rvert_{x = b_j} +
  \sum_{k\neq j}^N d_k \frac{b_j-b_k}{|b_j-b_k|}\biggr) \quad \quad \mbox{for
    some } \{d_j\}_{j=1}^N \in \{\pm 1\}^N,
\end{align*}
with $R_n(\cdot;\bm{b},d): \Omega \rightarrow \R$ being the solution of
\begin{align*}
  \begin{cases}
    \Delta_x R_n(x;\bm{b},d) = 0,\quad\text{in }\Omega,\\
    \displaystyle R_n(x;\bm{b},d) =  - \P_n\biggr(\sum_{j=1}^N d_j\log|x -
    b_j|\biggr), \quad\text{on }\partial\Omega,
  \end{cases}
\end{align*}
where $\P_n$ is the $L^2(\partial \Omega)$-orthogonal projection on the bases of harmonic polynomials up to order $n$ on the boundary. Then using harmonic polynomials up to degree $n$ to approximate the
function $H$ in \eqref{eq:H_rebuilt}, the
reconstructed wave-function at time $t \in \delta t\N$ is given by
\begin{align}
  \psi_\epsilon^*(t) = \psi_\epsilon^*(\bm{b}(t),d) = \exp\prt{\i H_n}
  \prod_{j=1}^N f_{\varepsilon}(|\cdot - b_j(t)|)
  \prt{\frac{\cdot - b_j(t)}{|\cdot - b_j(t)|}}^{d_j}, \label{eq:reconstructedwave}
\end{align}
where $b_j(t) \in \Omega$ denotes the $j^{\rm th}$ component of $\bm{b}(t)$, and $H_n$ is the unique zero-mean harmonic function satisfying
\begin{equation}\label{eq:H_rebuilt_n}
  \begin{cases}
    \Delta H_n = 0 \quad\text{in } \Omega,\\ \displaystyle
    \partial_\nu H_n(x) = \P_n \biggr(\sum_{j=1}^N d_j \partial_\tau \ln|x-b_j(t)| \biggr),
    \quad\text{on } \partial\Omega.
  \end{cases}
\end{equation}
Let us also denote by $\{\bm{a}(t)\}_{t \geq 0}$ the exact solution of the
Hamiltonian dynamics~\eqref{eq:Effective-equation2}. Then the first result of this section
is an estimate on the distance $|\bm{a}(t) - \bm{b}(t)|$ in terms of the
discretization parameters $\delta t$ and $n$.

\begin{lemma}[Error estimate on Hamiltonian dynamics]\label{lem:ODEerror} Let
  $\Omega \subset \R^2$ be the unit disk, let $\{\bm{a}(t)\}_{t \geq 0}$ denote
  the exact solution to the Hamiltonian dynamics~\eqref{eq:Effective-equation2} with
  initial condition $\bm{a}^0 =(a_j^0)_{j=1,\dots,N}$ and degrees
  $(d_j)_{j=1,\dots,N}$. Let $\bm{b}(t)$ be the approximated ODE trajectory with
  the same degrees and initial position $\bm{b}^0 = (b_j^0)_{j=1,\dots N}$, with
  harmonic polynomials of degree up to $n$ and time step $\delta t$. Suppose
  that, for some $T>0$,
  \begin{align*}
    \rho_T = \min_{t \leq T} \min_{j \neq k} \{|b_j(t) - b_k(t)|, \mathrm{dist}(b_k(t),
    \partial \Omega) , |a_j(t)-a_k(t)|, \mathrm{dist}(a_j(t),\partial \Omega) \} > 0.
  \end{align*}
  is such that $0 < \rho_T < 1$. Then we have
  \begin{align}
    |\bm{a}(t)-\bm{b}(t)| \lesssim_{\ell,T} \biggr(|\bm{a}^0-\bm{b}^0| + (1-\rho_T)^{n} n^{1 - \ell} \biggr)  + \delta t^4 \quad \mbox{for any $t\in \{k \delta t \}_{k\in \N}$ with $t\leq T$} \label{eq:ODEest}
  \end{align}
  for any $\ell \in \N$, where the implicit constant depends on $\ell \in \N$
  and $T>0$ but is independent of $\bm{a}^0,\bm{b}^0, n, \delta t$ and $t$.
\end{lemma}
\begin{proof} The proof is immediate from the triangle inequality and Lemmas~\ref{lem:dyn_err} and Lemma~\ref{lem:disc_error} in the appendix.
\end{proof}

Combining the above lemma with the previous results, we can now estimate the
error between the supercurrents of the reconstructed solution
in~\eqref{eq:reconstructedwave} and the exact solution of the Gross--Pitaevskii
equation~\eqref{eq:GPE}. Note that the dependency in $\varepsilon$ is linear and
does no longer restrict the choice of the discretization parameters:
the smaller $\varepsilon$, the more accurate the estimate. The limiting
parameter is now the inter-vortices distances $\rho_T$, which appears with
an exponential dependency in the implicit constants of the following Theorem.
\begin{theorem}[Error estimate on supercurrents] \label{thm:errorestimate}
  Let $\Omega \subset \R^2$ be the unit disk, and
  $\{\psi_\epsilon(t)\}_{\epsilon < \epsilon_0}$ be the solution of
  \eqref{eq:GPE} with well-prepared initial conditions, in the sense of
  Definition~\ref{def:WP} for some initial vortices with positions $\bm{a}^0 =
  (a_j^0)_{j=1,\dots,N}$ and degrees $(d_j)_{j=1,\dots,N}$. Let $\bm{b}(t)$ be
  the approximated ODE trajectory with the same degrees and initial position
  $\bm{b}^0 = (b_j^0)_{j=1,\dots N}$, with harmonic polynomials of degree up to
  $n$ and time step $\delta t$.
  Let $\psi^*_\epsilon(t)$ be the reconstructed wave function described above.
  Suppose that, for some $T >0$,
  \begin{align*}
    \rho_T = \min_{t \leq T} \min_{j \neq k} \{|b_j(t) - b_k(t)|, \mathrm{dist}(b_k(t),
    \partial \Omega) , |a_j(t)-a_k(t)|, \mathrm{dist}(a_j(t),\partial \Omega) \} > 0.
  \end{align*}
  is such that $0 < \rho_T < 1$. Then we have
  \begin{align}
    \norm{j\bigr(\psi^*_\epsilon(t)\bigr) -
      j\bigr(\psi_\epsilon(t)\bigr)}_{L^{\frac43}(\Omega)} \lesssim_{\ell,T}
    \epsilon^\gamma + n^{-\ell+\frac12} +
    \sqrt{|\bm{a}^0-\bm{b}^0| + (1-\rho_T)^{n} n^{1 -
        \ell}  + \delta t^4},
    \label{eq:supercurrenterror}
  \end{align}
  for any $t\in [0,T]$ and $\ell \in \N$,
  where $0<\gamma \leq \frac12$ and the implicit constant depends on $\ell$ and $T>0$ but is
  independent of $\bm{a}^0$,  $\bm{b}^0$, $\epsilon, n$, and $\delta t$.
\end{theorem}

\begin{proof} Throughout this proof we use various classical estimates, whose proofs are compiled in the appendix for the sake of presentation. At any time we can split the error into
  \begin{equation}\label{eq:triang}
    \begin{split}
      \norm{j(\psi_\varepsilon(t)) - j(\psi_\varepsilon^*(t))}_{L^{\frac43}} &\leq
      \norm{j(\psi_\varepsilon(t)) - j(u^*(\bm a(t),d))}_{L^{\frac43}} \\ &+
      \norm{j\bigr(u^*(\bm a(t), d)\bigr) - j\bigr(u^*(\bm b(t),d)\bigr)}_{L^{\frac43}}
      \\ &+ \norm{j\bigr(u^*(\bm b(t), d)\bigr) -
        j(\psi_\varepsilon^*(t))}_{L^{\frac43}},
    \end{split}
  \end{equation}
  where $\bm a(t)$ denotes the exact Hamiltonian trajectory of the vortices
  given by the solution of eq.~\eqref{eq:Effective-equation}. By Theorem~\ref{thm:JS}, the
  first term is
  bounded by $\lesssim \varepsilon^\gamma$ for some $0<\gamma<1$, which yields the first term in \eqref{eq:supercurrenterror}.

  The second term can be bounded using \eqref{eq:j_G} and combining estimates~\eqref{eq:loginteriorest} and \eqref{eq:error3est} in the appendix to obtain
  \begin{align*}
    \norm{j\bigr(u^*(\bm a(t), d)\bigr) - j\bigr(u^*(\bm b(t),d)\bigr)}_{L^{\frac43}}  &=
    \norm{\nabla G(x;\bm a(t),d) - \nabla G(x;\bm b(t),d)}_{L^{\frac43}}
    \\
    &\lesssim_{\rho_T} {|\bm a(t) - \bm b (t)|} + {|\bm a(t) - \bm b(t)|}^{\frac12},
  \end{align*}
  where the constant depends on $\rho_T$. Thus, by estimate~\eqref{eq:gronwalestdisk} in Lemma~\ref{lem:dyn_err} and Lemma~\ref{lem:disc_error} we find
  \begin{align}
    \norm{j\bigr(u^*(\bm a(t), d)\bigr) - j\bigr(u^*(\bm
      b(t),d)\bigr)}_{L^{\frac43}} &\lesssim_{\ell,T}
    \sqrt{\biggr(|\bm{a^0}-\bm{b^0}| + \frac{(1-\rho_T)^{n} n^{1 -
          \ell}}{\rho_T^{1+\ell}} T\biggr) e^{Ct/\rho_T^{5/2}} + \delta t^4}
    \\
    &\lesssim_{\ell,T} \sqrt{|\bm{a^0}-\bm{b^0}| + (1-\rho_T)^{n} n^{1 - \ell}+ \delta t^4}.
    \label{eq:err_j_canon}
  \end{align}
  where the implicit constant depends on $\ell \in \N$ and $T>0$ but is
  independent of $\bm{a}^0, \bm{b}^0, \delta t$ and $n$. This yields the last term
  in \eqref{eq:supercurrenterror}.

  To bound the third term, notice that the reconstructed wave function
  $\psi_\varepsilon^*$ is obtained after an additional approximation in the
  resolution of Laplace's equation \eqref{eq:H_rebuilt} that defines the phase factor
  $H$. Therefore, by the triangle inequality we find
  \begin{align*}
    \norm{j\bigr(u^*(\bm b(t), d)\bigr) -
      j(\psi_\varepsilon^*(t))}_{L^{\frac43}}
    &\leq \norm{j\bigr(u^*(\bm b(t), d)\bigr) -j\bigr(u_n^*(\bm b(t), d)\bigr) }_{L^{\frac43}} + \norm{j\bigr(u_n^*(\bm b(t), d)\bigr) - j\bigr(\psi_\varepsilon^*(t)\bigr)}_{L^{\frac43}},
  \end{align*}
  where $u^*_n(\bm b(t),d)$ is the approximated canonical harmonic map
  \begin{align*}
    u_n^*(\bm b(t), d) = \exp\prt{\i H_n} \prod_{j=1}^N \prt{\frac{\cdot -
        b_j(t)}{|\cdot - b_j(t)|}}^{d_j} \quad \mbox{with $H_n$ defined in \eqref{eq:H_rebuilt_n}.}
  \end{align*}
  Since $u_n^*(\bm{b}(t),d)$ and $u^*(\bm{b}(t), d)$ have the same
  singularities, from H\"older's inequality we have
  \begin{align*}
    \norm{j\bigr(u^*(\bm b(t), d)\bigr) -j\bigr(u_n^*(\bm b(t), d)\bigr)
    }_{L^{\frac43}} \lesssim \norm{\nabla H - \nabla H_n}_{L^2}.
  \end{align*}
  {We now note that the projection $\P_n$ on the first $2n+1$
    Fourier modes commutes with the tangential derivative since, using that the
    tangential and angular derivatives coincides when $\Omega$ is the unit disk,
    we have
    \begin{align}
      \partial_\theta (\P_n g)(e^{\i\theta}) &= \partial_\theta
      \biggr(\sum_{|k|\leq n} e^{\i k \theta} \frac1{2\pi}\int_0^{2\pi} e^{-\i k x} g(e^{\i x}) \mathrm{d}x\biggr) \nonumber \\
      &= \sum_{|k|\leq n} e^{\i k \theta} \frac1{2\pi}\int_0^{2\pi} \i k e^{-\i kx} g(e^{\i x}) \mathrm{d} x\nonumber \\
      &= \sum_{|k|\leq n} e^{\i k \theta} \frac1{2\pi}\int_0^{2\pi} e^{-\i kx} \partial_x g(e^{\i x}) \mathrm{d} x = \P_n(\partial_\theta g)(e^{i\theta}), \quad
      \mbox{for $g \in C^1(\partial \Omega),$} \label{eq:commuting}
    \end{align}
    where we used integration by parts in the second line. As a consequence, not only $\nabla H = -\nabla\times R$ for the function $R$ defined in eq.~\eqref{eq:R} (with $a_j$ replaced by $b_j(t)$), but also $\nabla H_n = -\nabla\times R_n$, where $R_n$ solves
    \begin{align*}
      \begin{dcases} \Delta R_n = 0 \quad &\mbox{ in $\Omega$,} \\
        R_n = - \P_n\biggr(\sum_{j=1}^N d_j \log |x-b_j(t)|\biggr), &\quad \mbox{on $\partial \Omega$.} \end{dcases}
    \end{align*}
    Indeed, from straightforward calculations and the commutation property in
    \eqref{eq:commuting}, one can verify that $\widetilde{j} = \nabla H_n + \nabla
    \times R_n$ solves the equation
    \begin{align*}
      \nabla \times \widetilde{j} = 0 , \quad \nabla \cdot \widetilde{j} = 0, \quad \widetilde{j} \cdot \nu= 0,
    \end{align*}
    which implies $\nabla H_n = - \nabla \times R_n$. So from estimates~\eqref{eq:classicaldiskest2} and~\eqref{eq:logspectralconv} in the appendix, we obtain
    \begin{align*}
      \norm{\nabla H - \nabla H_n}_{L^2} = \norm{\nabla R - \nabla R_n}_{L^2}
      \lesssim_\ell \frac{n^{-\ell + \frac12}}{\rho_T^{\ell-\frac12}} \lesssim_{\ell,T} n^{-\ell+\frac12}, \quad \mbox{for $\ell \in \N$.}
    \end{align*}}
  Finally, we can repeat the same steps in the proof of \eqref{eq:jest} to obtain
  \begin{align*}
    \norm{j\bigr(u_n^*(\bm b(t), d)\bigr) - j\bigr(\psi_\varepsilon^*(t)\bigr)}_{L^{\frac43}}\lesssim \epsilon^{\frac12},
  \end{align*}
  which can be absorbed by the term $\epsilon^{\gamma}$ in \eqref{eq:supercurrenterror} and completes the proof.
\end{proof}

\section{Numerical study of the approximation error}

In this section, we analyze the efficiency and the convergence of the numerical method presented in
this paper. First, we focus on the numerical resolution
of the reduced dynamical law \eqref{eq:Effective-equation2} where we illustrate the
convergence of the vortex trajectories. Next, we study the numerical error on
the supercurrents in order to illustrate Theorem~\ref{thm:errorestimate}.

\subsection{Numerical convergence of the Hamiltonian dynamics}

We study here the convergence of the vortex trajectories obtained by
solving the Hamiltonian dynamics obtained from the singular limit
$\varepsilon\to0$. To this end, we numerically solve
\eqref{eq:Effective-equation2} for various initial conditions given by two vortices of
degree $+1$ at position $(\pm m,0)$ for $m \in \set{0.1, 0.3, 0.6, 0.9}$:
we expect for each case a circular trajectory on the circle of radius $m$
centered at $(0,0)$. We use a RK4 solver and the harmonic polynomials basis
described in Section~\ref{sec:HD_sim}.
For each of these cases, we study the convergence of the trajectories when
$\delta t\to 0$ and $n\to +\infty$: according to Lemma~\ref{lem:ODEerror},
we respectively expect convergence of order 4 and spectral convergence.
In Figure~\ref{fig:error_HD}, we plot the error
$|\bm{a}(T) - \bm{b}(T)|$ at $T=1$ for the following setting:
\begin{itemize}
  \item Convergence of the approximate trajectory $(\bm{b}(t))_{0\leq t\leq T}$
    with respect to \[n\in\set{4, 8, 16, 32, 64}\] towards
    a reference trajectory $(\bm{a}(t))_{0\leq t\leq T}$ obtained with $n=256$ and $\delta t = 10^{-5}$.
  \item Convergence of the approximate trajectory $(\bm{b}(t))_{0\leq t\leq T}$
    with respect to \[\delta t\in\set{10^{-2},7.5\cdot10^{-3}, 5\cdot10^{-3},
        2.5\cdot10^{-3}, 1\cdot10^{-3}}\] towards a
    reference trajectory $(\bm{a}(t))_{0\leq t\leq T}$ obtained with $n=256$ and $\delta t = 10^{-5}$.
\end{itemize}
The convergence plots are compatible with a spectral rate for $n$ (even though a
transition to a polynomial regime cannot be ruled out, the log-lin plot from
Figure~\ref{fig:error_HD} suggests an exponential convergence, which is
compatible with a spectral convergence) and a $4$-th order rate for
$\delta t$. Moreover, the pre-factors are much higher for $m =
0.1$ and $m = 0.9$: in these cases, the vortices are respectively close to
each other and close to the boundary. The dependency of the pre-factors on the
quantity $\rho_T$ (highlighted in Lemmas~\ref{lem:dyn_err} and
\ref{lem:disc_error}) is thus amplified, as verified in
Figure~\ref{fig:error_HD}.

\begin{figure}[h!]
  \centering
  \includegraphics[width=0.48\linewidth]{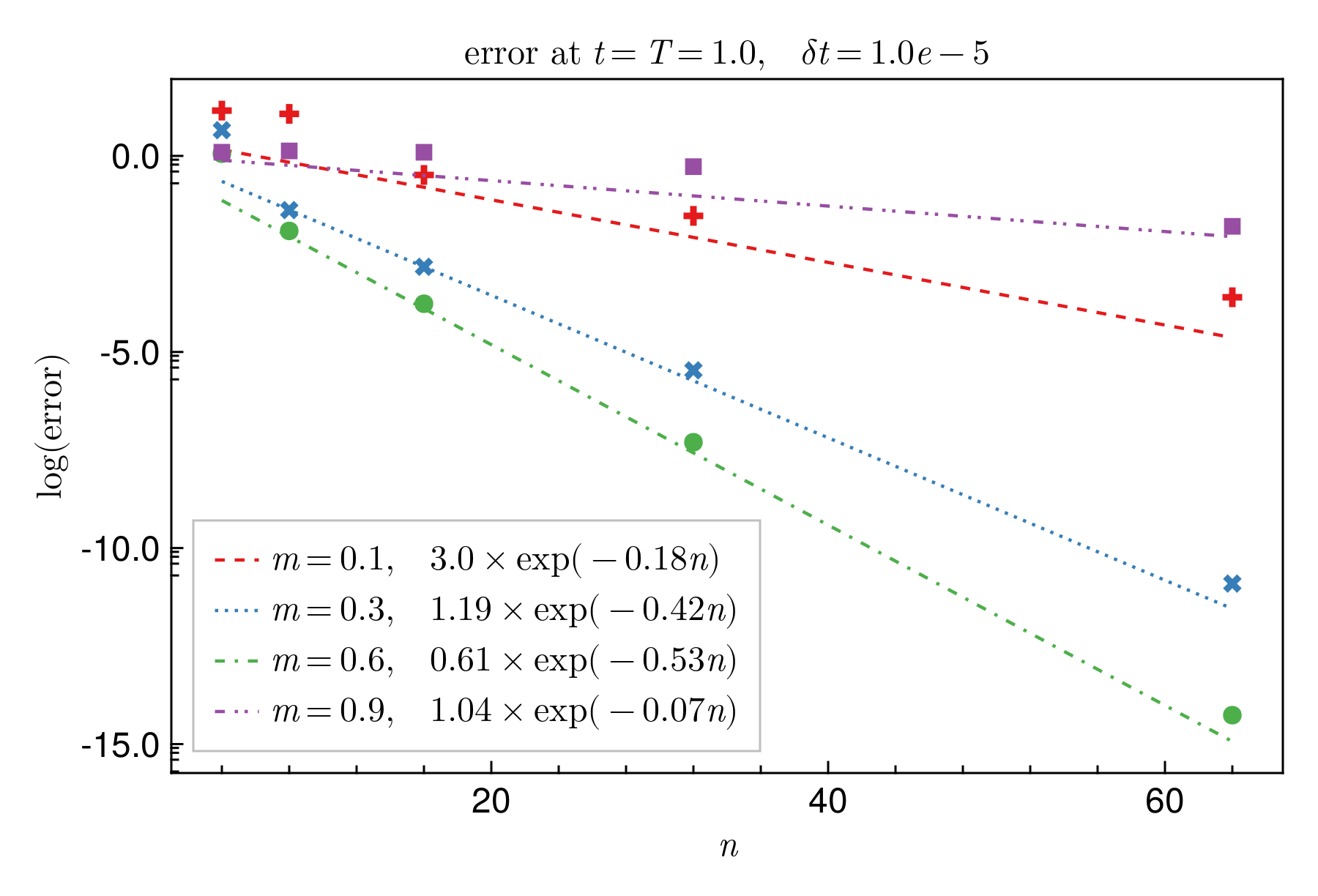}\hfill
  \includegraphics[width=0.48\linewidth]{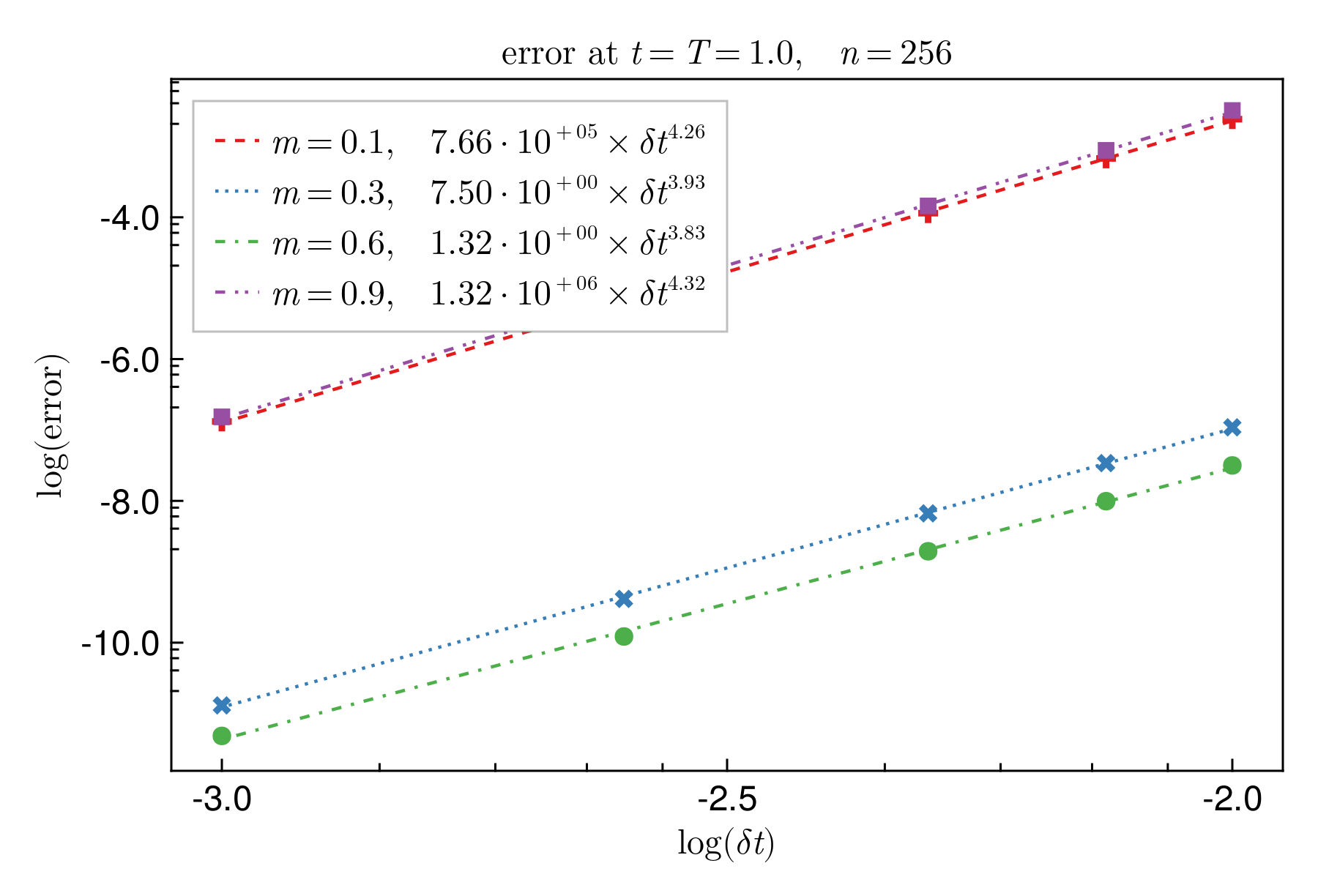}
  \caption{Convergence of the numerical resolution of the reduced dynamical law
    \eqref{eq:Effective-equation}. (Left) Convergence with respect to $n$: as expected,
    what we observe is compatible with a spectral convergence. (Right)
    Convergence with respect to $\delta t$: as expected, the convergence is
    algebraic of order about $4$.}
  \label{fig:error_HD}
\end{figure}

\subsection{Numerical convergence of the vortex-tracking method}

We focus now on Case 1: there are two vortices with identical winding number
$+1$ and initial position $(\pm0.5,0)$. From Figure~\ref{fig:trajectories}, we
know that we expect, for small values of $\varepsilon$, the vortices to move
around the circle of radius $0.5$ and centered at $(0,0)$.

\subsubsection{Computational framework}

In order to compare
the numerical method presented in Section~\ref{sec:method} to the exact solution
of the GP equation \eqref{eq:GPE}, we compute a reference solution
$\psi_\varepsilon(t)$ with the finite element solver
\texttt{Gridap.jl} \cite{badiaGridapExtensibleFinite2020} and meshing tool GMSH
\cite{geuzaineGmsh3DFinite2009}. The GP equation \eqref{eq:GPE}
is solved for initial conditions given by the vortex configuration from Case 1
and the smoothing procedure described in Lemma~\ref{lem:WP}.
We use a Strang splitting scheme, with $\P^1$ Lagrangian finite
elements for the spatial discretization. The mesh
size is of order $h=5\cdot10^{-2}$ with local \emph{a priori} refined mesh on the vortex
trajectory with mesh size $h=2\cdot10^{-3}$, see Figure~\ref{fig:mesh}.
The time step is $\delta t=10^{-5}$. The finite
element solver is then used to obtain various reference solutions
$\psi_\varepsilon$ for $\varepsilon \in \set{0.03, 0.05, 0.07, 0.1}$, up to
$T=1$. We also compute a reference solution for $\varepsilon = 0.01$, with
time step $\delta t = 10^{-6}$. Note that the simulation with such a mesh can
take several days for small values of $\varepsilon$: for instance, in our case, the simulations for $\varepsilon = 0.01$ took about a week to run with a non-optimized code on a small sized cluster.

The reconstructed approximation $\psi_{\varepsilon}^*(t)$ from
Section~\ref{sec:method} is obtained
from the trajectories in the limit $\varepsilon\to0$, with $r_0 = 0.3$ and
$n=128$ for the reconstruction step. The Hamiltonian dynamics is solved
numerically with a RK4 method, with time step $\delta t = 10^{-5}$ and $n=128$.
Note that this time, the simulation only takes a few seconds on a personal
laptop.

\begin{figure}[h!]
  \centering
  \includegraphics[width=0.5\linewidth]{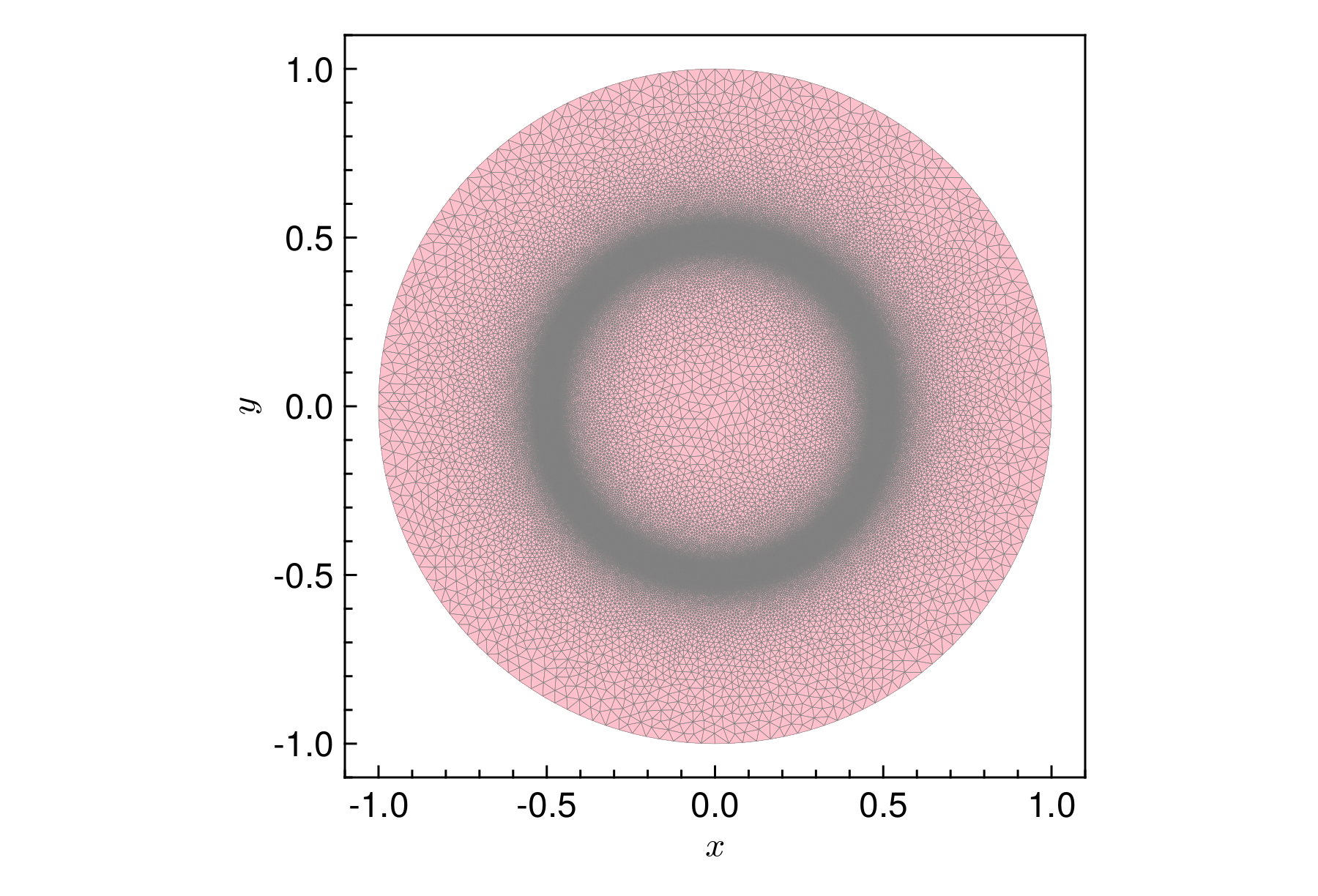}
  \caption{Mesh of the unit disk used to obtain reference solutions.}
  \label{fig:mesh}
\end{figure}

\subsubsection{Numerical illustration of the singular limit $\varepsilon\to0$}

We evaluate different quantities in order to formally compare the
reference solution $\psi_\varepsilon(t)$ and the reconstructed approximation
$\psi_\varepsilon^*(t)$. First, by tracking its isolated zeros, for instance
using one of the algorithms introduced in
\cite{dujardinNumericalStudyVortex2022,kaltIdentificationVorticesQuantum2023}, we are
able to localize the vortices of the reference solution $\psi_\varepsilon(t)$
and plot the associated trajectories in Figure~\ref{fig:traj_eps}, in the spirit of
the study realized in \cite{baoNumericalStudyQuantized2014}. We observe similar
results and, in particular, the oscillations of the vortex trajectories for
$\varepsilon>0$ around the limiting trajectory. In Figure~\ref{fig:error_eps},
we represent the distance between vortices $\norm{{\bm a}(t) -
  P_h(\psi_\varepsilon(t))}$ as $t$ goes from 0 to $T$, where $P_h$ is the
vortex localization operator on the finite element mesh. As expected, this error becomes small when
$\varepsilon\to0$. Then, we analyze the (relative) errors on the supercurrents and the
wave function (in $L^2$ and $H^1$ norms), \ie
\begin{equation*}
  \frac{\norm{j(\psi_\varepsilon(t)) - j(\psi_{\varepsilon}^*(t))}_{L^{\frac43}}}
  {\norm{j(\psi_\varepsilon(t))}_{L^{\frac43}}}, \quad
  \frac{\norm{\psi_\varepsilon(t) - \psi_{\varepsilon}^*(t)}_{L^2}}
  {\norm{\psi_\varepsilon(t)}_{L^2}} \quad\text{and}\quad
  \frac{\norm{\nabla \psi_\varepsilon(t) - \nabla \psi_{\varepsilon}^*(t)}_{L^2}}
  {\norm{\nabla \psi_\varepsilon(t)}_{L^2}}.
\end{equation*}
Again, these quantities are smaller for small values of $\varepsilon$, as expected.

Finally, we compare in Figure~\ref{fig:error_vs_canon} the relative error
in the $L^2$ norm when approximating $\psi_\varepsilon(t)$ by
$\psi_{\varepsilon}^*(t)$ or by the canonical harmonic map $u^*(\bm a(t), \bm
d)$ as well as for the supercurrents. The error achieved by the reconstructed wave
function is slightly better than the one achieved by the canonical harmonic map
with singularities located at ${\bm a}(t)$. Regarding
the supercurrents, the reconstruction procedure seems to improve significantly the
approximation.

\begin{figure}[p!]
  \centering
  \includegraphics[width=0.45\linewidth]{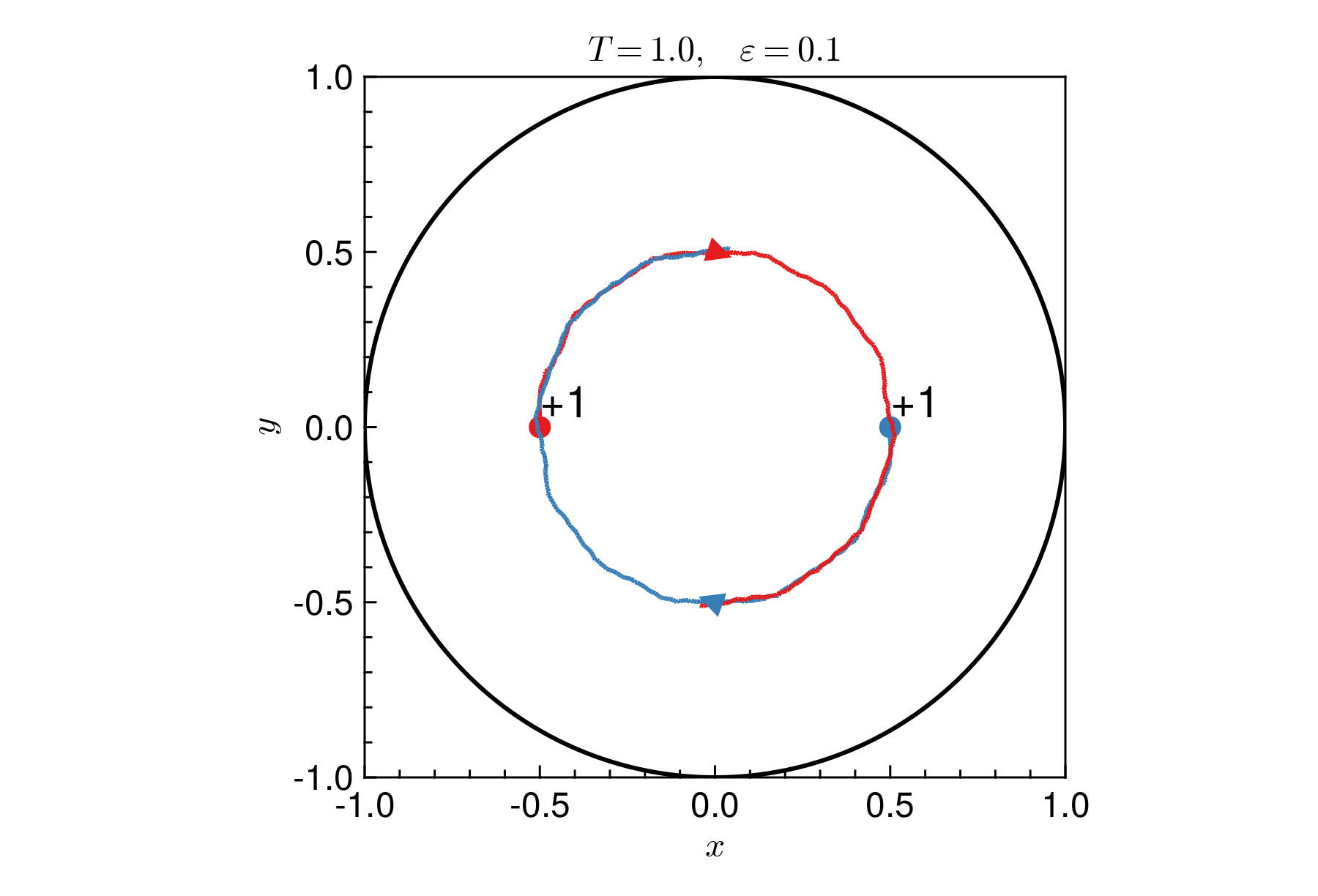}\hfill
  \includegraphics[width=0.45\linewidth]{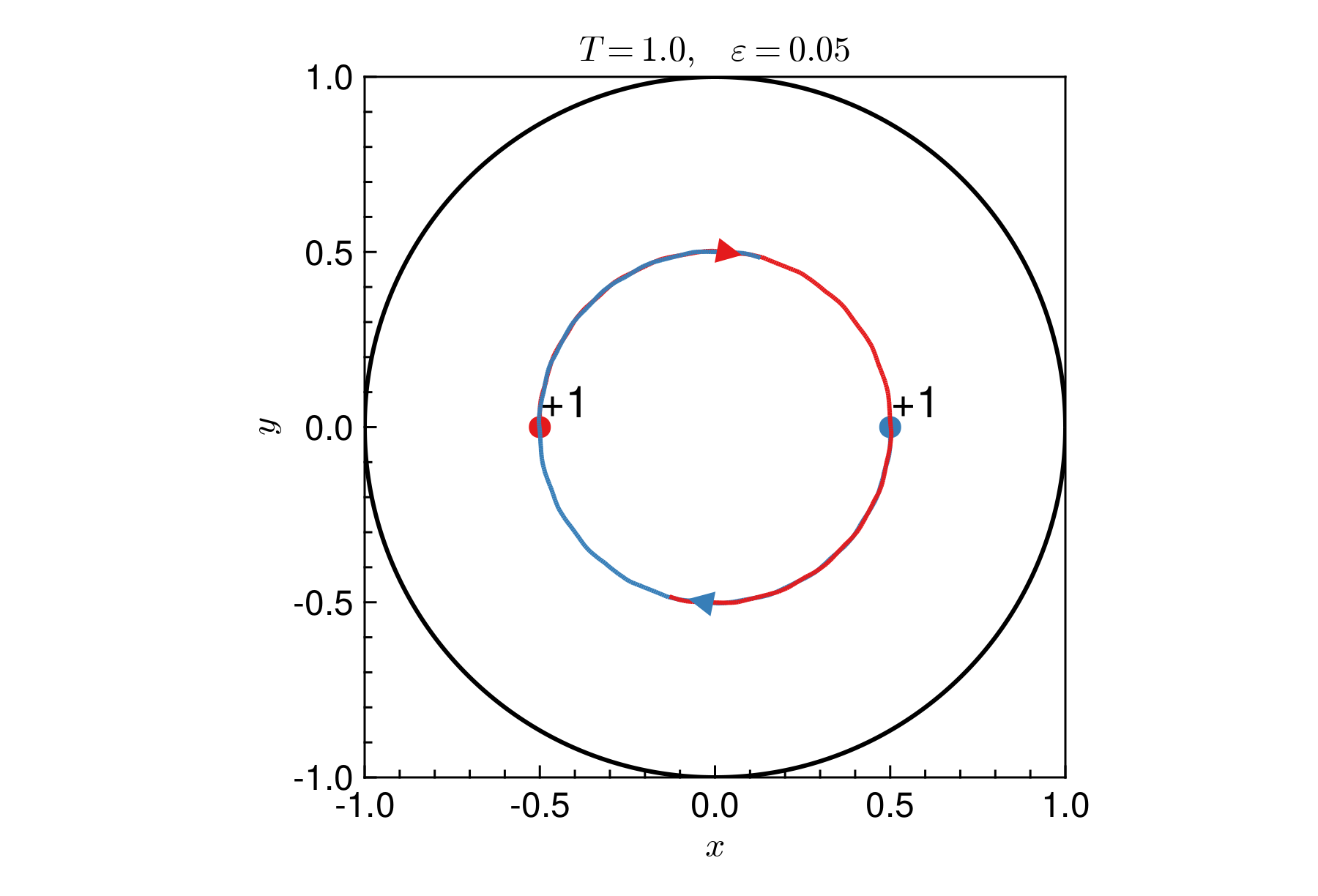}\\
  \includegraphics[width=0.45\linewidth]{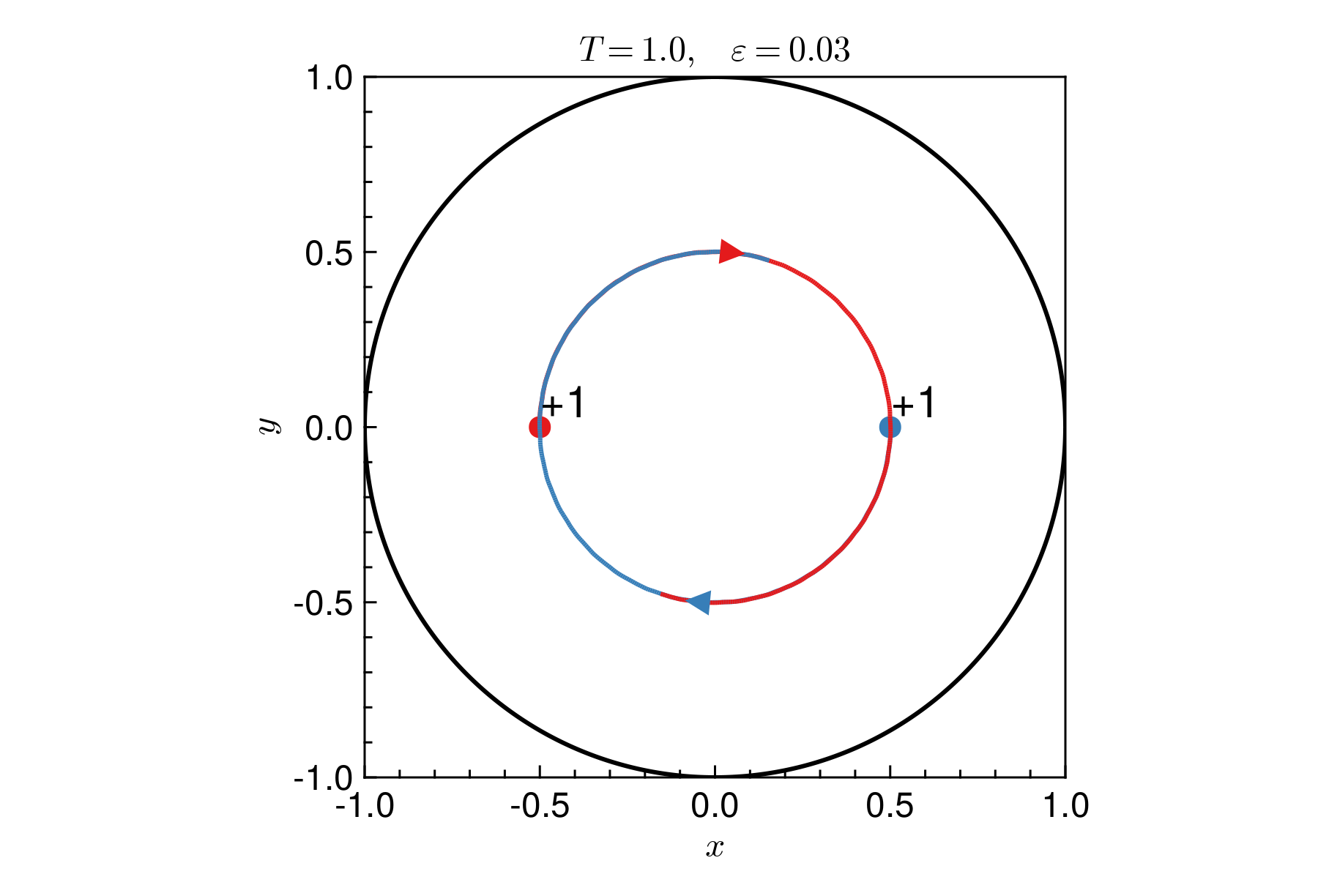}\hfill
  \includegraphics[width=0.45\linewidth]{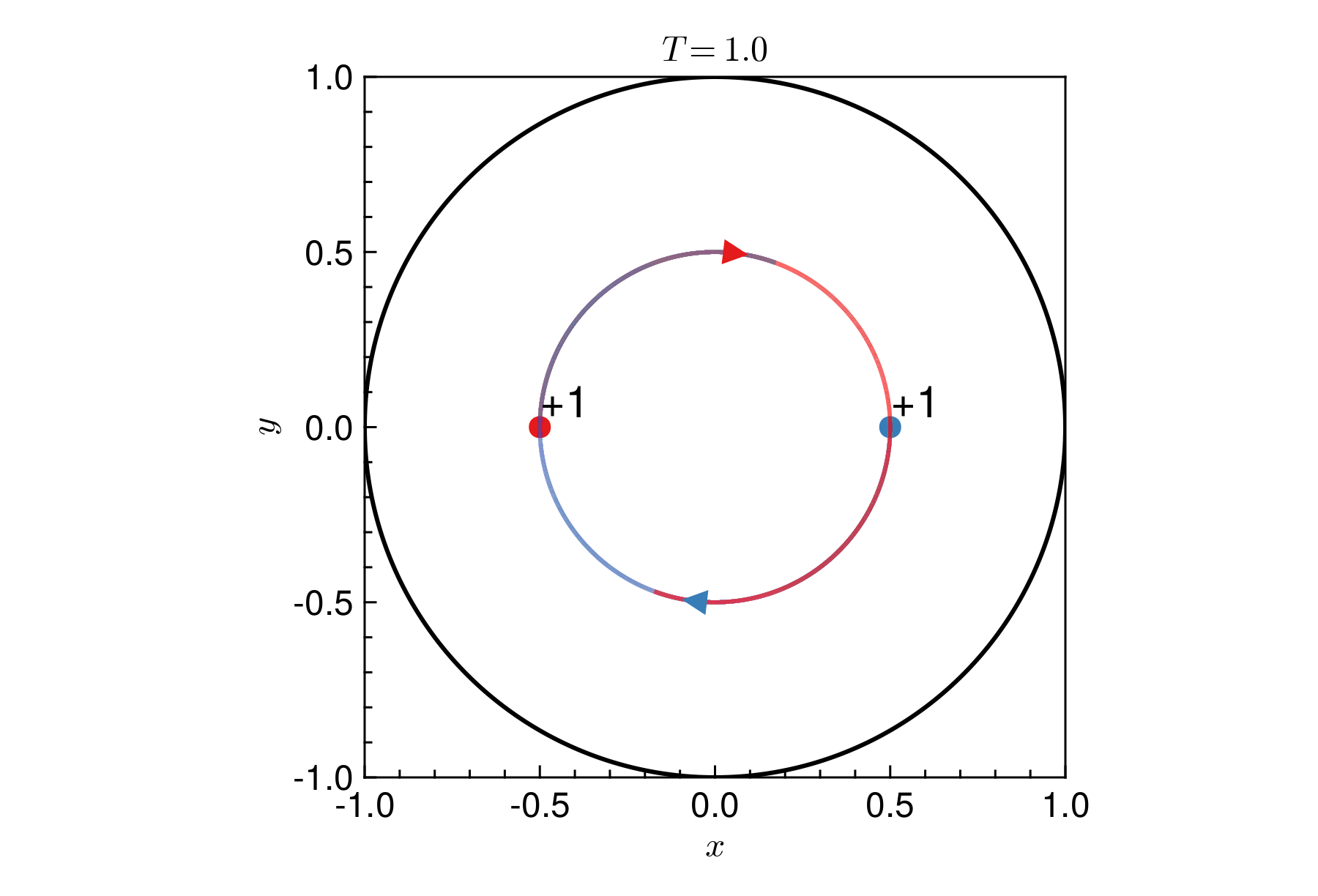}\\
  \caption{Vortex trajectories in the limit case for different values of
    $\varepsilon > 0$ and for $\varepsilon=0$ (bottom right). For finite $\varepsilon$, the
    trajectories oscillate around the limit trajectory.}
  \label{fig:traj_eps}
\end{figure}

\begin{figure}[p!]
  \centering
  \includegraphics[width=0.48\linewidth]{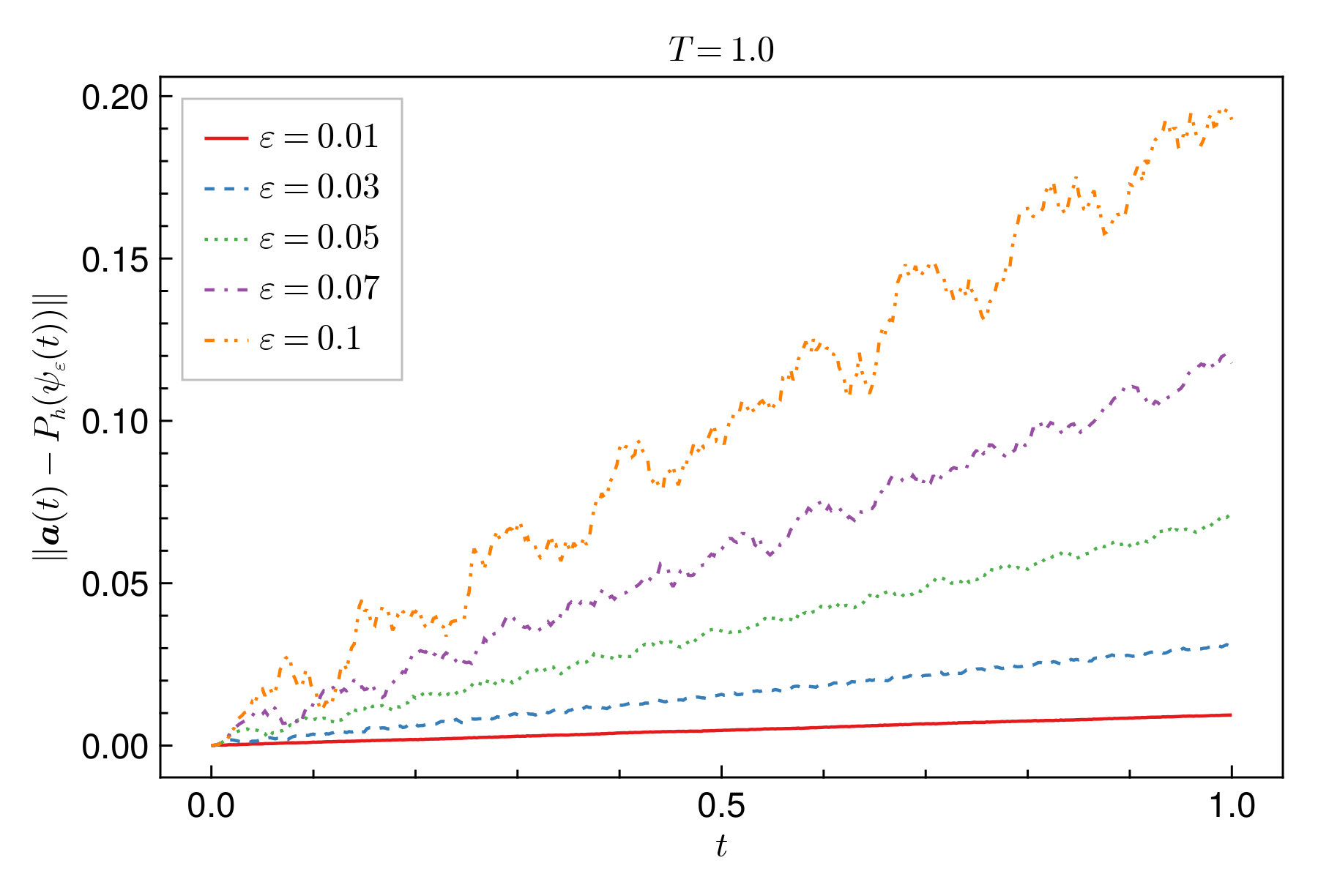}\hfill
  \includegraphics[width=0.48\linewidth]{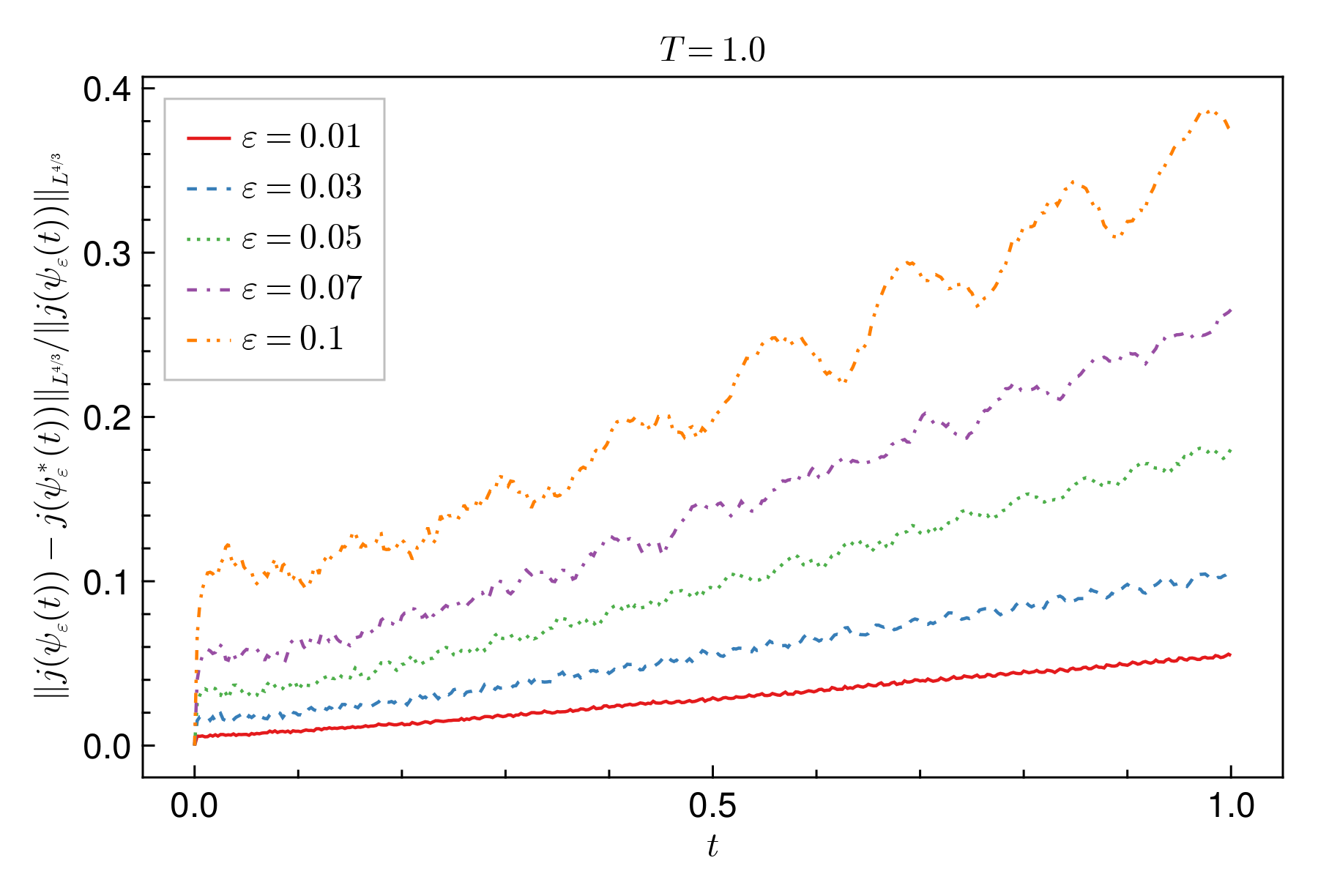}\\
  \includegraphics[width=0.48\linewidth]{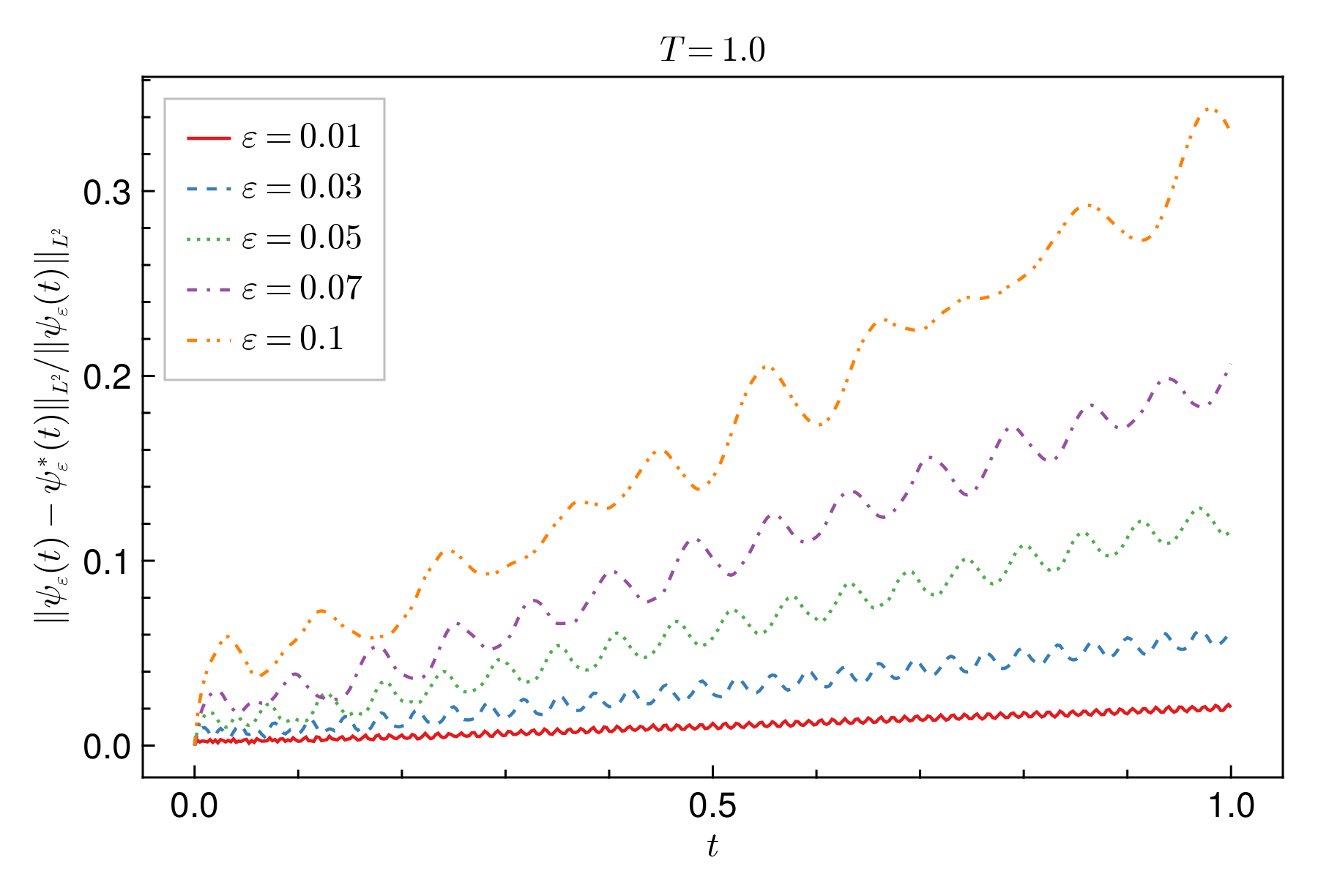}\hfill
  \includegraphics[width=0.48\linewidth]{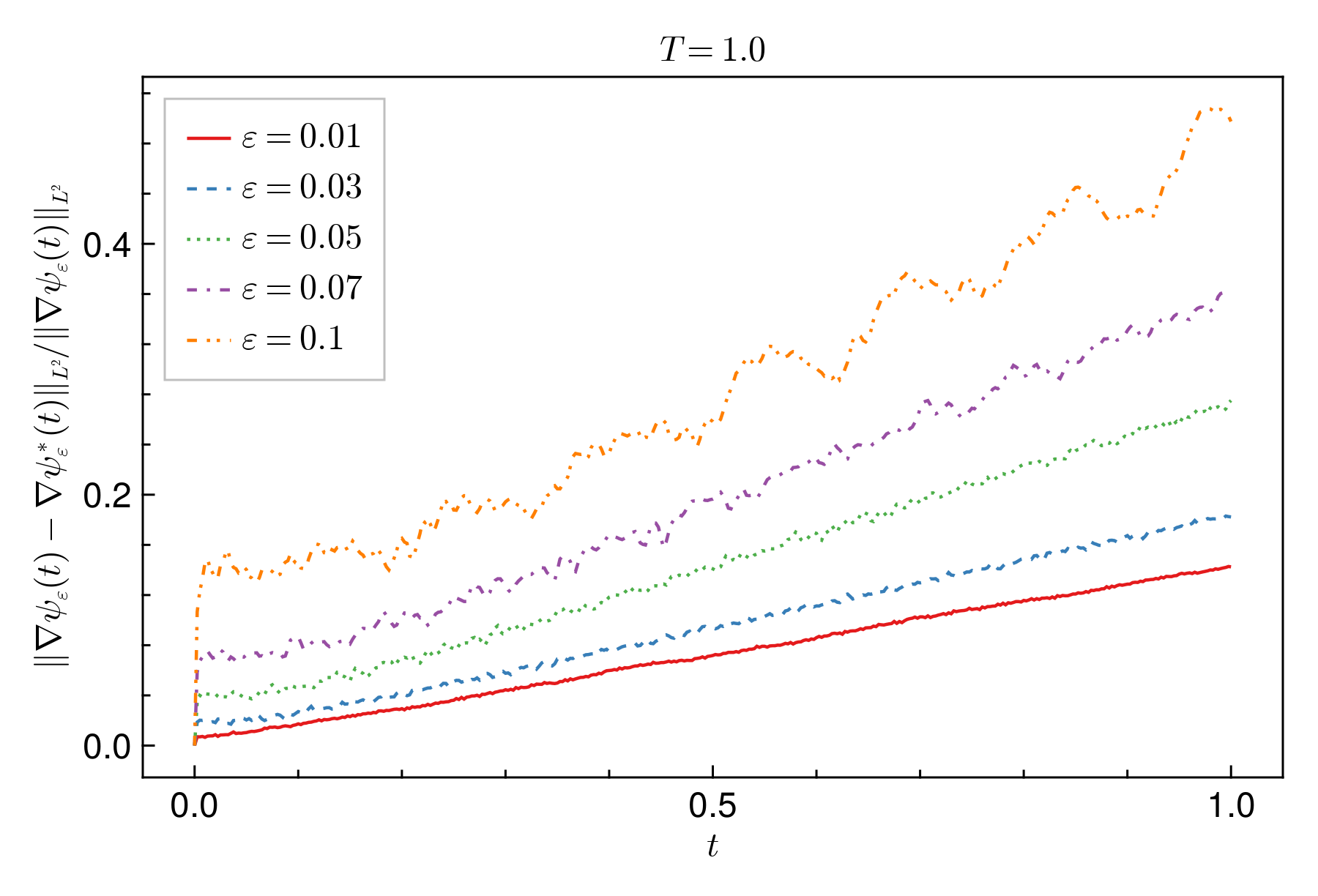}\\
  \caption{Error on different quantities up to time $T$ for case 1. (Top left)
    Error on vortices positions in $\Omega$. (Top right)
    $L^{\frac43}$ norm of the error on the supercurrents.
    (Bottom left) $L^2$ norm of the error between $\psi_\varepsilon(t)$
    and the reconstructed wave function $\psi_\varepsilon^*(t)$. (Bottom
    right) $L^2$ norm of the error between $\nabla \psi_\varepsilon(t)$
    and the reconstructed gradient $\nabla \psi_\varepsilon^*(t)$. .}
  \label{fig:error_eps}
\end{figure}

\begin{figure}[h!]
  \centering
  \includegraphics[width=0.48\linewidth]{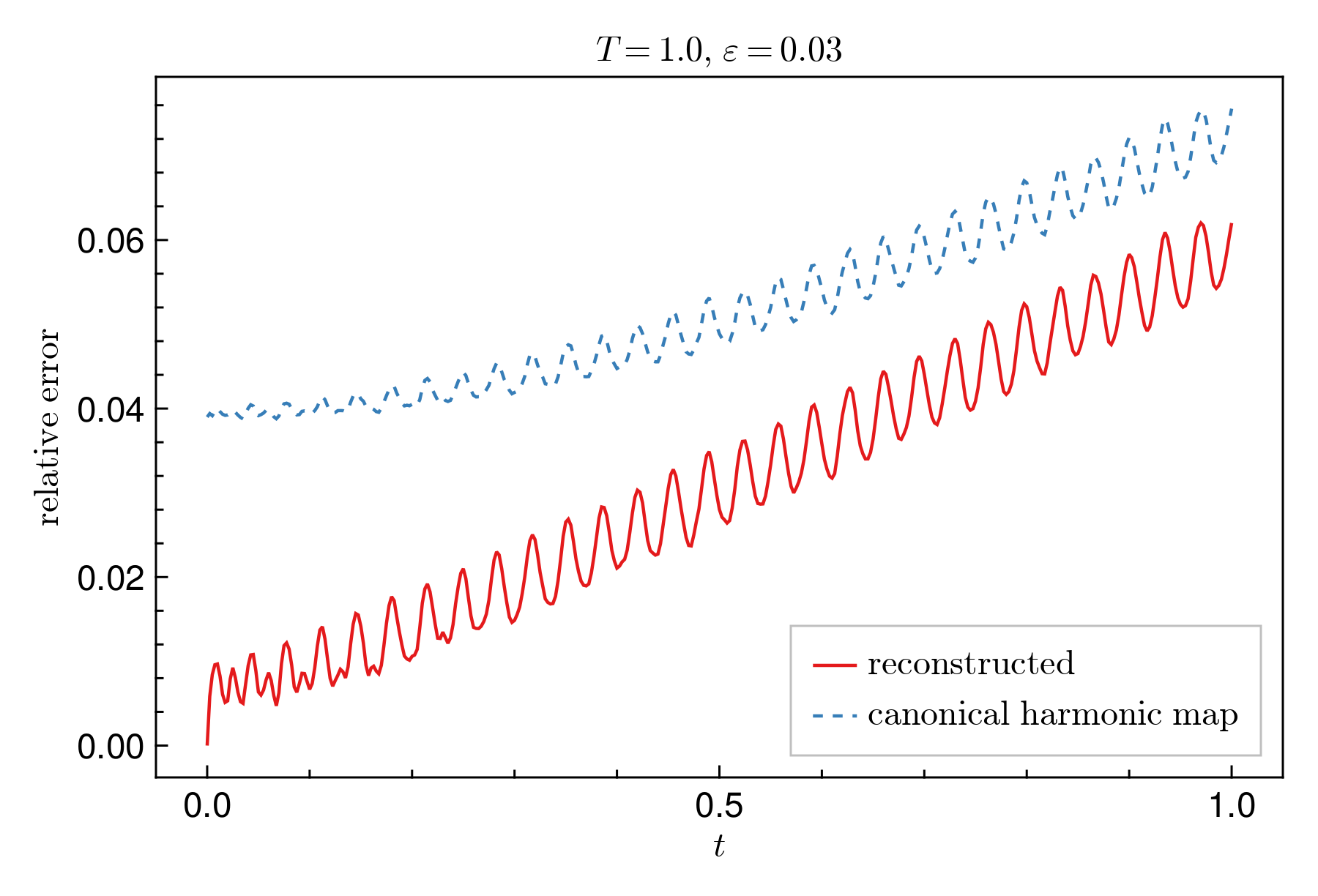}\hfill
  \includegraphics[width=0.48\linewidth]{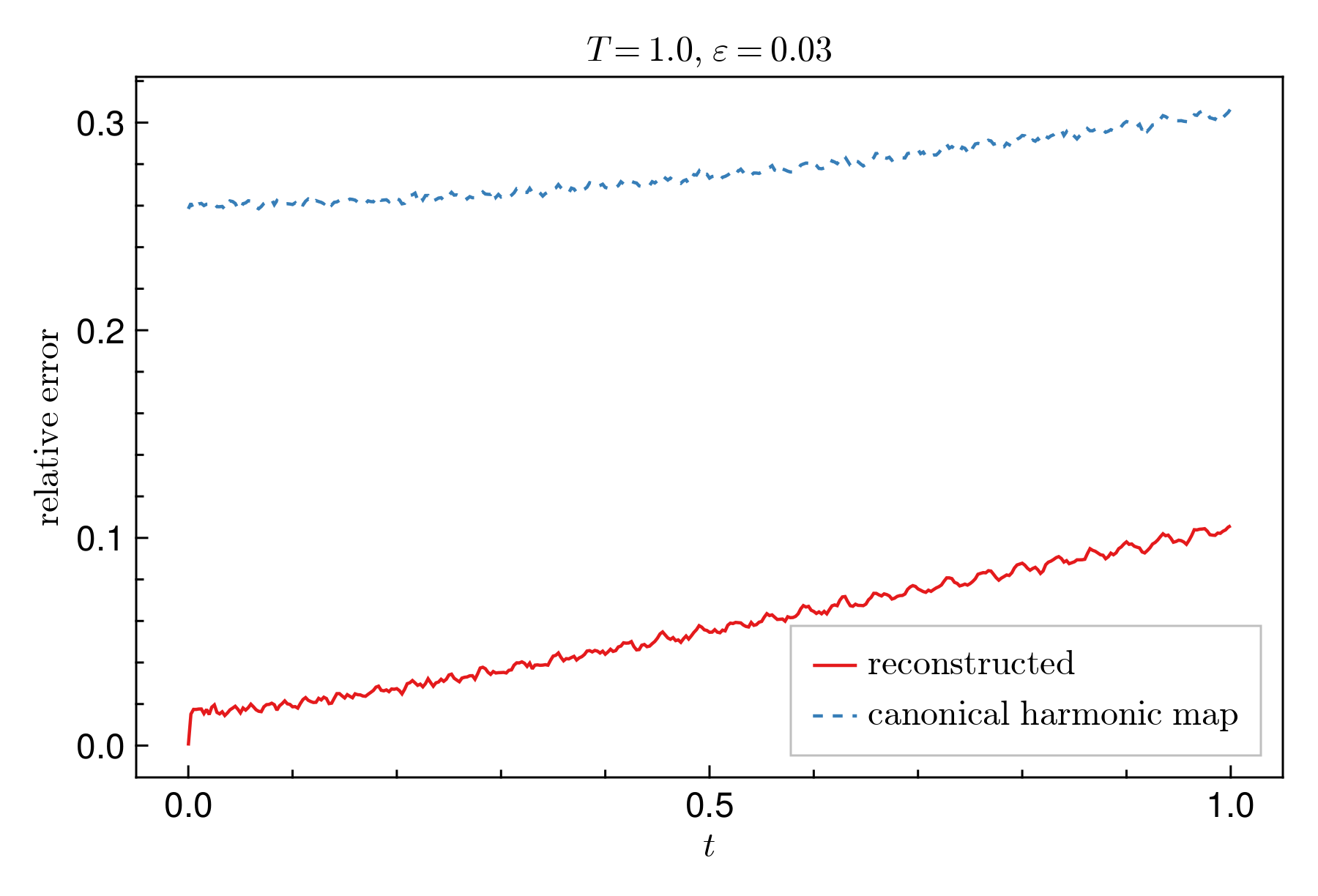}\\
  \caption{(Left) $L^2$ norm of the error between $\psi_\varepsilon(t)$
    and the canonical harmonic map $u^*({\bm a}(t), d)$
    \emph{vs} the error between $\psi_\varepsilon(t)$ and
    $\psi_\varepsilon^*(t)$, for $\varepsilon=0.03$. (Right)
    $L^{\frac43}$ norm of the error between $j(\psi_\varepsilon(t))$
    and $j(u^*({\bm a}(t), d))$
    \emph{vs} the error between $j(\psi_\varepsilon(t))$ and
    $j(\psi_\varepsilon^*(t))$, for $\varepsilon=0.03$. }
  \label{fig:error_vs_canon}
\end{figure}

\subsubsection{Numerical illustration of Theorem~\ref{thm:errorestimate}}

We end this section by a numerical illustration of
Theorem~\ref{thm:errorestimate} and the convergence with respect to
$\varepsilon$. Indeed,
with $\delta = 10^{-5}$ and $n=128$,
the previous numerical considerations lead to a dominating error contribution
coming from $\varepsilon >0$. From
Theorem~\ref{thm:errorestimate}, we expect it to be of order
$\varepsilon^\gamma$, with some exponent $\gamma$. In Figure~\ref{fig:linreg}, we plot
a linear regression in $\varepsilon$ of the error
\[
  \norm{j(\psi_\varepsilon(t)) - j(\psi_\varepsilon^*(t))}_{L^{\frac43}} \quad
  \text{and}\quad
  \norm{j(\psi_\varepsilon(t)) - j(u^*(\bm{a}(t),d)}_{L^{\frac43}}, \quad
\]
for different times $t$ and $\varepsilon \in \set{0.01, 0.03, 0.05, 0.07, 0.1}$.
The
linear regression suggests an exponent $\gamma > 1/2$ for the first error
and $\gamma$ closer to $\frac12$ for the second error. The latter is expected from
\eqref{eq:jest}, and the first plot therefore suggests that the smoothing procedure
reduces the error when approximating $j(\psi_\varepsilon(t))$ by
$j(\psi_\varepsilon^*(t))$.

\begin{figure}[h!]
  \centering
  \includegraphics[width=0.48\linewidth]{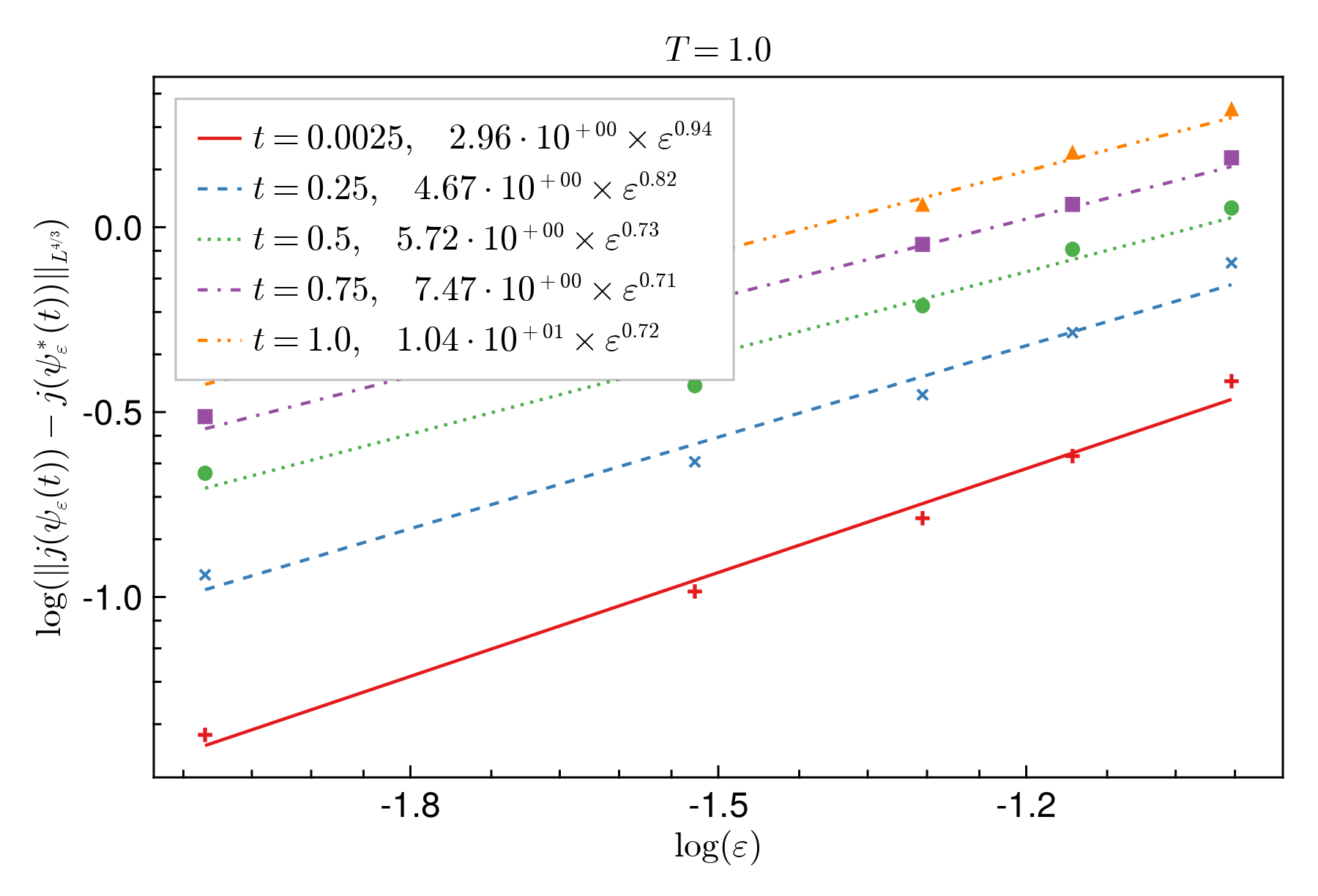}\hfill
  \includegraphics[width=0.48\linewidth]{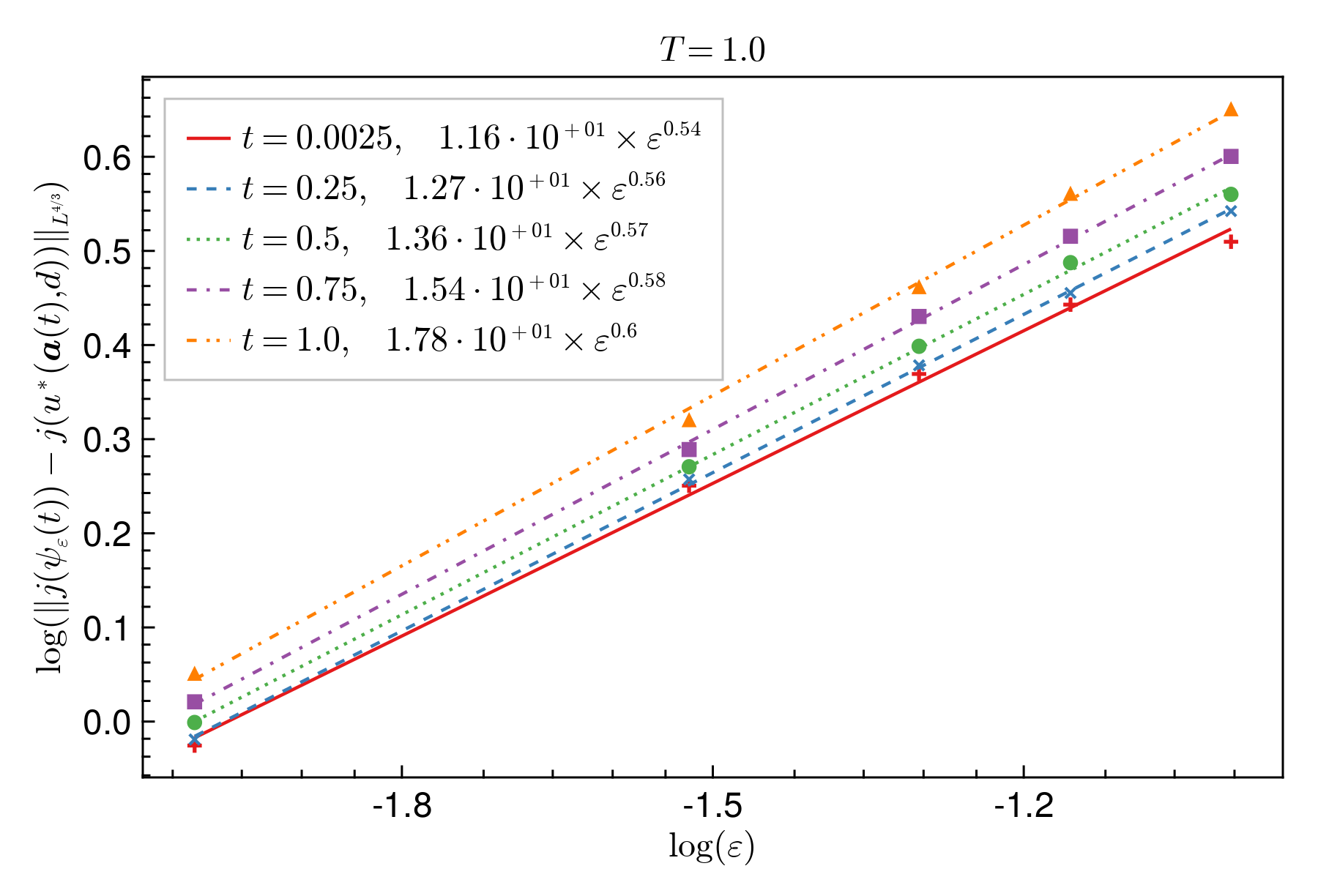}
  \caption{Linear regression in $\varepsilon$ on $\norm{j(\psi_\varepsilon(t)) -
      j(\psi_\varepsilon^*(t))}_{L^{\frac43}}$ (left) and $\norm{j(\psi_\varepsilon(t)) -
      j(u^*({\bm a}(t),d)}_{L^{\frac43}}$ (right), at different times $t$.}
  \label{fig:linreg}
\end{figure}

\section{Conclusion}

In this paper, we introduced a new method for the numerical simulation of the
Gross--Pitaevskii equation in the regime of small vortices of core size $\varepsilon
\ll 1$. From a computational point of view, this is a very demanding task as
small vortices require very fine spatial and time discretization to ensure
numerical stability. Based on the analytical theory of the singular limit
$\varepsilon\to0$, we introduced a cheap, efficient and
asymptotically valid numerical
method. We believe our method to have two main advantages: (i) the numerical
treatment of $\varepsilon$ is necessary in a pre-processing step only, where the
one-dimensional ODE \eqref{eq:ODE} can be solved accurately, and (ii) the
numerical solution is asymptotically valid, in the sense that the smaller
$\varepsilon$, the more accurate is the numerical approximation. We also
provided error estimates on the supercurrents in the case where $\Omega$ is the
unit disk, with numerical illustration of the convergence. While the limiting
dynamics are only valid up to the first vortex collision and are not able to
describe nonlinear phenomena induced when this happens (for instance, radiation
waves), one could imagine to simulate vortex interaction when one vortex and
anti-vortex become close to each other, see for instance Supplementary Materials
of \cite{simulaEmergenceOrderTurbulence2014}.
We expect this method to work in more general cases where the vortex motion can be
reduced to a system of ODE, for instance when using the gradient flow or mixed flow instead of
the Schrödinger flow, see \eg
\cite{baoNumericalStudyQuantized2013,kurzkeDynamicsGinzburgLandauVortices2009,sandier2004gamma}.
The case of non-vanishing vector potential, for instance by adding a magnetic field
\cite{gustafsonEffectiveDynamicsMagnetic2006a,sandier2004gamma,spirnVortexDynamicsFull2002,spirnVortexMotionLaw2003}
seems also applicable, even if further investigations are needed in this case as
there are now two unknowns: the order parameter $\psi_\varepsilon$ and the magnetic field potential $A_\varepsilon$ that scale together when~
$\varepsilon\to0$.

\section*{Data availability}

All the codes used to generate the plots from this paper are available at
\begin{center}\url{https://doi.org/10.18419/darus-4229}.\end{center}

\section*{Acknowledgments}

The authors T.C, G.K, C.M and B.S acknowledge funding by the Deutsche
Forschungsgemeinschaft (DFG, German Research Foundation) - Project number
442047500 through the Collaborative Research Center \enquote{Sparsity and
  Singular Structures} (SFB 1481). The authors are also grateful to Andreas A. Buchheit
and Israel M. Sigal for fruitful discussions and suggestions.

\bibliographystyle{abbrv}
\bibliography{biblio}

\appendix

\section{Error estimates}

In this section we derive some error estimates used in the proof of
Theorem~\ref{thm:errorestimate}.
There are two main sources of error for our numerical scheme: (i) the
evaluation of the forcing term of the ODE requires to solve a PDE at each
time step which introduces a first discretization error, and (ii) the numerical
resolution of the ODE itself introduces a time discretization error.

In all this appendix, we consider, for $s\in\R$, the fractional Sobolev spaces
$H^s(\partial\Omega)$ equipped with norm
\[
  \norm{g}_{H^s(\partial\Omega)}^2 = 2\pi \sum_{k\in\Z} (1+|k|^2)^s
  |\widehat{g}(k)|^2,
\]
where the Fourier coefficients are defined by
\[
  \widehat{g}(k) = \frac1{2\pi} \int_0^{2\pi} e^{-\i k \theta} g(e^{\i \theta})
  \mathrm{d}\theta.
\]

\subsection{Exact {\it vs} approximate forcing term}

According to Section~\ref{sec:HD_sim}, we are interested in the solution
$\bm{a}~:~\R_+~\rightarrow~\Omega^{* N}$ of the ODE
\begin{align}
  \begin{cases} \dot{\bm{a}}(t) = F\bigr(\bm{a}(t)\bigr),\quad t\in \R_+ \\
    \bm{a}(0) = \bm{a_0} \in \Omega^{* N}, \end{cases} \label{eq:exactdynamics}
\end{align}
where the components $F_j: \Omega^{* N} \rightarrow \R^{2}$ of the forcing term $F(\bm{a}) = (F_1(\bm{a}),...,F_N(\bm{a})) \in \R^{2N}$ are given by
\begin{align*}
  F_j(\bm{a}) = 2\J \biggr( \nabla_{x} R(x;\bm{a},d)\bigr\rvert_{x = a_j} +
  \sum_{k\neq j}^N d_k \frac{a_j-a_k}{|a_j-a_k|} \biggr)\quad \quad
  \mbox{for some } \{d_j\}_{j=1}^N \in \{\pm 1\}^N,
\end{align*}
and $R(\cdot;\bm{a},d): \Omega \rightarrow \R$ is the solution of
\begin{align*}
  \begin{cases}
    \Delta_x R(x;\bm{a},d) = 0,\quad\text{in }\Omega,\\
    \displaystyle R(x;\bm{a},d) =  - \sum_{j=1}^N d_j\log|x -
    a_j| \quad\text{on }\partial\Omega.
  \end{cases}
\end{align*}
After fixing the basis of harmonic polynomials up to order $n$, the approximate forcing term we evaluate is
\begin{align*}
  F_j^n(\bm{a}) = 2\J\biggr( \nabla_{x} R_n(x;\bm{a},d)\bigr\rvert_{x = a_j} +
  \sum_{k\neq j}^N d_k \frac{a_j-a_k}{|a_j-a_k|}\biggr) \quad \quad
  \mbox{for some } \{d_j\}_{j=1}^N \in \{\pm 1\}^N,
\end{align*}
with $R_n(\cdot;\bm{a},d): \Omega \rightarrow \R$ being the solution of
\begin{align*}
  \begin{cases}
    \Delta_x R_n(x;\bm{a},d) = 0,\quad\text{in }\Omega,\\
    \displaystyle R_n(x;\bm{a},d) =  - \P_n\biggr(\sum_{j=1}^N d_j\log|x -
    a_j|\biggr), \quad\text{on }\partial\Omega,
  \end{cases}
\end{align*}
where $\P_n$ is the $L^2(\partial \Omega)$-orthogonal projection on the bases of harmonic polynomials up to order $n$ on the boundary. Our task is to estimate the difference $R_n - R$ in suitable norms. To this end, we shall use the following well-known estimates.

\begin{lemma}[Classical estimates for Dirichlet problem
  \cite{gilbargEllipticPartialDifferential2001}]
  Let $\Omega \subset \R^2$ be a bounded smooth domain, $g~\in~C^\infty(\partial
  \Omega)$, and $u \in C^\infty(\overline{\Omega})$ be the unique solution of
  \begin{align}
    \begin{cases}
      \Delta u = 0, \quad \mbox{in $\Omega$,}\\
      u = g, \quad \mbox{in $\partial \Omega$.}
    \end{cases} \label{eq:Laplace}
  \end{align}
  Then there exists implicit constants depending only on $\Omega$ and $\ell \in \N\cup\{0\}$ such that
  \begin{align}
    &|\partial^\alpha u(x)| \lesssim_{\ell} \frac{1}{\dist(x,\partial \Omega)^{\frac{1}{2}+\ell}} \norm{g}_{L^2(\partial \Omega)} \quad\mbox{for any $\alpha \in (\N \cup \{0\})^2$ with $|\alpha| = \ell$,} \label{eq:classicalest1} \\
    &\norm{\nabla u }_{L^2(\Omega)} \lesssim \norm{g}_{H^{\frac12}(\partial \Omega)}. \label{eq:classicalest2}
  \end{align}
  Moreover, in the case where $\Omega$ is the unit disk and $\P_n g = 0$, there
  exists constants independent of $n$, $g$, and $x \in \Omega$ such that
  \begin{align}
    &|\partial^\alpha u(x)| \lesssim_{\ell} \frac{n^\ell |x|^{n+1-\ell}}{(1-|x|)^{\frac{1}{2}+\ell}} \norm{g}_{L^2(\partial \Omega)}, \quad \mbox{for any $\alpha \in (\N\cup \{0\})^2$ with $|\alpha| =\ell \geq 1$.} \label{eq:classicaldiskest} \\
    &\norm{\nabla u}_{L^2(\Omega)}^2  = \sum_{|k| >n} 2 \pi |k|
    |\widehat{g}(k)|^2
    \label{eq:classicaldiskest2}.
  \end{align}
\end{lemma}

\begin{proof} Estimate~\eqref{eq:classicalest1} and \eqref{eq:classicalest2} are
  rather classical and detailed proofs can be found in \cite[Section
  2.7]{gilbargEllipticPartialDifferential2001}. Let us thus focus on
  the specific estimate for the disk and briefly sketch the argument in the
  general case. First, note that for the disk, we have the explicit formula for
  the solution via the Poisson kernel $P_r(\theta)$ or, equivalently, via the
  Fourier coefficients of $g$,
  \begin{align*}
    u(re^{\i \theta}) = -\frac{1}{2\pi} \int_0^{2\pi} P_r(\theta-\phi) g(e^{\i
      \phi}) \mathrm{d}\phi = \sum_{k \in \Z} r^{|k|} \widehat{g}(k) e^{\i
      k\theta}.
  \end{align*}
  Now, note that the harmonic polynomials of order up to $n$ span the same space
  as $\{z^k\}_{k=0}^n \cup \{\overline{z}^k\}_{k=1}^n$, which are just the
  Fourier basis if restricted to the boundary. In particular, if $\P_n g = 0$, then the above
  sum runs over the integers $|k| > n$. From the orthogonality of $r^{|k|} e^{\i
    k\theta}$ on $L^2(\Omega)$, we then have
  \begin{align*}
    \norm{u}_{L^2(\Omega)}^2 &= \sum_{|k| >n} |\widehat{g}(k)|^2 \norm{|r|^k
      e^{\i k\theta}}_{L^2(\Omega)}^2 = \sum_{|k|>n} \frac{\pi}{|k|+1}
    |\widehat{g}(k)|^2.
  \end{align*}
  Moreover, we can express the radial and angular component of the (a priori weak) gradient of $u$ as
  \begin{align*}
    \partial_r u(re^{\i \theta}) = \sum_{|k|>n} |k| \widehat{g}(k) r^{|k|-1}
    e^{\i k\theta}\quad\mbox{and}\quad \frac{1}{r} \partial_\theta u(re^{\i
      \theta}) = \sum_{|k|>n} \i k \widehat{g}(k) r^{|k|-1} e^{\i k \theta}.
  \end{align*}
  Consequently, from Cauchy-Schwarz we find
  \begin{align}
    |\nabla u(r e^{\i \theta})|^2 &\leq 2\biggr(\sum_{|k|>n}
    |\widehat{g}(k)|^2\biggr)\biggr(2 \sum_{k>n} k^2 r^{2(k-1)}\biggr)  \nonumber \\
    &= \frac{1}{\pi} \norm{g}_{L^2(\partial \Omega)}^2 \biggr(
    \frac{2r^{2n}}{(1-r^2)^3}
    \underbrace{\bigr(n^2 r^4 - (2n^2+2n-1) r^2 + (n+1)^2\bigr)}_{\leq (n+1)^2
      \text{ for } r \in [0,1]}
    \biggr)\nonumber\\
    &\leq \frac{2}{\pi} \frac{(n+1)^2 r^{2n}}{(1-r)^3} \norm{g}_{L^2(\partial \Omega)}^2, \label{eq:pointwiseest} \\
    \norm{\nabla u}_{L^2(\Omega)}^2 &= \int_0^1 \int_0^{2\pi} \biggr(|\partial_r
    u|^2 + \frac{1}{r^2}|\partial_\theta u|^2 \biggr) r\mathrm{d}r \mathrm{d} \theta \nonumber \\
    &= \sum_{|j|,|k|>n} \overline{\widehat{g}(k)} \widehat{g}(j)
    \biggr(|k||j|\inner{r^{|k|-1} e^{\i k\theta}, r^{|j|-1} e^{\i
        j\theta}}_{L^2(\Omega)} + k j \inner{ r^{|k|-1} e^{\i k\theta},
      r^{|j|-1} e^{\i j\theta}}_{L^2(\Omega)}\biggr) \nonumber \\
    &= \sum_{|k| >n} 2 \pi |k| |\widehat{g}(k)|^2 \leq \norm{g}_{H^{\frac12}(\partial \Omega)}^2, \label{eq:gradest}
  \end{align}
  which completes the proof of the first part for $\ell=1$. The general $\ell$ case follows from the same steps by noticing that each radial and angular derivative gives a factor of $k/r$ in the series, and using the bound
  \begin{align*}
    \sum_{k >n} k^{2\ell} r^{2k} \lesssim_\ell \frac{n^{2\ell} r^{2n+2}}{(1-r)^{2\ell+1}}, \quad \ell \in \N,
  \end{align*}
  which follows by induction on (or differentiating) the well-known formula for the partial sums of the geometric series.

  For the general smooth domain case, one can apply a biholomorphic conformal
  transformation of $\Omega$ into the disk. As the transformation and its
  inverse are smooth up to the boundary (see, \eg \cite[Chapter
  5]{krantzGeometricFunctionTheory2007}) the $L^2$ and $H^{\frac12}$ norm on the
  boundary of $\Omega$ are, respectively, equivalent to the $L^2$ and
  $H^{\frac12}$ norms on the circle. Therefore, eqs.~\eqref{eq:classicalest1}
  and \eqref{eq:classicalest2} follows from \eqref{eq:pointwiseest} and
  \eqref{eq:gradest}.
\end{proof}

\begin{lemma}[Log function estimates] Let $(a,b) \in \Omega^{*2}$ and $g_a(x) = \log|x-a|$, then we have
  \begin{align}
    &\norm{g_a - g_b}_{H^s(\partial \Omega)} \lesssim_s \frac{|a-b|}{\min\{\dist(a,\partial \Omega) \dist(b,\partial \Omega)\}^{s+1}} \quad \mbox{for any $0 \leq s\leq 1$,} \label{eq:logboundaryest} \\
    &\norm{\nabla g_a- \nabla g_b}_{L^p(\Omega)} \lesssim |a-b|^{\frac{2}{p}-1} , \quad \mbox{for any $1\leq p <2$} \label{eq:loginteriorest} \\
    &\norm{g_a}_{H^{\ell}(\partial \Omega)} \lesssim \frac{1}{\dist(a,\partial \Omega)^{\max\{\ell-\frac{1}{2},0\}}}, \quad \mbox{for any $\ell \in \N\cup\{0\}.$} \label{eq:lognorms}
  \end{align}
  Moreover, if $\Omega$ is the unit disk we have
  \begin{align}
    &\norm{ \P_n^\perp g_a}_{H^s(\partial \Omega)} \lesssim_{\ell}
    \frac{n^{-\ell+s}}{\dist(a,\partial \Omega)^{\ell-\frac12}} \quad \mbox{for
      any $0\leq s\leq \ell$ and $\ell \geq 1$,}  \label{eq:logspectralconv}
  \end{align}
  where $\P_n^\perp = 1- \P_n$ and $\P_n$ is the orthogonal projection defined previously and the implicit constants are independent of $a, b$ and $n$.
\end{lemma}

\begin{proof} Since $\nabla \log|x| = x/|x|^2$ is positively homogeneous of
  degree $-1$, we find $\big|\nabla^k \log |x|\big| \lesssim_k |x|^{-k}$ for any $k \in
  \N$. In fact, by induction one can show that for any $\alpha \in \N^2$ with $|\alpha| \geq 1$ we have
  \begin{align}
    \partial^\alpha \log|x| = \frac{p_\alpha(x)}{|x|^{2|\alpha|}}, \quad \mbox{for some polynomial $p_\alpha$ homogeneous of degree $|\alpha|$,} \label{eq:polynomialrep}
  \end{align}
  which implies
  \begin{align}
    \big|\log|x| - \log|y|\big| \lesssim \frac{|x-y|}{\min\{|x|,|y|\}} \quad\mbox{and}\quad \big|\nabla \log|x| - \nabla \log|y|\big| \lesssim \frac{|x-y|}{\min\{|x|,|y|\}^2}. \label{eq:est1}
  \end{align}
  Integrating these estimate along the boundary, we obtain
  \eqref{eq:logboundaryest} for $s =0$ and $s=1$. The estimates for $0<s<1$ then follow by interpolation.

  For the $L^p$ estimate in the interior, we split the integration over the
  balls $B_\rho(b)$, $B_\rho(a)$, and the rest of the domain $\Omega_{a,b} =
  \Omega \setminus \bigr(B_\rho(b)\cup B_\rho(a)\bigr)$ for some $\rho<|a-b|/2$.
  Then since the terms $\nabla g_a \sim |x-a|^{-1}$ gives the dominant
  contribution inside the ball $B_\rho(a)$, we can use estimate~\eqref{eq:est1}
  on $\Omega_{a,b}$ to obtain
  \begin{align*}
    \norm{\nabla g_a - \nabla g_b}_{L^p(\Omega)}^p &\lesssim \int_{B_\rho(b)} \frac{1}{|x-b|^p} \mathrm{d} x + \int_{B_\rho(a)} \frac{1}{|x-a|^p} \mathrm{d}x + \int_{\Omega_{a,b}} \frac{|a-b|^p}{\min\{|x-b|,|x-a|\}^{2p}} \mathrm{d} x\\
    &\lesssim \rho^{2-p} + \frac{|a-b|^{p}}{\rho^{2p-2}}.
  \end{align*}
  Minimizing the right hand-side, we find $\rho \sim |a-b|$, which yields \eqref{eq:loginteriorest}.

  Lastly, to prove the spectral convergence for the $L^2(\partial \Omega)$ norm,
  we can use \eqref{eq:polynomialrep} again. Indeed, by integrating the function
  $|x-a|^{-2\ell}$ along the boundary we have
  \begin{align*}
    \norm{g_a}_{H^\ell(\partial \Omega)}^2 \cong \norm{\nabla^\ell
      \log|x-a|}_{L^2(\partial\Omega)}^2 \lesssim  \begin{dcases} \dist(a,\partial \Omega)^{-2\ell+1},\quad \mbox{for $\ell \in \N$,}\\
      1, \quad \mbox{for $\ell = 0$,} \end{dcases}
  \end{align*}
  which proves~\eqref{eq:lognorms}. Since we have the equivalence of norms
  $\norm{g}_{H^s(\partial \Omega)}^2 \cong \sum_{k \in \Z} |\widehat{g}(k)|^2
  (1+|k|)^{2s}$ on the disk, we find
  \begin{align*}
    \norm{\P_n^\perp g_a}_{H^s(\partial \Omega)}^2 \cong \sum_{|k| > n}
    |\widehat{g_a}(k)|^2(1+|k|)^{2s} \lesssim  n^{-2\ell+2s} \norm{g_a}_{H^{\ell}(\partial
      \Omega)}^2 =  \frac{n^{-2\ell+2s}}{\dist(a,\partial \Omega)^{2\ell-1}}, \quad
    \mbox{for $\ell \in \N$.}
  \end{align*}
  The case $\ell \geq 1$ for real $\ell$ follows by interpolation.
\end{proof}

Let us now denote by $\bm{b}: \R_+ \to \Omega^{* N}$ the solution of the
approximated ODE
\begin{align}
  \begin{cases} \dot{\bm{b}}(t) = F^n(\bm{b}(t)), \\
    \bm{b}(0) = \bm{b}^0,
  \end{cases} \label{eq:approximateddynamics}
\end{align}
where $F^n$ is the approximated forcing term. Then an application of Gronwal's inequality with the previous estimates yields
\begin{lemma}[Continuous dynamical error]\label{lem:dyn_err} Let $\bm{b}^0,
  \bm{a}^0 \in \Omega^{* N}$ and $\bm{a}(t)$ and $\bm{b}(t)$ be respectively
  the solutions of \eqref{eq:exactdynamics} and \eqref{eq:approximateddynamics},
  for a fixed $d\in\set{\pm1}^N$ and let
  \begin{align*}
    \rho(\bm{a}) \coloneqq \min_{j\neq \ell} \{\dist(a_j,\partial\Omega), |a_j-a_\ell|\}.
  \end{align*}
  Suppose that $\rho_T \coloneqq \inf_{t\in [0,T]} \rho\bigr(\bm{a}(t)\bigr) >0$. Then, there exists some $C = C(\Omega,N)>0$ such that
  \begin{align}
    |\bm{a}(t)-\bm{b}(t)| \lesssim \biggr(|\bm{a}^0-\bm{b}^0| + \frac{\sup_{s
        \in [0, T]} \norm{\P_n^\perp g_{\bm{a}(s)}}_{L^2(\partial
        \Omega)}}{\rho_T^{3/2}} T\biggr) e^{Ct/\rho_T^{5/2}}, \quad \mbox{for
      any $t\in [0,T]$,} \label{eq:gronwalest}
  \end{align}
  where $g_{\bm{a}}(x) = - \sum_{j=1}^N d_j \log|x-a_j|$. Moreover, if $\Omega$ is the unit disk we have
  \begin{align}
    |\bm{a}(t)-\bm{b}(t)| \lesssim_{\ell} \biggr(|\bm{a}^0-\bm{b}^0| + \frac{(1-\rho_T)^{n} n^{1 - \ell}}{\rho_T^{1+\ell}} T\biggr) e^{Ct/\rho_T^{5/2}}, \quad \mbox{for any $t\in [0,T]$,} \label{eq:gronwalestdisk}
  \end{align}
  and any $\ell \in \N$, where all the implicit constants are independent of $\bm{a}^0,\bm{b}^0, n,\rho_T$, and $t$.
\end{lemma}

\begin{proof} First, we split the error via the triangle inequality as
  \begin{multline*}
    |F^n(\bm{b}) - F(\bm{a})| \leq \underbrace{\sum_{j=1}^N |\nabla_x
      R_n(b_j;\bm{b},d)- \nabla_x R_n(a_j;\bm{a},d)|}_{\coloneqq
      E_1(\bm{a},\bm{b})} + \underbrace{\sum_{j=1}^N |\nabla R_n(a_j;\bm{a},d) -
      \nabla R(a_j;\bm{a},d)|}_{\coloneqq E_2(\bm{a})}\\
    + \underbrace{\sum_{k\neq j}^N \biggr|\frac{b_j-b_k}{|b_j-b_k|^2} -
      \frac{a_j-a_k}{|a_j-a_k|^2}\biggr|}_{\coloneqq E_3(\bm{a},\bm{b})}.
  \end{multline*}
  From estimate~\eqref{eq:est1}, we can bound the third term as
  \begin{align}
    E_3(\bm{a},\bm{b}) \lesssim
    \frac{|\bm{a}-\bm{b}|}{\min\{\rho(\bm{a}),\rho(\bm{b})\}^2}
    \label{eq:error3est}.
  \end{align}
  For the second error, we note that $R_n(\cdot;\bm{a},d) - R(\cdot;\bm{a},d)$
  is the solution to Laplace's equation with boundary data $\P_n^\perp
  g_{\bm{a}}$, where $g_{\bm{a}}(x) = -\sum_{j=1}^N d_j \log(x-a_j)$. Hence, from
  estimate~\eqref{eq:classicalest1} with $\ell=1$, we obtain
  \begin{align}
    E_2(\bm{a}) \lesssim \frac{1}{\min_j \{\dist(a_j,\partial
      \Omega)\}^{\frac32}} \norm{\P_n^\perp g_{\bm{a}}}_{L^2(\partial \Omega)}. \label{eq:error2est}
  \end{align}
  To control the last term, we observe (i) that $R_n(\cdot;\bm{b},d) -
  R_n(\cdot;\bm{a},d)$ solves Laplace's equation with boundary data $\P_n
  (g_{\bm{a}}-g_{\bm{b}})$, and (ii) that estimate~\eqref{eq:classicalest1} for
  the second derivative of $R_n(\cdot;\bm{a},d)$ gives an upper bound on the
  Lipschitz constant of the first derivative. Thus, from
  estimates~\eqref{eq:logboundaryest} and \eqref{eq:classicalest1}, we obtain
  \begin{align}
    E_1(\bm{a},\bm{b}) &\leq \sum_{j=1}^N \abs{\nabla_x R_n(a_j;\bm{a},d) -
      \nabla_x R_n(a_j;\bm{b},d)} + \abs{\nabla_x R_n(a_j;\bm{b},d) - \nabla_x R_n(b_j;\bm{b},d)} \nonumber \\
    &\lesssim \frac{\norm{\P_n(g_{\bm{a}}-g_{\bm{b}})}_{L^2(\partial \Omega)}}{\min_j \{\dist(a_j,\partial \Omega)\}^{\frac32}} + \frac{|\bm{a}-\bm{b}|}{\min_j\{ \dist(b_j,\partial \Omega)\}^{\frac52}} \norm{\P_n g_{\bm{b}}}_{L^2(\partial \Omega)} \nonumber \\
    &\lesssim |\bm{a}-\bm{b}| \frac{1}{\min_j\{\dist(b_j,\partial \Omega),\dist(a_j,\partial \Omega)\}^{\frac52}}, \label{eq:error1est}
  \end{align}
  where we used that $\norm{g_{\bm{a}}}_{L^2(\partial \Omega)} \lesssim 1$ with
  a constant independent of $\bm{a}$ because the logarithm function is locally
  integrable on $\R$. Summing up estimates~\eqref{eq:error3est},
  \eqref{eq:error2est}, and \eqref{eq:error1est}, using the assumption
  $\rho_T>0$, and the definition of $\bm{a}(t)$ and $\bm{b}(t)$ we find
  \begin{align*}
    |\dot{\bm{a}}(t) - \dot{\bm{b}}(t)| \lesssim \frac{|\bm{a}(t)-\bm{b}(t)|}{\rho_T^{\frac52}}  + \frac{\norm{\P_n^\perp g_{\bm{a}(t)}}_{L^2(\partial \Omega)}}{\rho_T^{\frac32}},
  \end{align*}
  from which we obtain~\eqref{eq:gronwalest} after applying Gronwal's inequality.
  For the estimate in the disk, one can simply
  apply~\eqref{eq:classicaldiskest} and \eqref{eq:logspectralconv} instead of
  \eqref{eq:classicalest1} to estimate the $E_2(\bm{a})$ term.

\end{proof}

\subsection{Time discretization}

The time integration error follows easily from the previous bounds and standard
considerations on the numerical integration of ODE, that we briefly recall
here.

\begin{lemma}\label{lem:disc_error}
  Let $\bm b$
  be the exact solution to $\dot{\bm b} = F^n(\bm b)$ and let $(\bm b_{\delta
    t}^k)_{k\delta t\leq T}$ be its numerical approximation with a RK4 method
  and time step $\delta t$:
  \[
    \begin{cases}
      \dot{\bm b} = F^n(\bm b),\\
      \bm b(0) = \bm b^0,
    \end{cases}\quad
    \begin{cases}
      \bm b_{\delta t}^k = \bm b_{\delta t}^{k-1} + \delta t \Phi_{\delta t,n}
      (\bm b_{\delta t}^{k-1}),\\
      \bm b_{\delta t}^0 = \bm b^0,
    \end{cases}
  \]
  where $\Phi_{\delta t, n}$ is the RK4 increment function associated with the
  approximate forcing term $F^n$ (see for instance
  \cite[Chapter II]{hairerSolvingOrdinaryDifferential1993}).
  Assume that
  \[
    0<\rho_T=\min\prt{\inf_{0\leq t\leq T}\rho(\bm b(t)),
      \min_{k\delta t\leq T}\rho(\bm b^k_{\delta t})}.
  \]
  Then, there exists constants $C_{\rho_T}>0$ and $\Lambda_{\rho_T}$, that depends on
  $\rho_T$ but not on $n$, $\bm b$ or $\delta t$, such that the following
  global in time error bound holds:
  \[
    \max_{k\delta t\leq T}|\bm b(k\delta t) - \bm b_{\delta t}^k | \leq
    \delta t^4 \frac{C_{\rho_T}}{\Lambda_{\rho_T}}\prt{e^{\Lambda_{\rho_T} T} - 1}.
  \]
\end{lemma}
\begin{proof}
  First, we recall, using notations from the proof of Lemma~\ref{lem:dyn_err}, that
  \[
    \Forall \bm a, \bm b\in\Omega,\quad
    |F^n(\bm a) - F^n(\bm b)| \leq E_1(\bm a,\bm b) + E_3(\bm a,\bm b) \lesssim
    \frac{|\bm a-\bm b|}
    {\min_j\set{\dist(a_j,\partial\Omega),\dist(b_j,\partial\Omega)}^{5/\textcolor{red}{2}}}
    + \frac{|\bm a-\bm b|}{\min(\rho(\bm a), \rho(\bm
      b))^{\textcolor{red}{2}}},
  \]
  with constants that do not depend on $n$. Thus, since
  \[
    0<\rho_T=\min\prt{\inf_{0\leq t\leq T}\rho(\bm b(t)),
      \min_{k\delta t\leq T}\rho(\bm b^k_{\delta t})},
  \]
  the function $F^n$ is Lipschitz continuous on the domain of interest with a Lipschitz
  constant $\Lambda_{\rho_T}$ independent of $n$. To conclude, note that
  $|\nabla_{\bm{b}}^4F^n(\bm b)| \leq C$ for some $C>0$ independent of $n$ and
  $\bm{b}$, as long as $\rho_T>0$ (this follows from the linearity of Laplace's
  equation and estimates~\eqref{eq:classicalest1} and~\eqref{eq:lognorms}). As a
  result, the local truncation error can be bounded as follows: there exists
  $C_{\rho_T}>0$, independent of $n$, such that
  \[
    |\bm b(\delta t) - \bm b^1_{\delta t}| \leq C_{\rho_T} {\delta t}^5.
  \]
  Thus, by combining the local truncation error and the Lipschitz continuity of
  the numerical integrator, we obtain the global in time error bound
  (see \cite[Theorem II.3.6]{hairerSolvingOrdinaryDifferential1993}).
\end{proof}

\end{document}